\documentclass[12pt,letterpaper,titlepage]{amsart}
\usepackage{amsmath,amssymb,amsthm,amsfonts,amscd}
\IfFileExists{amsaddr.sty}{\usepackage{amsaddr}}{}
\IfFileExists{wasysym.sty}{\usepackage{wasysym}}{}
\usepackage{enumerate,mathtools,hyperref}
\IfFileExists{lmodern.sty}{\usepackage{lmodern}}{}
\IfFileExists{comment.sty}{\usepackage{comment}}{\usepackage{verbatim}}
\usepackage[
    paper=a4paper,
    portrait=true,
    textwidth=425pt,
    textheight=650pt,
    tmargin=3cm,
    marginratio=1:1
        ]{geometry}
\usepackage[utf8]{inputenc}

\theoremstyle{plain}

\newtheorem{introtheorem}{Theorem}

\newtheorem{introprop}[introtheorem]{Proposition}
\newtheorem{introconj}[introtheorem]{Conjecture}
\newtheorem{introcor}[introtheorem]{Corollary}

\newtheorem{theorem}{Theorem}[subsection]

\newtheorem{cor}[theorem]{Corollary}

\newtheorem{prop}[theorem]{Proposition}
\newtheorem{lemma}[theorem]{Lemma}

\theoremstyle{definition}
\newtheorem{definition}[theorem]{Definition}

\newtheorem{rmk}[theorem]{Remark}
\numberwithin{equation}{section}
\newtheorem*{theoremA*}{Theorem A}
\newtheorem*{theoremB*}{Theorem B}
\newtheorem*{theorem1*}{Theorem A'}
\newtheorem*{theoremC*}{Theorem C}
\newtheorem*{theoremD*}{Theorem D}
\newtheorem*{theoremE*}{Theorem E}
\newtheorem*{theoremF*}{Theorem F}
\newtheorem*{theoremE2*}{Theorem E2}
\newtheorem*{theoremE3*}{Theorem E3}
\newcommand{\bs}{\backslash}
\newcommand{\Cc}{\mathcal{C}}

\newcommand{\C}{\mathbb{C}}

\newcommand{\A}{\mathcal{A}}
\newcommand{\Nc}{\mathcal{N}}

\newcommand{\Hc}{\mathcal{H}}

\newcommand{\Z}{\mathbb{Z}}

\newcommand{\R}{\mathbb{R}}
\newcommand{\N}{\mathbb{N}}

\newcommand{\Ind}{\operatorname{Ind}}

\newcommand{\SO}{\operatorname{SO}}
\newcommand{\Hom}{\operatorname{Hom}}

\newcommand{\GL}{\operatorname{GL}}
\newcommand{\SL}{\operatorname{SL}}

\newcommand{\SU}{\operatorname{SU}}
\newcommand{\tr}{\operatorname{tr}}

\newcommand{\Ad}{\operatorname{Ad}}

\newcommand{\ad}{\operatorname{ad}}
\newcommand{\diag}{\operatorname{diag}}

\newcommand{\vol}{\operatorname{vol}}

\newcommand{\Spec}{\operatorname{spec}}
\newcommand{\supp}{\operatorname{supp}}
\newcommand{\Span}{\operatorname{span}}

\newcommand{\rank}{\operatorname{rank}}

\newcommand{\re}{\operatorname{Re}}

\def\dotvar{\, \cdot\,}
\newcommand{\gap}{\operatorname{gap}}

\def\hat{\widehat}
\def\af{\mathfrak{a}}

\def\e{\epsilon}
\def\gf{\mathfrak{g}}

\def\hf{\mathfrak{h}}
\def\kf{\mathfrak{k}}

\def\mf{\mathfrak{m}}
\def\nf{\mathfrak{n}}

\def\pf{\mathfrak{p}}

\def\sf{\mathfrak{s}}
\def\sl{\mathfrak{sl}}
\def\gl{\mathfrak{gl}}

\def\so{\mathfrak{so}}
\def\sp{\mathfrak{sp}}
\def\su{\mathfrak{su}}
\def\tf{\mathfrak{t}}

\def\zf{\mathfrak{z}}
\def\la{\langle}
\def\ra{\rangle}
\def\1{{\bf1}}

\def\U{\mathcal{U}}
\def\B{\mathcal{B}}
\def\Cc{\mathcal{C}}

\def\oline{\overline}

\def\tilde{\widetilde}

\def\aut{\text{aut}}

\newcommand{\No}{\operatorname{Norm}}
\newcommand{\Noc}{\mathfrak{Norm}}

\newcommand{\id}{\operatorname{id}}

\usepackage{xcolor}

\hypersetup{hidelinks,pdftitle={Sobolev gap for Harish-Chandra modules},
 pdfauthor={Joseph Bernstein, Pritam Ganguly, Bernhard Kroetz, Job J. Kuit, Eitan Sayag}}

\newcounter{HCstep}

\hypersetup{pdftitle={Sobolev gap of Harish-Chandra modules}}
\title[Sobolev gap of Harish-Chandra modules]{Sobolev gap of Harish-Chandra modules}

\begin{document}
\date{\today}

\begin{abstract}
The Casselman--Wallach theorem, a foundational result in the
representation theory of real reductive groups, implies that any two
Banach globalization norms on a Harish-Chandra module can be compared
at the cost of finitely many derivatives. We introduce the
\emph{Sobolev gap}, a new invariant that makes this comparison
quantitative: it measures the worst loss of derivatives among
globalization norms in a growth class specified by a weight.

For a connected semisimple group with finite center, we compute
the gap of every minimal principal series with imaginary parameter.
It equals the homogeneous dimension of the unipotent radical of a
minimal parabolic subgroup. For every unitarizable Harish-Chandra
module of such a group, the Gelfand--Kirillov dimension gives a
lower bound for the gap. We also prove uniform upper bounds for the discrete series
at weight $\1$ and for the full minimal principal series at each fixed
weight. We conjecture such a uniform bound for all Harish-Chandra
modules admitting a globalization of the prescribed growth.

For $\SL(2,\R)$, the gap is one on the nontrivial irreducible unitary
dual. As an application, for
every lattice in $\SL(2,\R)$ and every nonzero cuspidal
realization of a nontrivial irreducible unitary representation,
the automorphic sup-norm is dominated by the unitary Sobolev norm
of order $s$ if and only if $s>1/2$.

On a compact quotient of a connected real reductive group, the
automorphic sup-norm of an infinite-dimensional irreducible
unitary representation is dominated by its unitary Sobolev norm
of order $s$ exactly when $s$ exceeds half its Gelfand--Kirillov
dimension.
Thus, whenever the Sobolev gap exceeds the Gelfand--Kirillov dimension
of \(V\), the automorphic sup-norm requires fewer derivatives than
the maximal globalization norm. 

On a finite-volume quotient, a nonzero regular Fourier coefficient of a bounded automorphic realization of an irreducible tempered minimal principal series forces its critical Sobolev order to equal half the gap. For the level-one spherical cuspidal spectrum of $\SL(n, \R)$ this order is $\frac{n(n^2-1)}{12}$.

\end{abstract}

\author[Bernstein]{Joseph Bernstein}
\email{bernstei@post.tau.ac.il}
\address{Tel Aviv University, School of Mathematical Sciences,\\Ramat Aviv,
Tel Aviv 69978}

\author[Ganguly]{Pritam Ganguly}
\email{pritam1995.pg@gmail.com}
\address{ Indian Statistical
Institute, Stat-Math Unit,\\
203 B. T. Rd., Kolkata 700108}

\author[Krötz]{Bernhard Krötz}
\email{bkroetz@gmx.de}
\address{Universit\"at Paderborn, Institut f\"ur Mathematik\\Warburger Stra\ss e 100,
33098 Paderborn}

\author[Kuit]{Job J. Kuit}
\email{j.j.kuit@gmail.com}
\address{Universit\"at Paderborn, Institut f\"ur Mathematik\\Warburger Stra\ss e 100,
33098 Paderborn}

\author[Sayag]{Eitan Sayag}
\email{sayage@math.bgu.ac.il}
\address{Ben-Gurion University of the Negev, Department of Mathematics\\PO Box 653,
Be'er Sheva 8410501}

\maketitle
\tableofcontents

\section{Introduction}

One of Harish-Chandra's basic insights was that a continuous representation of a
real reductive group can be studied through a more algebraic
object. 

Let $G$ be a real reductive group with Lie algebra $\gf$, and
let $K$ be a maximal compact subgroup. Harish-Chandra showed
that irreducible unitary representations can be studied through
their $K$-finite smooth vectors. A vector is $K$-finite if its
$K$-translates span a finite-dimensional space. These vectors form a
$(\gf,K)$-module: the action of $K$ is inherited from the
representation, and the action of $\gf$ is obtained by differentiation.

A \emph{Harish-Chandra module} is a $(\gf,K)$-module $V$ that is
finitely generated over the universal enveloping algebra $\U(\gf)$ and in which every irreducible
$K$-type occurs with finite multiplicity. More generally, the
smooth $K$-finite vectors of an admissible representation of
finite length form a Harish-Chandra module.
Two admissible representations
of finite length are called \emph{infinitesimally equivalent} if their
Harish-Chandra modules are isomorphic. The passage to a $(\gf,K)$-module retains the
infinitesimal action and the action of $K$, but leaves the topology
behind. Thus the next stage in studying topological representations of $G$ is to fix the algebraic
module $V$ and ask how the choice of a topology affects its
realization as a representation of $G$.

In going back from the module to a representation of $G$, we have
two things to do. We must complete the module in a suitable topology,
and we must extend the action to a continuous action of $G$ on that
completion. If the $K$-finite vectors of the resulting representation
recover the original module $V$, we call it a \emph{globalization}
of $V$.

\subsection{Globalization norms and the Sobolev gap}

A norm $p$ on a Harish-Chandra module $V$ is called $G$-continuous if
its Banach completion $V_p$ is a globalization of $V$. The choice of
$p$ can make a considerable difference to the Banach space we obtain.
What is remarkable is that this difference disappears on passing to
smooth vectors. The foundational Casselman--Wallach theorem gives,
up to isomorphism, a unique smooth Fr\'echet globalization of
moderate growth, which we denote by $V^\infty$
\cite{Casselman,Wallach_book_II}. In the norm formulation of the
theorem \cite{BK}, every Harish-Chandra module admits a
$G$-continuous norm, and any two such norms $p$ and $q$ are Sobolev
equivalent: there is an integer $k$ for which
\begin{equation} \label{eq:norm comparison} 
                   p\lesssim q_k,\qquad q\lesssim p_k.
\end{equation} 
Here $q_k$ is the Sobolev norm of order $k$ formed from $q$, and
$\lesssim$ denotes domination up to a constant.

How many derivatives are needed in (\ref{eq:norm comparison}) to pass from one Banach globalization
of $V$ to another? The Casselman--Wallach theorem tells us that
finitely many suffice for each pair of norms. We would like a bound
that works for a whole family of norms on $V$. To specify the family,
we fix a common bound on the growth of the $G$-action on the
corresponding Banach spaces.

For a $G$-continuous norm $p$, let $w_p(g)$ be the operator norm of
$g$ on $V_p$. A weight on $G$ is a positive, locally bounded, submultiplicative function on $G$. We write $\No(V,w)$ for the
$G$-continuous norms with $w_p\lesssim w$. Mutual domination defines
equivalence of norms, and $\Noc(V,w)$ denotes the ordered set of
equivalence classes. At the constant weight $\1$ we omit the weight
from the notation. The action is then uniformly bounded, and every
class has an isometric representative: for $[p]\in \Noc(V)$ simply replace $p$ by $q$ where
$q(v):=\sup_{g\in G}p(gv)$. Then $q$ is isometric and $[p]=[q]$.

Isometric norms arise naturally in harmonic analysis. Suppose that
$X=H\bs G$ carries an invariant measure and that an $H$-invariant continuous
functional $\eta$ on $V^\infty$ gives the realization
\[
        T_\eta(v)(Hg)=m_{v,\eta}(Hg)=\eta(gv).
\]
We say that $(V,\eta)$ is \emph{$(X,r)$-bounded} if the image lies in
$L^r(X)$. If the realization is also injective, then
\[
                  p_{\eta,X,r}(v)=\|m_{v,\eta}\|_{L^r(X)}
\]
is an isometric $G$-continuous norm. Injectivity is automatic when
$V$ is irreducible and $\eta\ne0$. Even on a single module this
construction gives many norms, as we vary the space, the functional,
or the exponent. We would like to compare them without having to
study each realization separately.

The observation that makes this possible is that, whenever the set
$\No(V,w)$ is nonempty, its ordered set of norm classes has a
smallest and a largest element. We denote representatives by
$p_{\min}$ and $p_{\max}$. At weight $\1$, the minimal norm can be described
directly in terms of matrix coefficients. If $\ell_1,\ldots,\ell_m$ generate
the contragredient module $\tilde V$, we may take
\[
                 p_{\min}(v)=\max_{1 \leq j \leq m}\sup_{g\in G}|\ell_j(g\cdot v)|.
\]
Every isometric $G$-continuous norm dominates this norm, since the
$\ell_j$ are continuous in the corresponding Banach globalization.
Dualizing the minimal norm on $\tilde V$ gives the maximal norm on
$V$. The comparison of all the norms in a growth class thus reduces
to the comparison of these two extreme norms.

Let $p$ be a $G$-continuous norm on $V$. For real $s$, we define a 

Sobolev norm $p_s$ on $V$ as follows: for $v \in V$ we put $$p_s(v)=p((1+\Delta_K)^{s/2}v)$$ where $\Delta_K$
is the nonnegative $K$-Laplacian. Note that $p_s$ is not necessarily $G$-continuous (see Remark \ref{rmk:projective-banach}).

We define the \emph{Sobolev gap} by
\[
 \gap(V,w)=\inf\{s\geq0:p_{\max}\lesssim p_{\min,s}\},
 \qquad
 \gap(V)=\gap(V,\1).
\]
The gap is a new invariant of the module in the given growth class:
it measures the worst loss of derivatives in comparing its Banach
globalization norms. Its finiteness follows from Casselman--Wallach.
Determining its value, or bounding it uniformly as the module varies,
makes the norm comparison \eqref{eq:norm comparison} in the Casselman--Wallach theorem quantitative.
When $V$ is unitarizable and the weight is $\1$, the unitary norm
$q$ divides the Sobolev gap equally. More precisely, the infimum of
Sobolev orders needed to dominate $q$ by a shift of $p_{\min}$ is
the same as the order needed to dominate $p_{\max}$ by a shift
of $q$. Both are equal to $\gap(V)/2$
(Proposition~\ref{prop unitary}). In particular, every isometric
globalization norm $p$ on $V$ satisfies
\begin{equation}\label{midpoint}
             p\lesssim q_s \qquad \bigl(s>\tfrac12\gap(V)\bigr).
\end{equation}
\subsection[The gap of the minimal principal series]
{{The gap of the minimal principal series}}

Let $P=MAN$ be a minimal parabolic subgroup, and let
$V_{\sigma,\lambda}$ be the normalized principal series induced
from $\sigma\in\widehat M$ and $\lambda\in\af_\C^*$. In the sequel we refer to $V_{\sigma,\lambda}$ as a minimal principal series. Our first result
identifies the Sobolev gap of this family at imaginary parameter
with a dimension attached to $N$.

The relevant dimension can already be seen for $G=\SL(3,\R)$.
Conjugation by $\diag(r,1,r^{-1})$, with $0<r<1$, multiplies the
upper triangular coordinates $n_{12},n_{23},n_{13}$ by $r,r,r^2$,
respectively. Thus Haar volume scales by $r^4$, although $N$ has
dimension three. In general, a restricted root space of height $j$
scales by $r^j$ under the dilation for which every simple root has
weight one. Writing $m_\alpha=\dim\gf^\alpha$, we obtain the
homogeneous dimension
\[
   2c_\gf=\sum_{\alpha>0}m_\alpha\operatorname{ht}(\alpha).
\]

\begin{introtheorem}\label{introthm:mps}
Let $G$ be connected semisimple with finite center.
\begin{enumerate}[(i)]
\item\label{intro:mps-exact}
For every $\sigma\in\widehat M$ and $\lambda\in i\af^*$,
\[
                     \gap(V_{\sigma,\lambda})=2c_\gf .
\]
\item\label{intro:mps-upper}
For every $\sigma\in\widehat M$ and $\lambda\in\af_\C^*$ for which
$V_{\sigma,\lambda}$ admits a globalization norm of growth $\1$,
\[
                     \gap(V_{\sigma,\lambda})\leq 2c_\gf .
\]
\end{enumerate}
\end{introtheorem}

In particular, 
every tempered minimal principal series of
$\SL(3,\R)$ has gap 4.
The upper bound comes from contracting a
test section on the open flag cell, see  \cite[Sect.~12]{BK}. To prove sharpness, we construct
oscillatory sections whose pairings are small against every translate
of the test vectors. The choice of a single phase that works for all
translates is the main point of
Theorem~\ref{thm:flag-gap}. The upper bound at general parameter
is Corollary~\ref{cor:full-induced-gap}. It applies in particular to
unitarizable full inductions, including the complementary series.

For imaginary $\lambda$, let $q$ be the unitary norm and let $p$ be
any isometric globalization norm. The midpoint formula
\eqref{midpoint} gives $p\lesssim q_{c_\gf+\epsilon}$ for every
$\epsilon>0$. On an individual $K$-type the exponent $c_\gf$ itself
is valid. By Corollary~\ref{Cor mps convexity},
\[
   p(v_\tau)\lesssim (1+|\tau|)^{c_\gf}q(v_\tau)
   \qquad(\tau\in\widehat K,\ v_\tau\in V[\tau]).
\]

\subsection[A dimension lower bound and uniform upper bounds]
{A dimension lower bound and uniform upper bounds}

What can one say about the gap for an arbitrary unitarizable
module $V$? Its algebraic growth already gives a lower bound. If $F$
is a finite-dimensional generating subspace of $V$, the span of the
vectors obtained from $F$ by
applying at most $n$ Lie algebra elements has dimension of order
$n^d$, where $d=\operatorname{GKdim}V$. We recall the definition
in Subsection~\ref{subsec:unitarizable-gk-gap}.

\begin{introtheorem}[See Theorem~\ref{thm:unitarizable-gk-gap}]
\label{introthm:gk-gap}
Let $G$ be connected semisimple with finite center. For every
unitarizable Harish-Chandra module $V$ of $G$, at weight $\1$ one has
\[
                       \operatorname{GKdim}V\leq\gap(V).
\]
\end{introtheorem}

For the minimal principal series with imaginary parameter,
the Gelfand--Kirillov dimension is the ordinary dimension of the
open cell, $\dim N=\sum_{\alpha>0}m_\alpha$. The gap thus counts
each root direction with its height, whereas the Gelfand--Kirillov
dimension counts it once. For $G=\SL(n,\R)$ this gives
\[
   \operatorname{GKdim}V_{\sigma,\lambda}=\frac{n(n-1)}2,
   \qquad
   \gap(V_{\sigma,\lambda})=\frac{n(n^2-1)}6 .
\]
Their ratio is $(n+1)/3$, so the discrepancy grows with the rank.
For a noncompact simple Lie algebra, equality for this family holds
precisely when every positive restricted root has height one, that
is, when $\gf\simeq\so(1,n)$ for some $n\geq2$.

\subsubsection{Uniform upper bounds}

We next ask for an upper bound on the gap that is uniform as the representation
varies. We obtain such bounds for two families.

Write $\mathcal{HC}$ for the class of Harish-Chandra modules and
$\mathcal{HC}_{\rm d}$, $\mathcal{HC}_{\rm mp}$ and
$\mathcal{HC}_{\rm u}$ for the discrete series, minimal
principal series and unitarizable classes, respectively.
The following bounds follow from Theorem \ref{thm disc series} and equation
\eqref{thm sup reloaded} in Remark \ref{rmk mps weight}.

\begin{introtheorem}\label{thm.main1.101}
Let $G$ be a real reductive group. 
\begin{enumerate}[(i)]
\item The Sobolev gap for discrete series is uniformly bounded, i.e.
$$
\sup_{V\in \mathcal{HC}_{\rm d}} \gap(V)< \infty.
$$
\item \label{thm.main1 - item 2}
For every weight $w$ the $w$-Sobolev gap for all minimal principal series 
admitting $G$-continuous norms with growth class at most $w$ 
is uniformly bounded, i.e.
$$
\sup_{\substack{V\in \mathcal{HC}_{\rm mp}\\ \No(V,w)\neq \varnothing}}\gap(V,w)
< \infty.
$$
\end{enumerate}
\end{introtheorem}
This leads us to the following Conjecture.

\begin{introconj}[Uniform Finiteness Conjecture] \label{conj uniform finiteness}

For any given weight $w$ on $G$ the $w$-Sobolev gaps $\gap(V,w)$ are uniformly bounded, i.e., there exists a constant $C_w>0$ such that 

\begin{equation} \label{Dgap} \gap(V,w)\leq C_w
\end{equation} 
for all Harish-Chandra modules $V$ with $\No(V,w)\neq \varnothing$. In particular, for the full unitary dual 
\begin{equation}\label{sgapu}
\sup_{V \in \mathcal{HC}_{\rm u}} \gap(V)
<\infty.
\end{equation}
\end{introconj}

In Section \ref{Section Conj D}
we verify \eqref{Dgap} for all weights and for all irreducible Harish-Chandra modules for $\SL(2,\R)$.

\subsection{\texorpdfstring{The unitary dual of $\SL(2,\R)$}{The unitary dual of SL(2,R)}} 

For $\mathrm{SL}(2,\mathbb R)$ we can compute the gap on the entire irreducible unitary dual.

\begin{introtheorem}[See Theorem~\ref{thm ultimate}]\label{intothm.Thm C}
For every nontrivial irreducible unitarizable Harish-Chandra module
$V$ of $\SL(2,\R)$,
\[
                            \gap(V)=1 .
\]
\end{introtheorem}
For the tempered principal series this is a special case of Theorem  \ref{introthm:mps}. Note that Theorem~\ref{intothm.Thm C} shows that the same value also occurs
for the complementary series, the discrete series, and their limits.

\begin{rmk}
For a general real reductive group, Theorem~\ref{thm.main1.101}(i)
gives a uniform upper bound for the Sobolev gap of the discrete
series. This leaves open a more precise question: for a fixed group,
how does the gap depend on the discrete series representation?
For an intermediate discrete series of $\SU(1,n)$, $n\geq2$,
the associated-variety computation in the proof of
Proposition~\ref{prop BR}\textup{(\ref{prop BR - item 5c})}
gives $\operatorname{GKdim}V=2n-1$. Thus
Theorem~\ref{introthm:gk-gap} gives $\gap(V)\geq 2n-1$.
The exact abstract gap of these modules remains to be determined. 
\end{rmk}

To describe the consequence of Theorem \ref{intothm.Thm C} for individual norms, 
it is useful to introduce a pseudometric structure on the set $\Noc(V)$ whenever $V$ is a Harish-Chandra module of a general real reductive group $G$.

Put
\[
 d_\to([p],[q])=\inf\{s\geq0:p\lesssim q_s\},\qquad
 d([p],[q])=\max\{d_\to([p],[q]),d_\to([q],[p])\}.
\]

Note, that the diameter of the pseudo-metric space $(\Noc(V),d)$ is $\gap(V)$. 

We now return to the basic case of $G=\SL(2,\R)$.
For a nontrivial irreducible unitarizable $V$ of $\SL(2,\R)$,
Theorem~\ref{intothm.Thm C} and the midpoint formula give
\[
       d([p_{\min}],[q])=d([q],[p_{\max}])=\tfrac12 .
\]
Thus every isometric globalization norm is dominated by $q_s$
for $s>1/2$. On an individual $K$-type we obtain the exponent
$1/2$ itself, by Proposition~\ref{thm upper-lower}.

\begin{introcor}[Abstract Convexity Bound]\label{introthm.abst. conv. bnds}
Let $V$ be an irreducible unitarizable Harish-Chandra module for $\SL(2,\R).$ Let $q$ be a unitary norm on $V$. Let $S=\Spec_K(V)\subset \hat K=\Z$ and let $(e_n)_{n\in S}$ be 
a corresponding orthonormal basis consisting of $K$-isotypical vectors.
Let $p$ be any isometric  $G$-continuous norm on $V$ and $\epsilon>0$. Then there exists a constant $C_\e=C_\e(V,p,q)>0$, such that $p\leq C_\e\, q_{\frac{1}{2}+\epsilon}$. Moreover, there exists a constant $C=C(V,p,q)>0$ such that  
$$
p(e_n)
\leq C(1+|n|)^{\frac{1}{2}}\qquad (n \in S).
$$
\end{introcor}

Interestingly, this conclusion can also be expressed by stabilization of the projection to $G$-invariant norms of the Sobolev enhancement of a norm. More precisely, let $q_s^G$
be the largest $G$-invariant seminorm dominated by $q_s$ (the
infimum construction of Subsection~\ref{Subsection infimum}). The result is that $[q_s^G]=[p_{\max}]$ for $s>1/2$. 

For the unitary principal series this stabilization was discussed in  \cite[Appendix~A.2]{BR2}; here it holds on the entire nontrivial irreducible unitary dual of $\SL(2,\R)$.

\subsection{Applications to automorphic forms}\label{subsec. subconvexity}
How large can an $L^2$-normalized automorphic form be at a single
point? On a compact arithmetic hyperbolic surface, the classical
sup-norm problem asks how the maximum of a Hecke--Maass eigenform
grows with its Laplace eigenvalue. Iwaniec and Sarnak~\cite{IS}
showed that the arithmetic of the surface plays a role in this
question: by exploiting the Hecke operators, they obtained a power
saving over the general bound for Laplace eigenfunctions on a
compact surface.

Let $\Gamma<G$ be a lattice and put $X=\Gamma\bs G$.
For $G=\SL(2,\R)$ and $K=\SO(2)$, a Maass cusp form on $X/K$
lifts to a $K$-fixed vector in an automorphic representation on $X$.
In the spectral sup-norm problem the representation varies, as in
\cite{IS,BHMM}. Here we fix the representation and its realization
on $X$. How large can the sup-norm of a unit vector be as its
$K$-type varies? More generally, which Sobolev norms dominate the
sup-norm on all smooth vectors in this realization? These questions
are closely related to the study of automorphic functionals and
distributions in \cite{BR,BR2,Sch2}.

For the role of Sobolev norms on representation spaces,
see \cite{BernsteinICM}. In the analytic theory of automorphic representations they enter in estimates
for forms and periods \cite{KS,MV,Nel,NV,Venkatesh}, effective
equidistribution \cite{EMV}, and the study of families of
$L$-functions \cite{Mic}. For related estimates with weighted norms,
see \cite{He}.

For general $G$, let $V$ be an irreducible unitarizable
Harish-Chandra module and let $\eta:V^\infty\to\C$ be a nonzero
continuous $\Gamma$-invariant functional. The corresponding
automorphic realization sends $v$ to the function
\[
   m_{v,\eta}(\Gamma g)=\eta(g\cdot v).
\]
Assume that $m_{v,\eta}$ is bounded on $X$ for every smooth $v$.
On a compact quotient this boundedness is automatic. For smooth cuspidal automorphic
forms in the arithmetic setting it follows from rapid decay \cite{MillerSchmidDecay}. For spherical Maaß
cusp forms, exponential estimates are available under the hypotheses of \cite{KroetzOpdamCrown}. For cofinite
lattices in $\SL(2, \R)$ see \cite[Proposition 4.1 and Appendix]{BR2}.

Since $X$ has finite volume, these functions are square-integrable.
We use the norms
\[
   q(v)=\|m_{v,\eta}\|_{L^2(X)},\qquad
   p_{\aut}(v)=\|m_{v,\eta}\|_{L^\infty(X)}.
\]
Right translation preserves the sup-norm, so $p_{\aut}$ is an
isometric globalization norm. Hence $p_{\aut}\lesssim p_{\max}$,
and the midpoint formula (Proposition~\ref{prop unitary}) gives
\[
   p_{\aut}\lesssim_s q_s
   \qquad\bigl(s>\tfrac12\gap(V)\bigr).
\]

For $G=\SL(2,\R)$, Theorem~\ref{intothm.Thm C} gives $\gap(V)=1$
for every nontrivial irreducible unitarizable $V$. Thus the
automorphic sup-norm is dominated by $q_s$ for every $s>1/2$.
For a cuspidal realization this range is exact: 
the domination fails at $s=1/2$
(Section~\ref{Sec: AF}). On an individual $K$-type, however,
Corollary~\ref{introthm.abst. conv. bnds} gives the exponent $1/2$
itself. The same corollary applies to the norm of any $L^r$-realization,
with $q$ a unitary norm on the abstract representation.

The resulting bounds apply to arbitrary lattices and to all
irreducible unitary types, including the complementary and discrete
series. For the cuspidal principal series case, compare
\cite[Appendix~A.2]{BR2}.
We summarize the discussion.

\begin{introtheorem}\label{introthm: theorem automorphic}
Let $\Gamma$ be a lattice in $G=\SL(2,\R)$ and put
$X=\Gamma\bs G$. Let $V$ be an irreducible unitarizable
Harish-Chandra module with unitary norm $q$, and let
$\eta:V^\infty\to\C$ be a nonzero continuous $\Gamma$-invariant
functional. For each $n\in\Spec_K(V)$, choose a $q$-unit
vector $e_n$ of weight $n$.
\begin{enumerate}[(i)]
\item Suppose that $\eta$ is cuspidal and $V$ is nontrivial.
Then the automorphic sup-norm satisfies
\[
   p_{\rm aut}\lesssim q_s
   \quad\Longleftrightarrow\quad s>\tfrac12 .
\]
On individual $K$-types one has
\[
   |\eta(e_n)|\leq p_{\rm aut}(e_n)
   \lesssim (1+|n|)^{1/2}
   \qquad(n\in\Spec_K(V)).
\]

\item Suppose that $\eta$ gives a realization of $V^\infty$
in $L^r(X)$ for some $1\leq r\leq\infty$. Then for every
$\epsilon>0$,
\[
   \|m_{v,\eta}\|_{L^r(X)}
   \lesssim_{r,\epsilon} q_{\frac12+\epsilon}(v)
   \qquad(v\in V^\infty).
\]
\end{enumerate}
\end{introtheorem}

Part~\textup{(ii)} applies, in particular, to unitary Eisenstein
series in the integrability range $1\leq r<2$. 

Returning to the general case of arbitrary real rank, Theorem~\ref{thm Lorentz} gives
a bound on each $K$-type for every $X$-bounded realization of a
minimal principal series with imaginary parameter. For fixed
$\Gamma$ and $\sigma$, Proposition~\ref{prop mps aut} gives a bound for individual $K$-types as in the first item of Theorem \ref{introthm: theorem automorphic}. In fact this is a uniform bound that is valid for $\lambda$ in a compact set and for all the
$L^2$-isometric realizations considered there.


\subsubsection{The automorphic Sobolev threshold}

We now return to a fixed nonzero $X$-bounded realization
$T_\eta$ on a finite-volume quotient $X=\Gamma\bs G$.
How many derivatives suffice to control the sup-norm on all smooth
vectors? With $q(v)=\|T_\eta v\|_2$ and
$p_{\aut}(v)=\|T_\eta v\|_\infty$, we define the
\emph{automorphic Sobolev threshold} by
\[
   s_{\aut}(V,\eta)
   :=\inf\{s\geq0:p_{\aut}\lesssim q_s\}
   =d([q],[p_{\aut}]).
\]
The last equality follows from
$q\leq\vol(X)^{1/2}p_{\aut}$.

\begin{introprop}
\label{introthm:aut-threshold}
Let $\Gamma$ be a lattice in a connected real reductive group $G$,
and let $V$ be an infinite-dimensional irreducible unitarizable
Harish-Chandra module. Suppose that a nonzero automorphic realization $T_\eta$ sends every smooth
vector to a bounded function on $X=\Gamma\bs G$. Then
\[
   \tfrac12\operatorname{GKdim}V
   \ \leq\ s_{\rm  aut}(V,\eta)
   \ \leq\ \tfrac12\gap(V).
\]
On every such quotient, domination $p_{\rm aut}\lesssim q_s$ requires $s > \frac12 \operatorname{GKdim} V$. If $X$ is compact, this condition is also sufficient. It is likewise sufficient on any such quotient when
$\gap(V ) = \operatorname{GKdim} V$.
\end{introprop}

Proposition~\ref{introthm:aut-threshold} is proved in
Proposition~\ref{prop:aut-finite-volume-bounds}. The lower bound
comes from the relative-trace argument of Bernstein and Reznikov
\cite{BR} and Vogan's $K$-type growth theorem
\cite[Proposition~4.6]{VoganSize103}. The upper bound comes from
the midpoint formula, Proposition~\ref{prop unitary}.

For an irreducible tempered minimal principal series of a
connected semisimple group with finite center, the two ends of the
interval are half the ordinary dimension and half the homogeneous
dimension of $N$:
\[
   \tfrac12\operatorname{GKdim}V=\tfrac12\dim N,
   \qquad
   \tfrac12\gap(V)=c_{\gf}.
\]
On a compact quotient the automorphic threshold is the lower end.
The maximal norm, on the other hand, is at distance $c_{\gf}$ from
$q$. The triangle inequality therefore gives
\[
   d([p_{\max}],[p_{\aut}])
   \geq c_{\gf}-\tfrac12\dim N
\]
(Corollary~\ref{cor:aut-max-distance}). Thus the two norms are at
positive Sobolev distance whenever the homogeneous dimension of $N$
exceeds its ordinary dimension.

On a noncompact quotient the automorphic threshold can
instead be the upper end. Under the hypotheses of
Proposition~\ref{prop:regular-fourier-threshold}, a nonzero regular
Fourier coefficient forces $s_{\aut}(V,\eta)=c_{\gf}$. This means
that no order $s<c_{\gf}$ suffices and every order $s>c_{\gf}$
does. 

To illustrate these results we consider 
$G=\SL(n,\R)$ and a co-compact $\Gamma < G$. The threshold in this case is $n(n-1)/4$. In the classical, non co-compact set up of $\SL(n,\Z)\bs\SL(n,\R)$, every nonzero cuspidal realization of
an irreducible spherical minimal principal series with imaginary
parameter has threshold $n(n^2-1)/12$
(Corollary~\ref{cor:sln-cuspidal-threshold}). Such realizations
with regular imaginary parameter occur by the joint Weyl law of
Finis and Lapid \cite[Theorem~5.11]{FinisLapidWeyl}. Returning to
the $\SL(3,\R)$ example, the ordinary dimension three gives the
threshold $3/2$ on a compact quotient, while the homogeneous
dimension four gives the threshold $2$ for these level-one
cuspidal realizations.

\begin{rmk}
Equal thresholds need not mean zero distance. For a
spherical unitary principal series of $\SL(2,\R)$ on a compact
quotient, $p_{\aut}$ and $p_{\max}$ are both at distance $1/2$
from $q$, but their distance from one another is at least $1/6$
(Proposition~\ref{prop upper gap}).
\end{rmk}

\subsection{\texorpdfstring{Geometry and topology of $\Noc(V,w)$}{Geometry and topology of Noc(V,w)}}The article concludes with first investigations into the rich geometry of the pseudometric space $(\Noc(V,w),d)$. The first observation is that distinct norm classes can have Sobolev distance zero. 
After identifying such classes, we obtain a complete metric
space $\mathbf{Norm}(V,w)$, see Proposition~\ref{thm:sobolev-norm-geometry}.
Theorem~\ref{thm:sobolev-norm-contraction} contracts this space
to the maximal norm class, and
Theorem~\ref{thm:sobolev-norm-geodesic} gives an explicit
constant-speed geodesic between any two points.

\subsection[Structure of the paper]{Structure of the paper}

In Section~\ref{Notions} we study the globalization norms on
a Harish-Chandra module with a prescribed growth weight. We show that
whenever this class is nonempty, it has a minimal and a maximal element
up to equivalence (Propositions~\ref{prop minimal norm}
and~\ref{prop max norm}).

In Section~\ref{Section Sobolev norms} we compare the Sobolev norms
obtained from Lie algebra derivatives with those defined by the
Laplacian of a maximal compact subgroup $K$. The latter allow us to
compare globalization norms at any real Sobolev order.

In Section~\ref{section Sobolev gap} we use these comparisons to define
a pseudometric on the space of globalization norms. Its diameter is
the Sobolev gap. For a unitarizable module at weight $\1$, the unitary
norm lies at half this distance from each of the two extremal norms
(Proposition~\ref{prop unitary}).

Section~\ref{Section 5: Uniform finiteness} contains the results in
arbitrary real rank. We prove the Gelfand--Kirillov lower bound of
Theorem~\ref{introthm:gk-gap} and the uniform upper bounds for the
discrete series and the minimal principal series in
Theorem~\ref{thm.main1.101}. For the minimal principal series
with imaginary parameter, we determine the gap exactly: it is the
homogeneous dimension of the unipotent radical of a minimal parabolic
subgroup (Theorem~\ref{introthm:mps}).

In Section~\ref{Sobolev SL2} we determine the gap on the entire
nontrivial irreducible unitary dual of $\SL(2,\R)$. It is equal to
one in every case (Theorem~\ref{intothm.Thm C}). We also prove the
abstract convexity bound of Corollary~\ref{introthm.abst. conv. bnds},
including the exponent $1/2$ on individual $K$-types.

In Section~\ref{Sec: SR} we turn to $H$-invariant functionals on a
unitarizable module. When the corresponding generalized matrix
coefficients are bounded on $H\bs G$, Proposition~\ref{Prop X-bounded}
gives a Sobolev regularity bound for the functional and a bound for
its restriction to each $K$-type.

In Section~\ref{Sec: AF} we study the sup-norm of vectors in a
fixed automorphic realization of an irreducible unitary representation of $G$ on $X=\Gamma\bs G$. We begin with compact
quotients and then treat arbitrary lattices in $\SL(2,\R)$. For every
nonzero cuspidal realization of a nontrivial irreducible unitary
representation of $\SL(2,\R)$, we prove that $p_{\aut}\lesssim q_s$
holds exactly for $s>1/2$ (Theorem~\ref{introthm: theorem automorphic}).
In arbitrary real rank we bound the sup-norm on individual $K$-types
for each $X$-bounded realization of a minimal principal series with
imaginary parameter (Theorem~\ref{thm Lorentz}). For all smooth vectors,
Proposition~\ref{prop:aut-finite-volume-bounds} places the automorphic
Sobolev threshold $s_{\aut}(V,\eta)$ between $\tfrac12\operatorname{GKdim}V$ and
$\tfrac12\gap(V)$ on finite-volume quotients. On compact quotients the
threshold equals the lower bound. On noncompact quotients a nonzero
regular Fourier coefficient can force equality with the upper bound,
as in the level-one spherical cuspidal case for $\SL(n,\R)$
(Proposition~\ref{prop:regular-fourier-threshold} and
Corollary~\ref{cor:sln-cuspidal-threshold}).

Section~\ref{Appendix B} proves completeness and contractibility of the metric quotient and constructs its constant-speed geodesics.

Appendix~\ref{Sec9: key estimates} contains the matrix coefficient
estimates for $\SL(2,\R)$ used in Section~\ref{Sobolev SL2}.
These include estimates on individual $K$-types and an averaging
estimate for the discrete series and their limits.

\subsection{Acknowledgments}

We thank Andre Reznikov for his interest in this work, for helpful correspondence, and for a particularly enlightening conversation that led to the convexity bounds for automorphic functionals with respect to any lattice. In addition, we thank Valentin Blomer for answering our questions regarding \cite{BHMM}.

During the second revision of the arXiv-version of the article, the authors used OpenAI's
ChatGPT in working through the proofs and exposition. These discussions
helped identify the midpoint equality in Proposition~\ref{prop unitary}
and develop the Gelfand--Kirillov lower bound and the dyadic and
oscillatory arguments for the minimal principal series. They also
led to corrections in the treatment of restricted duality and the
extremal unit balls. The authors take responsibility for the final
statements and proofs.

During the preparation of this work, E.S. was partially supported by ISF grant No. 1781/23.

\section{Norms on Harish-Chandra modules}\label{Notions}
This section is elementary and has a variety of objectives. We introduce standard notation with respect to real reductive groups and their Harish-Chandra modules. Second, we recall and restructure a variety of basic facts concerning Banach completions of Harish-Chandra modules. With this in hand, we introduce a central concept of this paper, namely minimal and maximal $G$-continuous norms in Subsection \ref{subsect: min-max norms}.

\subsection{Harish-Chandra modules}
We recall some notions and background from \cite{BK}.
Let $G$ be a real reductive group of inner type in the sense of Wallach; see \cite[\S 2.1.1 \& \S 2.2.8]{Wallach1} and $K$ a maximal compact subgroup of $G$. The Lie algebra of $G$ is denoted by $\gf$. We recall that a $(\gf, K)$-module is called a Harish-Chandra module provided it is 
finitely generated as a $\U(\gf)$-module and $K$-admissible, i.e., $\Hom_K (\tau, V)$ is finite dimensional for all irreducible representations $\tau$ of $K$. The following standard criterion is useful:

\begin{lemma}\label{Lemma HC criterion} A $K$-admissible $(\gf, K)$-module is a Harish-Chandra module if and only if it is ${\mathcal Z}(\gf)$-finite, where ${\mathcal Z}(\gf)$ is the center of $\U(\gf)$.\end{lemma}

The category of Harish-Chandra modules is denoted by $\mathcal{HC}$.

\subsection{\texorpdfstring{$G$-continuous norms}{G-continuous norms}}

Let $V$ be a Harish-Chandra module and $p$ a norm on $V$. We write $V_p$ for the completion of the normed space $(V,p)$. 
The norm $p$ is called $G$-continuous provided there exists a continuous representation $\pi:G\times V_{p}\to V_{p}$ of $G$ on the Banach space $V_p$ such that the embedding $V \hookrightarrow V_p$ is $K$-equivariant and induces an isomorphism  $V \to V_p^{K-\rm{finite}}$ of $(\gf, K)$-modules.

\begin{prop}\label{Prop Existence of G-continuous norms}
Every Harish-Chandra module admits a $G$-continuous norm.
\end{prop}

\begin{proof}
By Casselman's subrepresentation theorem every Harish-Chandra module $V$ can be embedded into the Harish-Chandra module of an induced representation $I=(\Ind_P^G\sigma)^{K-{\rm finite}}$ where $P\subset G$ is a minimal parabolic subgroup and $\sigma$ is a finite-dimensional representation of $P$.
The module $I$ admits many $G$-continuous norms, for example $L^p$-norms of $\sigma$-valued functions on $K/K\cap P$ for $1\leq p < \infty$. Therefore, $V$ admits $G$-continuous norms as well.
See \cite[Prop. 4.2.3 and Cor. 4.2.4]{Wallach1}.
\end{proof}

For two norms $p$ and $q$ we write $p\lesssim q$ if there exists a constant $C>0$ such that $p\leq C q$. We say that two norms $p$ and $q$ are equivalent, in symbols $p\sim q$, provided that $p\lesssim q $ and $q\lesssim p $.

\begin{rmk}\,\begin{enumerate}[(a)]
\item We recall that a complete locally convex topological vector space $E$ is called a {\it globalization} or {\it completion} of $V$ provided that $E$ supports a $G$-representation such that $E^{K-{\rm finite}} \simeq_{(\gf, K)} V$.

The fact that one has an action of the test function algebra $C_c^\infty(G)$ on any globalization $E$ led Schmid \cite{Sch} to the concept of minimal globalization $V_{\min}$ of $V$. It can be defined as a quotient of finitely many copies of $C_c^\infty(G)$ and naturally has the property 
that it embeds into any globalization $E$ of $V$. We emphasize that $V_{\min}$ should not be confused with the minimal Banach globalizations, $V_{\min, w}$ (cf. \ref{rmk minimal and maximal}) within a certain growth class controlled by a weight $w$ (cf. Remark \ref{Rem Weights}) which are introduced in this paper. For a Borel weight $w$, the convolution algebra $L^1(G,w)$, with norm $\int_G |f(g)|w(g)\,dg$, acts on every Banach representation of operator growth $O(w)$. If $w$ is bounded below by a positive constant, then $L^1(G,w)$ embeds continuously in $L^1(G)$.

We have seen in the proof of Proposition \ref{Prop Existence of G-continuous norms} that Harish-Chandra modules $V$ are $(\gf,K)$-submodules of the spaces of $K$-finite sections of a finite dimensional vector bundle over the compact space $X=G/P$. From this viewpoint the smallest globalization which comes to mind 
is the space of analytic sections $V^\omega$.  It is a remarkable general theorem of Schmid \cite{Sch} that $V_{\min}=V^\omega$ - see also \cite{GKKS} for a quantitative version in the context of a Paley-Wiener theory for Harish-Chandra modules.

The minimal globalization $V^\omega$ is an example of a globalization which is an inductive limit of Banach spaces with nuclear transition maps.
\item Infinite dimensional Harish-Chandra modules admit many interesting norms, many of which are not $G$-continuous, for instance those
appearing in the 
inductive limit $V^\omega$ of Banach spaces.
\end{enumerate}
\end{rmk}

For a Harish-Chandra module $V$ we denote by $\tilde V$ its contragredient, i.e., the space of $K$-finite vectors in the algebraic dual $V^*$ of $V$. Lemma \ref{Lemma HC criterion} implies that $\tilde V$ is a Harish-Chandra module as well. Furthermore, double duality becomes immediate:  
$$
\tilde{\tilde V}
=V.
$$

Notice that if $p$ is a $G$-continuous norm on $V$ and $\ell:V \to \mathbb{C}$ a $K$-finite linear functional, then $\ell$ is bounded with respect to $p$. 

Given a $G$-continuous norm $p$ on $V$ we recall that the dual norm $\tilde p$, which is defined by
$$
\tilde p (\tilde v)
:=\sup_{p(v)\leq 1} |\tilde v (v)|
\qquad (\tilde v \in \tilde V),
$$
is $G$-continuous as well, see \cite[Sect. 3.1]{BK}. We summarize some basic properties.

\begin{lemma}\label{pq-lemma}
Let $p, q$ be $G$-continuous norms on a Harish-Chandra module $V$. Then
\begin{enumerate}[(i)]
\item\label{pq-lemma - item 1} $\tilde{\tilde p}\sim p$, with equality if $p$ is $K$-invariant.
\item\label{pq-lemma - item 2} $p\lesssim q\iff \tilde q\lesssim \tilde p$. If both
representatives are $K$-invariant, this equivalence also holds with $\leq$.
\end{enumerate}
\end{lemma}
\begin{proof}
For a $K$-invariant norm, fix a Hahn--Banach functional norming $v$ and convolve it on $K$ with a
positive $K$-finite approximate identity of mass one. The resulting functionals belong to $\tilde V$,
have norm at most one, and converge at $v$ to its norm. This proves exact double duality $\tilde{\tilde p}=p$. An arbitrary $p$ is equivalent to $\sup_{k\in
K}p(k\,\cdot)$. Duality reverses norm inequalities, and a second application gives the converse.
\end{proof}
\begin{rmk} If $p$ is not $K$-invariant, then we cannot claim equality $\tilde{\tilde p}=p$. To see that take on the projective circle in the spherical $\SL(2,\R)$ principal series $p(f)=\|f\|_\infty+|f(x_0)|$. This norm is equivalent to the sup-norm on the compact picture, hence is G-continuous, although it is not K-invariant.
A $K$-finite dual functional is given by integration against a trigonometric density $h$. Its norm is $\|h\|_1$:
approximate its norming phase in $L^1(|h|\,dx)$ by continuous functions bounded by one and vanishing near
$x_0$, then approximate uniformly by trigonometric polynomials. Consequently $\tilde{\tilde
p}=\|\cdot\|_\infty$, whereas $p(1)=2$.
\end{rmk}

 Next we introduce matrix-coefficients associated to a Harish-Chandra dual pair 
$(V,\tilde V)$. 
Let $V_p$ be any Banach globalization of the Harish-Chandra module $V$ with respect to a $G$-continuous norm $p$. 
Write $\pi_p(g)v$ for the action of $g\in G$ on $v\in V_p$.

Restriction to $V$ identifies $(V_p')^{K-{\rm finite}}$ with $\tilde V$: every $K$-finite functional on $V$ is $p$-bounded, and its extension to $V_p$ is unique.

Let now $v\in V\subset V_p$ and $\tilde v \in \tilde{V}$.
Then the map
$$
G \to \mathbb{C},
\quad g\mapsto \tilde v(\pi_p(g)v)
$$
is real analytic. Indeed, the sum of the $K$-isotypes meeting $v$
is finite dimensional and is preserved by
$\Delta=-\Cc+2\Delta_K$. Thus $P(\Delta)v=0$ for some nonconstant
polynomial $P$. The coefficient $f(g)=\tilde v(\pi_p(g)v)$
satisfies the corresponding elliptic equation $P(\Delta)f=0$ on $G$
and is therefore real analytic by analytic elliptic regularity.

All derivatives of $f$ at the identity are determined by the module
and the pairing. Hence $f$ is independent of the globalization on
$G_e$, and therefore on $G=KG_e$ by the prescribed $K$-action.
We thus call

$$
m_{v, \tilde v}:G\to\C,\quad g\mapsto \tilde v (g \cdot v)
$$
the matrix-coefficient attached to $v$ and $\tilde v$.

A norm $p$ on $V$ is called $K$-Hermitian if it is induced by a $K$-invariant scalar product.
The next proposition shows that many questions on $G$-continuous norms can be reduced to the setting of $K$-Hermitian $G$-continuous norms.

\begin{prop}\label{Prop K-Hermitian squeezing}
Let $p$ be a $G$-continuous norm on a Harish-Chandra module $V$. Then there exist $K$-Hermitian $G$-continuous norms $q_{1}$ and $q_{2}$ on $V$ so that
$$
q_{1}
\lesssim p\lesssim q_{2}\,.
$$
\end{prop}

\begin{proof}
From \cite[Theorem 5.5]{BK} it follows that there exists a $K$-Hermitian $G$-continuous norm $q_{2}$ so that $p\lesssim q_{2}$. Likewise, there exists a $K$-Hermitian $G$-continuous norm $\tilde{q}_{1}$ on $\tilde{V}$ so that $\tilde{p}\lesssim\tilde{q}_{1}$. The assertion now follows from Lemma \ref{pq-lemma} with $q_{1}=\tilde{\tilde{q}}_{1}$.
\end{proof}

\subsection{Weights and norms}

By a weight of $G$ we understand a locally bounded positive function $w: G \to \R_{>0}$ that is submultiplicative, i.e.
$$
w(gh)
\leq w(g) w(h)
\qquad (g, h\in G).
$$
We denote by $\No(V,w)$ the set of all $G$-continuous norms $p$
on $V$ 
of growth bounded by a multiple of the weight $w$, i.e.
$$
p(g \cdot v)
\leq C w(g) p(v)
\qquad (g \in G, v\in V)
$$
for some constant $C>0$.

\begin{rmk}\label{Rem Weights}
\begin{enumerate}[(a)]
\item\label{Rem Weights - item 1}
Every $G$-continuous norm $p$ is bounded by a weight, namely the operator norm
$$
w_p(g)
:= \sup_{p(v)\leq 1} p(g\cdot v)
\qquad(g\in G).
$$
\item\label{Rem Weights - item 2}
We write $w\lesssim w'$ if there exists a $c>0$ so that $w\leq c w'$. We may introduce an equivalence relation on the set of weights analogous to the equivalence relation on the set of norms as follows: Two weights $w$ and $w'$ of $G$ are equivalent if
$w\lesssim w'$ and $w'\lesssim w$, i.e., there exists a $c>0$ so that
$$
\frac{1}{c}w(g)
\leq w'(g)
\leq c w(g)
\qquad (g\in G).
$$
\item\label{Rem Weights - item 3} 
Let $V$ be finite dimensional. Then, as all norms on a finite-dimensional space are equivalent, the
weight construction in (\ref{Rem Weights - item 1}) defines a unique equivalence class $[w_V]$. If we
also assume that $V$ is irreducible and $G$ is connected, split and semisimple, then
$[w_V]$ determines $V$. In order to see that, let $G=K\oline{A^+}K$ be a Cartan decomposition. As $G$ is split, the Lie algebra $\af$ of $A$ is a Cartan subalgebra. Hence $V$ is determined uniquely by its highest weight $\lambda_V \in \af^*$. Now observe that there exists a constant $c>0$ so that
$$
c^{-1} a^{\lambda_V}
\leq w_V(a)
\leq c a^{\lambda_V}
\qquad (a\in \oline{A^+}).
$$
Therefore, $[w_V]$ determines $\lambda_V$, and hence also $V$.

These hypotheses exclude invisible unitary twists. On $\GL(n,\R)$ the trivial character and
$\operatorname{sgn}\det$ have the same weight, and on a connected split central torus an imaginary
character has constant modulus.

\end{enumerate}
\end{rmk}

For a weight $w$ we define a new weight $w^\sharp$ of $G$ by
$$
w^\sharp(g)
:=w(g^{-1})
\qquad (g\in G)
$$
and record
\begin{equation}\label{dual}
p \in \No(V,w)\iff \tilde p \in \No(\tilde V, w^\sharp).
\end{equation}
We call a weight $w$ {\it reflexive} if $w=w^\sharp$. 

\subsection{Standard weights} We now provide a variety of basic weights which will be useful later on. To begin with, we recall the standard notation for a real reductive group $G$. 

With regard to our fixed choice of a maximal compact subgroup $K \subset G$, we fix an Iwasawa decomposition $G=KAN$. Let $P=MAN$ be the corresponding minimal parabolic subgroup of $G$ with $M=Z_K(A)$. We write $\overline{N}$ for the unipotent radical of the opposite minimal parabolic subgroup with Levi subgroup $MA$. As customary we use the symbol $\rho$ to denote the Weyl half sum, i.e., the element $\rho\in\af^*$ given by
$$
\rho(X)
=\frac{1}{2} \tr\big(\ad(X)\big|_\nf\big)
\qquad (X\in \af).
$$
Let $\Sigma=\Sigma(\af,\gf)\subset \af^*\bs\{0\}$ be the restricted root system and $\Sigma^+$ the positive system corresponding to the choice of $\nf$. With these choices we obtain
an open Weyl chamber $\af^+$ and we set $A^+=\exp(\af^+)$ as usual.
Finally, we write $W$ for the Weyl group of $\Sigma$ 
and note that $W=N_K(\af)/ M$.

Returning to weights, we first note that $\No(V,w)$ depends only on $[w]$. Next, observe that every weight $w$ is equivalent to a $K\times K$-invariant weight, that is, for
$$
\tilde w (g)
=\sup _ {k_1, k_2\in K} w(k_1 g k_2)
\qquad (g\in G)
$$
we have $[w]=[\tilde w]$. We may therefore assume that the weight is
bi-$K$-invariant. Such a weight is uniquely determined by its
restriction $w_A=w|_A$, which is $W$-invariant.

The following inclusion for double cosets will be used to construct
two standard families of weights.

\begin{lemma} 
For $a,b\in A$ one has 
$$
aKb
\subset K \exp(\mathrm{co} (W\cdot \log a) + \mathrm{co}(W \cdot \log b))K
$$
with $\mathrm{co}(S)$ referring to the convex hull of a subset $S\subseteq\af$.
\end{lemma}
\begin{proof} Since the statement is certainly known we only sketch the argument.  In view of $G=K\oline{A^+}K$ it suffices to estimate the operator norm $\|akb\|$ of $akb$ in all finite dimensional $K$-spherical highest weight representations $V_\lambda$ of $G$ with highest weight $\lambda\in \af^*$. The Cartan-Helgason theorem gives an ample supply of spherical highest weights $\lambda$, namely a semi-lattice of full rank in $\af^*$. We recall Kostant's theorem which says that the convex hull of $A$-weights of $V_\lambda$ is given by $\mathrm{co}(W\cdot \lambda)$. This implies that  $\|k_1 a k_2\|=\max_{w\in W} a^{w\cdot \lambda} $ for $k_1,k_2\in K$ and $a\in A$. With $\|akb\|\leq \|a\|\|b\|$ one completes the proof. 
\end{proof} 

The Lemma implies the following criterion. 

\begin{lemma} Let $w_A$ be a $W$-invariant weight on $A$ such that 
\[
 \sup_{H\in\mathrm{co}(W\cdot\log a)}w_A(\exp H)\leq w_A(a)
 \qquad(a\in A).
\]
Then $w_A$ extends uniquely to a $K\times K$-invariant weight on $G$. 
\end{lemma}
The above criterion gives us two standard families of weights. If $\|\cdot\|$ is a $W$-invariant norm on $\af$ and $d>0$ we obtain the logarithmic weights which are prominent in tempered representation theory
\begin{equation} \label{standard weight log} v_d(g)= ( 1 + \|\log a\|)^d \qquad (g \in KaK, a\in A).\end{equation} 
Further for $\mu\in \af^*$ the prescription 
\begin{equation}\label{standard weight exp} 
w_\mu(g)
=\sup_{w\in W} \max \{ a^{w\cdot \mu}, a^{-w\cdot\mu}\} \qquad (g\in KaK, a\in A)
\end{equation} 
defines the standard $K\times K$-invariant weights on $G$ of exponential type $\mu$.

Both families are bi-$K$-invariant and reflexive; in particular,
$w_\mu=w_{-\mu}=w_\mu^\sharp$.

\subsection{Infimum construction} \label{Subsection infimum}
Recall that for a family $(q_{\alpha})_{\alpha\in \A}$ of seminorms on a vector space $E$ one can define the seminorm $\inf_{\alpha\in \A} q_{\alpha}$ of the family by
\begin{equation} \label{def inf norm}
\inf_{\alpha\in \A} q_{\alpha}(v)
:=\inf_{v =\sum_{\alpha \in \A} v_\alpha }\sum_{\alpha\in \A} q_{\alpha}(v_\alpha)
\qquad(v\in E),
\end{equation}
where the infimum is taken over finite presentations $v =\sum_{\alpha \in \A} v_\alpha$. It is straightforward to see that $\inf_{\alpha\in \A}q_\alpha$ is indeed a seminorm and that it is dominated by all $q_\alpha$. The latter property characterizes this infimum seminorm: 

\begin{lemma} \label{lemma inf characterization} Let $(q_{\alpha})_{\alpha\in \A}$ be a family of seminorms on a vector space $E$. Then 
$\inf_{\alpha\in \A} q_{\alpha}$ is the largest seminorm on $E$ which is dominated by all $q_\alpha$, $\alpha\in \A$.
\end{lemma} 
\begin{proof} Let $q$ be a seminorm with $q\leq q_\alpha$ for all $\alpha\in\A$. Express $v\in E$ as $v=\sum_{\alpha\in \A} v_\alpha$. Then 
$$q(v)=q(\sum_{\alpha\in \A} v_\alpha)\leq \sum_{\alpha\in \A} q_\alpha(v_\alpha)\, .$$
It follows that $q\leq \inf_{\alpha\in \A} q_{\alpha}$, proving the lemma. 
\end{proof} 

\begin{rmk} \label{rmk sup of seminorms}
Dual to the infimum construction of seminorms is the supremum construction. Let $(p_\alpha)_{\alpha\in \A}$ be a family of seminorms on a vector space $F$ which are all dominated by a seminorm $p$, i.e., $p_\alpha\leq p$ for all $\alpha\in \A$. Then we can define 
$$ (\sup p_{\alpha})(v) = \sup_{\alpha\in \A} p_\alpha(v)\qquad (v\in F).$$
Clearly $\sup p_\alpha$ is a seminorm. It is by definition the smallest seminorm on $F$ which dominates all $p_\alpha$. 
\par Let us return to our original setup with a family of seminorms $(q_\alpha)_{\alpha\in \A}$ on some vector space $E$. 
For our purposes it does no harm to assume that $\inf q_\alpha$ is a norm and we will request that henceforth in this remark.
Then for each $\alpha$ the identity map yields a contractive inclusion of normed spaces $\iota_\alpha: (E, q_\alpha)\to (E,\inf q_\alpha)$. Dually we obtain contractive inclusions $\iota_\alpha': (E, \inf q_\alpha)'\to (E, q_\alpha)'$ for the continuous duals. Set 
$F:= (E, \inf q_\alpha)'$ and define restricted dual norms on $F$ by $\tilde q_\alpha:= q_\alpha'\circ \iota_\alpha'$. Then it is not difficult to show that  
$$ (\inf q_\alpha)' = \sup_{\alpha\in \A}\tilde q_\alpha. $$
This identity is on the common dual $F=(E,\inf_\alpha q_\alpha)'$
specified above.

The supremum construction has the advantage that it is geometric and thus more accessible to concrete calculations. 
\end{rmk}

For a $G$-continuous norm $p$ the space of smooth vectors $V_p^\infty$ defines an SF-representation \cite[\S 2.3 \& \S 2.4]{BK}. The Casselman-Wallach theorem asserts that every Harish-Chandra module admits an $SF$-globalization which is unique up to isomorphism. (In Theorem \ref{Thm Casselman-Wallach} another formulation of this theorem will be given.) We will denote the Casselman-Wallach globali\-za\-tion of a Harish-Chandra module $V$ by $V^\infty$. The advantage of $V^\infty$ over $V$ is that 
$V^\infty$ carries a $G$-action. 

Given a Harish-Chandra module $V$, a weight $w$ and a norm $p$ on $V^{\infty}$, which is not necessarily $G$-continuous, we define
$$
p^{G,w}
=p^G
:=\inf_{g\in G} w(g^{-1}) p(g\dotvar )
$$
as the infimum of the family of norms $\big(w(g^{-1}) p(g\dotvar)\big)_{g\in G}$. Here $p(g \dotvar)$ is the norm on $V^{\infty}$ that assigns to $v \in V^{\infty}$ the value $p(g \cdot v)$. 
An explicit formula for $p^{G,w}$ is thus given by: 
$$
p^{G,w}(v)
=\inf_{ v =\sum_{g \in G} v_g } \sum_{g\in G}w(g^{-1})p(g \cdot v_g)
\qquad (v\in V^{\infty}).
 $$

In the remainder of this article we will frequently use that every $G$-continuous norm $p$ extends uniquely to 
$V^\infty=V_p^\infty\subset V_p$.
 We summarize the basic properties of $p^{G,w}$.

\begin{lemma}\label{lemma:Basic properties of inf construction}
Let $p$ be an arbitrary norm on $V^{\infty}$. We denote by $p(g\dotvar)$, $g\in G$,  the norm defined above. 
\begin{enumerate}[(i)]
    \item \label{lemma:Basic properties of inf construction - item 1}

The seminorm $p^{G,w}$ is the largest seminorm on $V^\infty$ satisfying $p^{G,w} \leq w(g^{-1}) p(g\dotvar)$ for all $g \in G$. 
\item \label{lemma:Basic properties of inf construction - item 4} We have $p^{G,w}(g \cdot v) \leq w(g) p^{G,w}(v)$ for any $v \in V^{\infty}$ and $g \in G$.
\item  \label{lemma:Basic properties of inf construction - item 2}
Let $q$ be a norm on $V^{\infty}$. If $p \lesssim q$, then $p^{G,w} \lesssim q^{G,w}$. 
\item  \label{lemma:Basic properties of inf construction - item 3}    Let $p \in \No(V,w)$ considered as a norm on $V^\infty$.
Then $p^{G,w}\sim p$.
\end{enumerate}
\end{lemma}

\begin{proof} (\ref{lemma:Basic properties of inf construction - item 1}) is a special case of Lemma \ref{lemma inf characterization}. 

For (\ref{lemma:Basic properties of inf construction - item 4}), take
a finite decomposition $v=\sum_h v_h$. In the defining infimum for
$p^{G,w}(gv)$, use the decomposition $gv=\sum_h gv_h$, assigning the
summand $gv_h$ to the index $hg^{-1}$. Since $(hg^{-1})g=h$, its
total cost is
\[
 \sum_h w(gh^{-1})p(hv_h)
 \leq w(g)\sum_h w(h^{-1})p(hv_h).
\]
Taking the infimum over the decompositions of $v$ proves the assertion.

For (\ref{lemma:Basic properties of inf construction - item 2}) notice that the positivity of $w$ shows that $p \leq q$ implies $p^{G,w} \leq q^{G,w}$.

For (\ref{lemma:Basic properties of inf construction - item 3}), choose $C$ with $p(gu)\leq Cw(g)p(u)$. Then $C^{-1}p$ is dominated by every norm in the defining infimum, while the term $g=1$ gives $p^{G,w}\leq w(1)p$.

\end{proof}

It can happen that $p^{G,w}$ degenerates to zero. However, this is prevented under a natural condition.

\begin{lemma} \label{inf w lemma}
Let $V$ be a Harish-Chandra module and $w$ a weight. Let $p$ be a norm on $V^{\infty}$ (not necessarily $G$-continuous) such that there exists a norm $q\in \No(V,w)$ and a $G$-continuous norm $r$ with $q\leq p\leq r$. Then $q\lesssim p^{G,w}$ and $p^{G,w}\in \No(V,w)$ is a $G$-continuous norm. 
\end{lemma}

\begin{proof} 
By monotonicity and the preceding lemma, $q\sim q^{G,w}\leq p^{G,w}\leq w(1)r$. Thus $p^{G,w}$ is a norm; its translation bound is $w(g)$. Orbit continuity on $V^\infty$ follows from domination by $w(1)r$, and extends to the completion by the local boundedness of $w$.
For each $\tau\in\widehat K$ the continuous isotypic projection
maps the dense subspace $V$ onto $V[\tau]$. This space is
finite dimensional and hence closed, so it is also the full
$\tau$-isotype in the completion. The continuous equivariant
map from $V^\infty$ to the completion identifies the derived
action on these vectors with the original action on $V$.
\end{proof}

\subsection{Minimal and maximal norms}\label{subsect: min-max norms}

For a Harish-Chandra module $V$ and a weight $w$ of $G$ we consider the set of equivalence classes $\Noc(V, w):= \No(V, w) /\sim $ of $G$-continuous norms on $V$ with growth bounded by a multiple of $w$. We denote the equivalence class of a norm $p$ by $[p]$. Notice that $\lesssim $ on $\No(V,w)$ induces a partial order
$\leq $ on $\Noc(V,w)$.

The following observation is crucial for this article. 

\begin{prop}\label{prop minimal norm}
Suppose that $V$ is a Harish-Chandra module and $w$ is a weight such that $\No(V,w)\neq \varnothing$. Then
$\Noc(V, w)$ has a smallest element $[p_{\min}]$.
\end{prop} 
We will prove this Proposition by providing useful representatives of $[p_{\min}]$ in two different lemmas:

\begin{lemma} \label{lemma minimal norm} Assume that $V$ is a Harish-Chandra module and $w$ is a weight such that $\No(V,w)\neq \varnothing$.
Suppose that $C$ is a compact subset of a finite-dimensional subspace of $\tilde V$ that generates $\tilde V$ as a $(\gf,K)$-module. Then the prescription 
\begin{equation} \label{def pmin}
p_{\min, C}:V\to\R_{\geq 0},\quad v\mapsto \sup_{g \in G}\sup_{\tilde{v}\in C} \frac{ |m_{v, \tilde{v}} (g)|}{w(g)}
\end{equation}
defines a $G$-continuous norm in $\No(V,w)$ which satisfies $[p_{\min,C}]=[p_{\min}]$.
\end{lemma}

\begin{proof}
Let us first show that $p_{\min,C}(v) $ is finite for every $v\in V^\infty$.
For that let $p\in \No(V,w)$ and let $c>0$ be such that $p(g\cdot v)\leq c w(g) p(v)$ for all $g\in G$ and $v\in V^\infty$. Then it follows from
$$
|\tilde{v} (g \cdot v )|\leq p (g \cdot v) \tilde p (\tilde{v})
\leq c w(g) p(v) \tilde p(\tilde{v})
\qquad\big(\tilde{v}\in\tilde{V}, v\in V^\infty, g\in G\big)
$$
that
\begin{equation}\label{eq p_j leq p}
p_{\min,C}(v)
\leq c \big( \sup_{\tilde v \in C} \tilde p(\tilde{v})\big) p(v)
= \tilde c p(v)
\qquad (v\in V^\infty)
\end{equation}
with $\tilde c : = c \sup_{\tilde v \in C} \tilde p (\tilde{v})<\infty$. Hence, $p_{\min, C}$ is indeed
finite and defines a seminorm. Since $C$ is generating there exists for every nonzero $v\in V$ a
$\tilde v \in C$ so that the matrix-coefficient $m_{v,\tilde{v}}$ does not vanish. Therefore, $p_{\min}$
is a norm on $V$.
Moreover, from
$$
w(gh^{-1})
\geq \frac{w(g)}{w(h)}
\qquad(g,h\in G)
$$
we infer that
\begin{equation}\label{eq p_min is w-bounded}
p_{\min,C}(g\cdot v)
\leq w(g) p_{\min,C}(v)
\qquad(g\in G, v\in V^\infty).
\end{equation}

The natural map from $V^\infty$ into its $p_{\min,C}$-completion is continuous by \eqref{eq p_j leq p},
and $V$ is dense in that completion by $K$-finite approximation. One may first quotient the smooth space
by the null space of the seminorm. Estimate \eqref{eq p_min is w-bounded} extends translations to bounded
operators on the completion. Strong continuity follows from continuity of the smooth orbit maps, density,
and local boundedness of $w$. For each $\tau\in\widehat K$ the continuous isotypic projection has dense
image of $V[\tau]$ in the completed isotype. Since $V[\tau]$ is finite dimensional and the norm separates
$V$, its image is already closed. This identifies the $K$-finite vectors with $V$ and proves
$G$-continuity.

Finally, the above arguments show that $p_{\min, C }\lesssim p$ for every $p\in \No(V,w)$. This implies that $[p_{\min,C}]$ is the minimal element of $\Noc(V,w)$.
\end{proof}

An alternative, slightly more flexible construction of the minimal norm in Lemma \ref{lemma minimal norm} can be given using the Casselman-Wallach globalization $V^{\infty}$ of a Harish-Chandra module $V$.

\begin{lemma}\label{lemma smooth minimal} Assume that $V$ is a Harish-Chandra module and $w$ is a weight such that $\No(V,w)\neq \varnothing$.
Suppose that $C \subset {\tilde V}^\infty$ is a compact subset so that the translates $G\cdot C$
separate $V$, that is, if $v\in V$ and $\tilde{v}(v)=0$ for all $\tilde{v}\in G\cdot C$, then $v=0$.
Then 
$$
p_{\min, C}:V\to\R_{\geq 0},
\quad v\mapsto \sup_{g \in G}\sup_{\tilde{v}\in C} \frac{ |m_{v, \tilde{v}} (g)|}{w(g)}
$$
is a norm in $\No(V,w)$ and satisfies $[p_{\min,C}]=[p_{\min}]$.
\end{lemma}

\begin{proof}
Let $p\in \No(V,w)$ and $E=V_p$ the corresponding Banach globalization and likewise $\tilde E$ the
globalization of $\tilde V $ with respect to $\tilde p$. According to Casselman-Wallach we have a
continuous $G$-equivariant inclusion $\tilde V^\infty\to \tilde E$. In particular $C\subset \tilde E$ is
compact and hence $\tilde p$ is bounded on $C$. The fact that $G\cdot C$ separates $V$
guarantees that
$p_{\min, C}$ is a norm. As in the proof of Lemma \ref{lemma minimal norm} we obtain 
$p_{\min,C}(g\cdot v)\leq C w(g) p_{\min,C}(v)$ for a constant $C>0$ and for all $g\in G$, $v\in V^\infty$. The rest of the proof is entirely parallel to the proof of Lemma \ref{lemma minimal norm}.

\end{proof}

\begin{rmk}
Since $\tilde V$ is finitely generated we can take for $C$ a finite generating subset $\{\tilde v_1, \ldots, \tilde v_m\}$ of $\tilde V$. In the special case where $V$, and hence also $\tilde V$, is irreducible, any singleton $C=\{\tilde v\}$ with $\tilde{v}\neq 0$ will do. By abuse of notation we will henceforth drop the extra notation on $C$ and simply write $p_{\min}$ instead of $p_{\min, C}$.
\end{rmk}

By dualizing the construction of the minimal norm in Lemma \ref{lemma minimal norm} we obtain a maximal norm. 
In more detail, recall $\tilde{\tilde V} =V$ and \eqref{dual}. Thus, if $p$ is representative of the
minimal element in $\Noc(\tilde V,w^\sharp)$, then it follows from Lemma \ref{pq-lemma} that
$p_{\max}:=\tilde p$ is a representative of the maximal element in
$\Noc(V,w)$. We have thus shown: 

\begin{prop}\label{prop max norm}
Suppose that $V$ is a Harish-Chandra module and $w$ is a weight such that $\No(V,w)\neq \varnothing$. Then $\Noc(V, w)$ has a greatest element $[p_{\max}]$.
\end{prop}

\begin{rmk} \label{rmk minimal and maximal}(Minimal and maximal Banach globalizations) For a weight $w$ and a Harish-Chandra module $V$ with $\No(V,w)\neq \varnothing$, we choose representatives $p_{\min}$ and $p_{\max}$ of $[p_{\min}]$ and $[p_{\max}]$. Define $V_{\min, w}=V_{p_{\max}}$ as the completion of $(V, p_{\max})$
and likewise $V_{\max,w}=V_{p_{\min}}$. Notice that if $E$ is any Banach completion with respect to a $G$-continuous norm $p$ of growth
$w_p\lesssim w$, then we have continuous monomorphisms $V_{\min,w} \to E \to V_{\max,w}$ of Banach spaces. In particular, $V_{\min, w}$ and 
$V_{\max, w}$ are unique up to isomorphism in the category of Banach-representations of growth class $[w]$. 

Indeed, the comparison maps induced by the identity are injective on each finite-dimensional $K$-isotype.
Applying every isotypic projection to a vector in their kernel and using Peter--Weyl approximation shows that the kernel is trivial and proves
injectivity on the completions.
\end{rmk} 

\subsection{On the geometries of unit balls for the minimal and maximal norm}\label{Subsect geometry of balls} 
In this subsection we reformulate the existence of minimal and maximal elements $[p_{\min}]$ and $[p_{\max}]$ in $\Noc(V,w)$ in geometric terms, i.e., by an explicit description of the unit balls $B_{\min}$ and $B_{\max}$ for suitable representatives $p_{\min}$ and $p_{\max}$.
The ambient space and the topology of closure enter the two ball descriptions explicitly.

We write $\tilde V^\infty$ for a fixed choice of an $SF$-globalization of $\tilde V$ and $\tilde V^{-\infty}$ for the strong dual of $\tilde V^\infty$. Note that $V\subset \tilde V^{-\infty}$. Let now $p$ be an element of $\No(V,w)$. We can and will realize $V_p$ in $\tilde V^{-\infty}$. Let $B_p\subset V_p$ be the unit ball of $V_p$. Then the existence of minimal and maximal elements 
in $\Noc(V,w)$ translates into the geometric squeezing 
\begin{equation} \label{ball comparison}
\frac{1}{c}B_{\max}
\subset B_p \subset cB_{\min}\,
\end{equation}
for a constant $c=c(p)>0$.
We now need concrete representatives $p_{\min, C}$, resp. $p_{\max, D}$ of $[p_{\min}]$, resp. $[p_{\max}]$.
Recall from Lemma \ref{lemma minimal norm} that for a generating compact subset $C$ of a finite-dimensional subspace of $\tilde V$ the norm
$$
p_{\min,C} (v)
=\sup_{\substack{g\in G\\\tilde v \in C}} \frac{|\tilde v (g\cdot v)|}{w(g)}
\qquad(v\in V)
$$
defines a representative of $[p_{\min}]$ in $\No(V,w)$.

For a generating compact subset $D$ of a finite-dimensional subspace of $V$ we define
$$
p_{\max,D}
:= \tilde{p_{\min,D}}\, ,
$$
where $p_{\min,D}\in \No(\tilde{V},w^{\sharp})$ is the representative for the minimal element in
$\Noc(\tilde V,w^{\sharp})$ defined by the set $D$ as above.
Note that $p_{\max,D}$ is a representative of $[p_{\max}]$.

We introduce the notation
$$
G\cdot_w C
=\left\{ \frac{g \cdot \tilde v}{w^{\sharp}(g)}\mid g \in G, \tilde v \in C\right\}
$$
for the weighted $G$-orbit of $C\subseteq \tilde{V}$.
The following properties for unit balls follow immediately from the definitions and properties of $p_{\min,C}$ and $p_{\max,D}$.

\begin{prop}\label{Prop Balls}
Let $V$ be a Harish-Chandra module and $w$ a weight such that
$\No(V,w)\neq \varnothing$. Let $C\subset \tilde V$ and $D\subset V$ be generating compact subsets of finite-dimensional subspaces of $\tilde{V}$ and $V$, respectively. Then the following assertions hold:
\begin{enumerate}[(i)]

\item In the realization $V_{p_{\min,C}}\subset\tilde V^{-\infty}$ the completed unit ball is
\[
 B_{\min}=V_{p_{\min,C}}\cap(G\cdot_w C)^\circ .
\]

\item Set $E=V_{p_{\max,D}}$ and pair it with $\tilde V$ by continuous extension from $V$. Then
\[
 B_{\max}=\overline{\operatorname{aconv}(G\cdot_{w^\sharp}D)}^{\,\sigma(E,\tilde V)}
 \quad\text{in }E ,
\]
where $\operatorname{aconv}$ denotes the absolutely convex hull.

\end{enumerate}
\end{prop}

\begin{proof}
The weighted evaluations defining $p_{\min,C}$ are continuous on the completion and have norm at most
one. Their supremum agrees with the norm by density, giving (i). For (ii) write $S=G\cdot_{w^\sharp}D$.
Its polar in $\tilde V$ is precisely the unit ball for the dual-module minimal norm. By the definition of
$p_{\max,D}$ this set norms $V$, and by density it norms $E$ as well. It separates $E$. The bipolar
theorem for $(E,\tilde V)$ now gives the asserted closure.
\end{proof}

The intersection in (i) is necessary: the disk-model example in
Remark~\ref{r97:disk-polar} gives a distribution vector in the polar
that does not belong to the completion.

\section{On the various definitions of Sobolev norms for a representation}
\label{Section Sobolev norms}

Let $(\pi, E)$ be a Banach representation of a Lie group $G$ with Lie algebra $\gf$
and $p$ a $G$-continuous norm on $E$ which induces the topology.
In this context one can define a collection of norms called Sobolev norms on the space of smooth vectors $E^{\infty}$ in $E$. There are various ways of constructing these families.

\subsection{Standard Sobolev norms}

For a fixed basis $X_1, \ldots, X_n$ of $\gf$ we define for every $k\in\N_{0}$ a  norm $p_k^{\rm st}$
by
$$
p_k^{\rm st}(v)
:=\left(\sum_{\substack{\alpha\in \N_0^n\\ |\alpha|\leq k}} p(X_1^{\alpha_1}\ldots X_n^{\alpha_n} v)^2 \right)^{\frac{1}{2}} \qquad (v\in E^\infty).
$$
Different choices of a basis for $\gf$ lead to equivalent norms.
By slight abuse of notation we refer to $p_k^{\rm st}$ as the $k$-th standard Sobolev norm of $p$. We remark that $p_k^{\rm st}$ is $G$-continuous as well.

If $p$ is $K$-Hermitian, the equivalent Hilbert norm
$\big(\int_K p_k^{\rm st}(h v)^2\,dh\big)^{1/2}$ is $K$-Hermitian.
Indeed, $\Ad(K)$ acts boundedly on the finite-dimensional space of differential operators of order at
most $k$, and applying the same bound to $h^{-1}$ gives the reverse comparison.

\subsection{Sobolev norms constructed from Laplacians}

For the second construction we assume for convenience that $G$ is unimodular.
Again we fix a basis $\B=\{X_1, \ldots, X_n\}$ of $\gf$ and associate to it a Laplace element in $\U(\gf)$
    \begin{equation}\label{eq Delta}
\Delta
:=\Delta_\B:=-(X_1^2 + \ldots + X_n^2).
\end{equation}

There exists a constant $R_E\geq 0$ depending on the growth rate $w_p$ of $p$ such that for all $R>R_E$
the operator
$R^2+\Delta: E^\infty\to E^\infty$ is invertible; see \cite{GK}. 
This allows us to define the Laplace Sobolev norms of even order $p_{2k}^\Delta$ on $E^{\infty}$ for $k\in\Z$ by
$$
p_{2k}^\Delta(v)
:=p((R^2+\Delta)^kv)
\qquad(v\in E^{\infty}).
$$
Since $p_{2k}^\Delta$ depends on $\B$ and $R$ the notation
$p_{2k}^{\Delta_\B, R}$ would be more precise. However, for reasons of readability we will not use it anymore after the following digression.

For a unitary Hilbert norm one can take $R_E=0$, and at each
even order the equivalence class of the Laplace norm is
independent of the basis and of $R>0$. For Banach norms a
change of basis can change this class already at order two;
see Remark~\ref{rmk:projective-banach}. A comparison with a
fixed shift of order always follows from the next proposition.

\begin{prop}[{\cite[Prop. 4.1]{GK}}]\label{Prop equivalence of Sobolev norms}
Let $m\in2\N_{0}$ be the smallest even number with $m\geq \dim(G)+1$. Then for all $k\in \N_{0}$
\begin{equation}\label{standard Sobolev squeeze}
p_{2k}^\Delta
\lesssim p_{2k}^{\rm st}
\lesssim p_{2k+m}^\Delta.
\end{equation}
\end{prop}

\begin{rmk}
A single Laplace Sobolev norm can fail to be $G$-continuous,
as in Remark~\ref{rmk:projective-banach}. The comparison
\eqref{standard Sobolev squeeze} nevertheless shows that the
family of these norms induces the smooth Fr\'echet topology.
\end{rmk}

\subsection{The Casselman-Wallach globalization}
We return to reductive groups and recall the Casselman-Wallach theorem from \cite{Casselman}, \cite{Wallach} and \cite[Chapter 11]{Wallach_book_II}. The theorem may be reformulated as follows. See \cite[Theorem 1.1]{BK}.

\begin{theorem}\label{Thm Casselman-Wallach}
For any pair of $G$-continuous norms $p, q$ on a Harish-Chandra module $V$ there exists a $k\in \N$ such that $p \lesssim q_k^{\rm st}$.
\end{theorem}

The above theorem can also be phrased as follows: For any two $G$-continuous norms $p$ and $q$ on a Harish-Chandra module $V$ the identity map $V\to V$ extends to a $G$-isomorphism of Fr\'echet spaces $V_p^\infty\to V_q^\infty$. 
We denote by $V^\infty$ the unique $SF$-globalization up to isomorphism of $V$. (See \cite[\S 2.3 \& \S 2.4]{BK} for the notion of an SF-representation.)

If $\kappa: \gf \times \gf \to\R$ denotes the Cartan-Killing form
of $\gf$ and $\theta$ the Cartan involution on $G$ with fixed point group 
$K$, then 
$$\la X, Y\ra_{[\gf,\gf]}:= -\kappa (X, \theta(Y)) \qquad (X, Y\in [\gf,\gf])$$
defines a scalar product on the semisimple part $[\gf,\gf]$ of $\gf$.
Note $\gf=[\gf, \gf]\oplus \zf(\gf)$ with $\zf(\gf)$ the center of $\gf$. We fix an $\Ad(K)$- and $\theta$-invariant scalar product $\la\cdot, \cdot\ra_{\zf(\gf)}$ on $\zf(\gf)$ and define a scalar product on $\gf=[\gf,\gf]\oplus \zf(\gf)$ by 
$$\la\cdot, \cdot\ra = \la\cdot, \cdot\ra_{[\gf,\gf]} + \la\cdot, \cdot\ra_{\zf(\gf)}.$$

We now characterize
the Casselman-Wallach globalization $V^{\infty}$ of a Harish-Chandra module $V$ in terms of $K$-types.     
Fix a maximal torus $T$ of $K_e$ and positive roots for $(\kf_\C,\tf_\C)$. For $\tau\in\hat K$, let $\lambda_\tau$ be the highest weight of a constituent of $\tau|_{K_e}$ and let $|\tau|=|\lambda_\tau|$ in the induced norm. Its value and the $\Delta_K$-eigenvalue are independent of the constituent, since the constituents are conjugate under $K$.

We choose an orthonormal basis of $\kf$ and extend it to an orthonormal basis $X_1, \ldots, X_n$ of $\gf$. Let $\Delta_{K}$ and $\Delta$ be the corresponding Laplace elements for $\kf$ and $\gf$, respectively, as in (\ref{eq Delta}). Then
\begin{equation}\label{eq Delta=-Cc+2Delta_K}
\Delta
=-\Cc + 2 \Delta_K,
\end{equation}
where $\Cc$ is the Casimir element. We recall from \cite[Cor. 3.10]{BK} that a vector $v \in V_p$ is smooth
if and only if it is $K$-smooth and note that
\begin{equation}\label{eq restriction of Delta_K to V[tau]}
\Delta_K|_{V[\tau]}
=(|\lambda_\tau+\rho_\kf|^2 -|\rho_\kf|^2)\cdot \id_{V[\tau]} \qquad(\tau\in\hat{K}).
\end{equation}
Here $V[\tau]$ is the $\tau$-isotypical component of $V$ and $\rho_{\kf}$ is the half-sum of the positive roots of $\tf$ in $\kf_{\C}$.
This leads us to the following.

\begin{prop} \label{prop rapid decay}
Let $V$ be a Harish-Chandra module and let $p$ be a $G$-continuous norm on $V$. Consider the formal series $v =\sum_{\tau \in \hat K}v_\tau$, where $v_\tau \in V[\tau]$ for every $\tau\in\hat{K}$. Then the following are equivalent.
\begin{enumerate}[(i)]
\item\label{prop rapid decay - item 1} The series $v$ converges in
$V_p^\infty=V^\infty$.
\item\label{prop rapid decay - item 2} $\displaystyle \sum_{\tau \in \hat K} ( 1+|\tau|)^n p(v_\tau)<\infty$ for every $n\in\N$.
\item \label{prop rapid decay - item 3}$\displaystyle \sup_{\tau\in \hat K} (1+|\tau|)^n p(v_\tau)<\infty$ for every $n\in\N$.
\end{enumerate}
\end{prop}

\begin{proof}

Write $c_\tau=|\lambda_\tau+\rho_\kf|^2-|\rho_\kf|^2$, so $1+c_\tau\asymp(1+|\tau|)^2$. Both the number of $K$-types of size at most $R$ and their dimensions are polynomially bounded in $1+R$: for $K_e$ this follows by lattice counting and Weyl's dimension formula. For $K$ use restriction to $K_e$, the identity component of $K$, and induction back across the finite index: every $K$-type containing a given constituent occurs in its induced representation, of dimension $[K:K_e]$ times the constituent dimension. This proves (ii)$\Leftrightarrow$(iii).

Use the equivalent norm $\sup_{k\in K}p(k\dotvar)$. The projection 
$$P_\tau v=(\dim\tau)\int_K\overline{\chi_\tau(k)}\pi_p(k)v\,dk$$
 has norm at most $(\dim\tau)^2$. For smooth $v$,
\begin{equation}\label{Estimate Delta}
 (1+c_\tau)^j p(P_\tau v)\leq(\dim\tau)^2 p((1+\Delta_K)^j v)
\end{equation}
for the renormed $p$, which proves (i)$\Rightarrow$(iii). 

Conversely, assume (ii). On $V[\tau]$ the operator
$(R^2+\Delta_K)^j$ acts by the scalar $(R^2+c_\tau)^j$,
which is $O_{R,j}((1+|\tau|)^{2j})$. Hence, for every $j\geq0$,
\[
 \sum_{\tau\in\hat K}p\bigl((R^2+\Delta_K)^jv_\tau\bigr)<\infty .
\]
Taking $R$ sufficiently large and applying
Proposition~\ref{Prop equivalence of Sobolev norms} to $K$,
we see that the partial sums are Cauchy, and hence converge,
in the Fr\'echet space of $K$-smooth vectors.
By \cite[Cor.~3.10]{BK}, the natural inclusion of the
$G$-smooth vectors into the $K$-smooth vectors is surjective.
It is also continuous, so the open mapping theorem identifies
the two Fr\'echet topologies. Thus the series converges in
$V_p^\infty$, as asserted in (i).

\end{proof}

\subsection{Sobolev norms with real Sobolev parameter}\label{Subsection p_s}

For a given $G$-continuous norm $p$ we now construct a family of Sobolev norms $p_{s}$ with $s\in\R$.
We define a block diagonal 
operator
$D^s:= (1+\Delta_K)^{\frac{s}{2}}$ on
$\prod_{\tau\in \hat K} V[\tau]=(\tilde V)^*\supset V$ by
\begin{equation}\label{eq D^s}
D^s|_{V[\tau]} = \big(1+|\lambda_\tau+\rho_\kf|^2 -|\rho_\kf|^2\big)^{\frac{s}{2}}\cdot \id_{V[\tau]}
\qquad(\tau\in\hat{K}).
\end{equation}
From Proposition \ref{prop rapid decay} (\ref{prop rapid decay - item 3}) and the fact that $D^{-s}$ is the inverse of $D^{s}$ we see that $D^{s}$ induces a continuous isomorphism on $V^\infty$. We define an $s$-th Sobolev norm $p_s$ on $V$ by
$$
p_s(v)
:= p(D^s v )
\qquad (v \in V).
$$
We list the following properties
\begin{align}
\label{pq comparison}
&p\leq q \quad\Rightarrow\quad p_s \leq q_s \qquad (s\in \R)\\
\label{s plus t}
&(p_s)_t = p_{s+t} \qquad (s, t\in \R)\\
\label{ps dual}
&\tilde{(p_s)}= \tilde p_{-s} \qquad (s\in \R)\\
\label{ps double dual}
&\tilde{\tilde{(p_s)}}\mathrel{\sim}p_s\qquad (s\in\R).
\end{align}
Although $p_s$ might not be $G$-continuous, it is certainly $K$-continuous. This allows us to adapt the proof of Proposition 
\ref{prop rapid decay} and derive that the family $(p_s)_{s\geq 0}$ determines the topology on $V^\infty$. 

In \eqref{ps double dual} one has equality for $K$-invariant $p$, by Lemma~\ref{pq-lemma}.

\begin{prop}\label{Prop s-topology}
Let $V$ be a Harish-Chandra module and let $p$ be a $G$-continuous norm on $V$. Then 
the family of norms $(p_s)_{s\geq 0}$ determines the topology on $V_{p}^\infty\simeq V^{\infty}$.
\end{prop}

\begin{proof} For the equivalent $K$-invariant norm used above, $p_s(P_\tau v)\leq(\dim\tau)^2p_s(v)$, since $D^s$ commutes with $P_\tau$. Conversely $p_s(v)\leq\sum_\tau(1+c_\tau)^{s/2}p(v_\tau)$. These inequalities and the polynomial counting bound identify the topology of the $p_s$ with the rapid-decay topology in Proposition \ref{prop rapid decay}.\end{proof}

The Casselman-Wallach Theorem \ref{Thm Casselman-Wallach} can thus be reformulated as follows.

\begin{cor}Let $p,q$ be $G$-continuous norms on a Harish-Chandra module $V$. Then there exists an $s \geq 0$ such that $q \lesssim p_s$.
\end{cor}

The family $(p_{s})_{s\in\R}$ is monotone in the following sense.

\begin{prop}\label{Prop monotonous}{\rm{(Monotonicity)}} Let $p$ be a $K$-continuous norm on $V$. Then 
$p\lesssim p_s$ for all $s\geq 0$.
\end{prop} 

\begin{proof}
Let $M_p=\sup_{k\in K}\|\pi(k)\|_p<\infty$. Integration against the probability heat kernel of $\Delta_K$
gives $\|e^{-t\Delta_K}\|_p\le M_p$. For $s>0$ the scalar gamma integral on each $K$-isotype gives
\[
 D^{-s}v=\frac{1}{\Gamma(s/2)}\int_0^\infty
 t^{s/2-1}e^{-t} e^{-t\Delta_K}v\,dt .
\]
Thus $p(D^{-s}v)\le M_p p(v)$. Substitute $D^sv$ for $v$ to obtain the assertion. The case $s=0$ is immediate.
\end{proof}

\begin{lemma}[Smooth dyadic cutoffs and spectral blocks]]\label{lem:K-dyadic}
Let $(\pi,E)$ be a continuous Banach representation of $K$, with
norm $Q$, and put $M_Q=\sup_{k\in K}\|\pi(k)\|_Q$.
\begin{enumerate}[(i)]\item For $m\in C_c^\infty((0,\infty))$, $b\in\R$, and $j\geq0$ one has
\begin{equation}\label{eq:sobolev-norm-dyadic}
 Q\bigl(D^b m(2^{-j}D)v\bigr)
 \leq C_{K,m,b}\,M_Q\,2^{jb}Q(v)
 \qquad(v\in E).
\end{equation}
\item Choose $\psi,\chi\in C_c^\infty((0,\infty))$ with $\supp\psi\subset[1/2,2]$, $\sum_{j\geq0}\psi(2^{-j}x)=1$ for $x\geq1$, and $\chi=1$ on $[1/2,2]$. Write
\[
 B_j=\psi(2^{-j}D),\qquad S_j=\chi(2^{-j}D).
\]
Then $v=\sum_{j\geq0}B_jv$ is a finite sum on $K$-finite vectors, and $S_jB_j=B_j$. The two images have $D$-spectrum in $[2^{j-1},2^{j+1}]$ and $[c2^j,C2^j]$, respectively, where $0<c<C$ and $\supp\chi\subset[c,C]$. For every real $b$,
\[
 Q(D^bB_jv)+Q(D^bS_jv)\le C_b M_Q 2^{jb} Q(v),
\]
with the cutoffs and $K$ fixed. For a unitary $K$-representation the orthogonal projection $P_j=\mathbf1_{[2^j,2^{j+1})}(D)$ satisfies $P_jS_j=P_j$.
\end{enumerate}
\end{lemma}

\begin{proof}
On $K_e\times\mathbb T$ take the product Laplacian $\Delta$,
with $\mathbb T$ of length $2\pi$. Put
$h(x)=x^{b/2}m(\sqrt{x})$, extended by zero at $x=0$.
By \cite[Lemma~4.1]{GellerMayeli}, the convolution kernels of
$h(2^{-2j}\Delta)$ have uniformly bounded $L^1$-norms.
In the representation $(k,z)\mapsto z\pi(k)$ the circle Laplacian
acts by $1$, so these kernels act by
$h(2^{-2j}D^2)=2^{-jb}D^b m(2^{-j}D)$.
Integration against the kernels gives the estimate.
For (ii), take $\psi(x)=\theta(x)-\theta(2x)$, where $\theta$ is smooth, equals one on $[0,1]$, and vanishes on $[2,\infty)$. The partition identity follows by telescoping. The support assertions and operator identities hold on each $D$-eigenspace; only finitely many blocks meet the spectrum of a $K$-finite vector. Applying (i) to the two cutoffs gives the stated norm estimate.
\end{proof}

For $p\in \No(V,w)$ the Sobolev norm $p_s$ is in general not contained in $\No(V,w)$, not even in case $p_s$ is $G$-continuous. The reason for this is that the growth bound for $p$ need not persist under a Sobolev shift. This can be remedied with the following lemma, which will be useful in the constructions in the next section. 

\begin{lemma} \label{lemma Sobolev domination with and without infimum construction} Let $p$ be a $G$-continuous norm. 
\begin{enumerate}[(i)]
    \item  Let $p\in\No(V,w)$ and $s\geq0$. Then
$$
p_s^{G,w}
:= (p_s)^{G,w}
\in \No(V,w).
$$
\item Let $p \in \No(V,w)$ and let $q$ be a $G$-continuous norm on $V$. Then for any $s \geq 0$ we have: 
    $p \lesssim q_s$ if and only if $p \lesssim q_{s}^{G,w}$ 
\end{enumerate}
\end{lemma}

\begin{proof} (i) As the topology on $V^\infty=V_p^\infty$  is induced from the family $(p_{2k}^{\rm st})_{k\in \N}$ we find a $k\in \N$ such that $p_s \lesssim p_{2k}^{\rm st}$. We recall that $p_{2k}^{\rm st}$ is $G$-continuous. 
With Proposition \ref{Prop monotonous} we thus get $p\lesssim p_s \lesssim p_{2k}^{\rm st}$ and the assertion follows from 
Lemma \ref{inf w lemma}.

(ii) 
    For any norm $r$ on $V$ we have $r^{G,w} \lesssim r$. Thus the inequality  $p \lesssim q_s^{G,w}$ implies $p\lesssim q_s^{G,w} \lesssim q_s.$
    On the other hand, if $p \lesssim q_s$ then in view of Lemma \ref{lemma:Basic properties of inf construction} we have $p\sim p^{G,w} \lesssim q_{s}^{G,w}$. 
\end{proof}

For $K$-Hermitian $G$-continuous norms $p$ our newly introduced Sobolev norms $p_s$ relate directly to Laplace Sobolev norms.

\begin{lemma}\label{Lemma equivalence of sobolev norms for K-hermitian norms}
Let $V$ be a Harish-Chandra module on which the Casimir acts by a scalar and let $p$ be a
$G$-continuous $K$-Hermitian norm on $V$. Then the following hold:
 For all $k\in\Z$ one has 
$$
p^{\Delta}_{2k}
\sim p_{2k}.
$$
\end{lemma}

\begin{proof}
Let $\gamma$ be that scalar and $c_\tau$ the eigenvalue of $\Delta_K$ on $V[\tau]$. By \eqref{eq
Delta=-Cc+2Delta_K} the multiplier of $R^2+\Delta$ on this block is $R^2-\gamma+2c_\tau$. For an
admissible sufficiently large $R$, its modulus divided by $1+c_\tau$ is bounded above and below by
positive constants independently of $\tau$. The $K$-blocks are orthogonal for $p$. Raising the ratio to
any integral power proves the claim.
\end{proof}

\begin{lemma}[Sobolev norms for a unitary globalization]
\label{lem:unitary-K-Sobolev}
Let $V$ be a Harish-Chandra module on which the Casimir acts by a
scalar, and let $q$ be a unitary norm on $V$. Then
\[
   q_j^{\rm st}\sim q_j \qquad (j\in\N_0).
\]
Moreover, $q_s$ is $G$-continuous for every $s\in\R$.
\end{lemma}

\begin{proof}
Write $(\pi,\Hc)$ for the unitary globalization with respect to the inner product 
$\la v, v\ra =q(v)^2$, put $D=D^1$, and
denote the completion of $V$ in $q_s$ by $\Hc_s$. Thus $\Hc_0=\Hc$.
By orthogonality of the $K$-isotypes,
\[
   q_s(v)^2=q(D^sv)^2
      =\sum_{\tau\in\widehat K}(1+c_\tau)^s q(v_\tau)^2,
   \qquad c_\tau=\Delta_K|_{V[\tau]}\ \text{as a scalar}.
\]

Choose the orthonormal basis $(X_i)$ used in
\eqref{eq Delta=-Cc+2Delta_K}, and write $X_iv=d\pi(X_i)v$.
For $v\in V^\infty$, unitarity gives
\[
   q(v)^2+\sum_i q(X_iv)^2
      =\langle(1+\Delta)v,v\rangle
      \asymp q(Dv)^2.
\]
Indeed, if $\gamma$ is the Casimir scalar, the two quadratic forms
have respective weights $1-\gamma+2c_\tau$ and $1+c_\tau$.
The first weight is at least one by positivity of $\Delta$, and
the ratio tends to two as $c_\tau\to\infty$.

To pass to higher orders, consider the $K$-intertwiner
\[
   T:V^\infty\longrightarrow\gf_\C^*\otimes V^\infty,
   \qquad Tv=\sum_i X_i^*\otimes X_iv.
\]
Let $D_{\rm diag}$ be the operator $(1+\Delta_K)^{1/2}$ for the
diagonal $K$-action on the target. Then $TD^j=D_{\rm diag}^jT$.
On $\gf_\C^*\otimes \Hc[\tau]$, every diagonal $K$-type $\nu$
satisfies $1+|\nu|\asymp1+|\tau|$, uniformly in $\tau$:
the upper bound is the highest-weight bound for tensoring with a
fixed finite-dimensional module, and the reverse bound follows
by tensoring with its dual. Since
$(1+c_\tau)^{1/2}\asymp1+|\tau|$, orthogonality gives
\[
   \|D_{\rm diag}^j Tv\|^2_{\gf_\C^*\otimes\Hc}
      \asymp \|(1\otimes D^j)Tv\|^2_{\gf_\C^*\otimes\Hc}
      =\sum_i q_j(X_iv)^2.
\]
Applying the order-one estimate to $D^jv$ therefore yields
\[
   q_{j+1}(v)^2\asymp q_j(v)^2+\sum_i q_j(X_iv)^2.
\]
Induction identifies $q_j$ with the norm formed from all
differential words of length at most $j$. By the
Poincar\'e--Birkhoff--Witt theorem these words span the same
filtered space as the ordered monomials defining $q_j^{\rm st}$.
This proves the first assertion.

We next prove $G$-continuity. For
$u\in \U_{\leq j}(\gf)$ and $v\in V^\infty$ one has
\[
   d\pi(u)\pi(g)v
     =\pi(g)d\pi\bigl(\Ad(g^{-1})u\bigr)v.
\]
The adjoint action on the finite-dimensional space
$\U_{\leq j}(\gf)$ is uniformly bounded on compact sets.
By unitarity and the integer-order comparison just proved,
$\pi(g)$ extends to $\Hc_j$, and for each compact $C\subset G$,
\[
   M_j(C):=\sup_{g\in C}\|\pi(g)\|_{\Hc_j\to \Hc_j}<\infty.
\]
The same identity, applied to each of the finitely many monomials
defining $q_j^{\rm st}$, shows that
$\pi(g)v\to v$ in $\Hc_j$ as $g\to e$ for $v\in V^\infty$.

Now let $s=j+\theta$ with $j\in\N_0$ and $0<\theta<1$.
Complex interpolation of weighted $\ell^2$-spaces
\cite[Th.~5.5.3]{BerghLofstromSob} gives, isometrically,
\[
    [\Hc_j,\Hc_{j+1}]_\theta=\Hc_s.
\]
Indeed, in an orthonormal basis adapted to the $K$-isotypes,
the squared-norm weights interpolate as
\[
    \big((1+c_\tau)^j\big)^{1-\theta}
    \big((1+c_\tau)^{j+1}\big)^\theta=(1+c_\tau)^s.
\]
The endpoint actions agree on $V^\infty$. Interpolating  their
operator bounds \cite[Cor.~5.5.4]{BerghLofstromSob} gives
\[
   \sup_{g\in C}\|\pi(g)\|_{\Hc_s\to \Hc_s}
      \leq M_j(C)^{1-\theta}M_{j+1}(C)^\theta .
\]
For $v\in V^\infty$, continuity in $\Hc_{j+1}$ implies continuity
in $\Hc_s$, since $q_s\leq q_{j+1}$. As $V^\infty$ contains the
dense subspace $V$ of $\Hc_s$, the local uniform bound gives strong
continuity on all of $\Hc_s$. This argument also gives strong
continuity at the integer endpoints.

Finally, for $t>0$ the $\Hc$-pairing identifies $\Hc_{-t}$ with the
continuous anti-dual of $\Hc_t$. By unitarity, the action of $g$ on
$\Hc_{-t}$ is dual to that of $g^{-1}$ on $\Hc_t$ and hence is locally
uniformly bounded. It is strongly continuous on the dense subspace
$\Hc\subset \Hc_{-t}$ because $q_{-t}\leq q$. Density gives strong
continuity on $\Hc_{-t}$ as well.
\end{proof}

\begin{rmk}\label{rmk:projective-banach}
A $G$-continuous Banach norm $p$ can have a shift $p_2$
which fails to be $G$-continuous. Take $G=\SL(3,\R)$ acting on
half-densities on $X=\mathbb{RP}^2$. Trivialize by the round
density and put $p(f)=\|f\|_\infty$. The action on $C(X)$ is
continuous, and $\Delta_K$ is a positive multiple of the round
Laplacian.

The geometric point is that a projective transformation fixing a
line can stretch its two tangent directions unequally. Indeed, at
$o=[\R e_1]$ we have
\[
   T_oX=\operatorname{Hom}(\R e_1,\R^3/\R e_1).
\]
Thus $g=\diag(1,2,1/2)$ fixes $o$ and acts on the tangent
directions corresponding to $e_2,e_3$ by $2$ and $1/2$,
respectively. A Hessian whose two diagonal entries cancel in the
round Laplacian can therefore acquire a nonzero trace under
pullback by $g$.

To realize this observation, choose conformal coordinates
$z=(x,y)$ centered at $o$, with axes in these two directions.
In these coordinates $\Delta_K=-b(z)\Delta_0$, where $b$ is
smooth and positive and $\Delta_0=\partial_x^2+\partial_y^2$.
Rescale the coordinates so that the chart contains the closed
unit disk. Choose $u\in C_c^\infty(\{|z|<1\})$ equal to
$x^2-y^2$ on $|z|\leq 1/2$, and put
\[
   f_N(z)=\sum_{j=1}^N 4^{-j}u(2^jz),
\]
extended by zero outside the chart. The geometric series bounds
$\|f_N\|_\infty$ uniformly. Moreover,
\[
   \Delta_K f_N(z)=-b(z)\sum_{j=1}^N(\Delta_0u)(2^jz).
\]
Since $x^2-y^2$ is harmonic, the summands are supported on
annuli with disjoint interiors. Hence $\|\Delta_K f_N\|_\infty$
is uniformly bounded, and so is $p_2(f_N)$.

Near $o$, however, every summand of $f_N$ equals $x^2-y^2$.
Thus $f_N(o)=0$, $df_N(o)=0$, and its Hessian at $o$ is
$2N\diag(1,-1)$. Write $\pi(g)f=m_g(f\circ g^{-1})$, where
$m_g$ is the smooth positive half-density multiplier. The
vanishing of the value and first derivatives leaves only the
Hessian term in the chain rule, giving
\[
   |\Delta_K\pi(g)f_N(o)|
   =2N b(o)m_g(o)\,\bigl|\,\tfrac14-4\,\bigr|.
\]
Thus $p_2(\pi(g)f_N)$ grows at least linearly in $N$. Approximation
in $C^2(X)$ by $K$-finite vectors preserves these estimates up to
a bounded error. Consequently $p_2$ cannot be $G$-continuous on
the Harish-Chandra module.

The same example gives basis dependence for the full Laplacian.
For the adapted basis $\B$ we have $\Delta_\B=-\Cc+2\Delta_K$,
and $\Cc$ acts by a scalar. For $\B'=\operatorname{Ad}(g^{-1})\B$
we have $\Delta_{\B'}=\pi(g^{-1})\Delta_\B\pi(g)$. Since $\pi(g)$
and its inverse are bounded for $p$, the preceding estimates give,
for the same sufficiently large $R$,
\[
   \sup_N p\bigl((R^2+\Delta_\B)f_N\bigr)<\infty,
   \qquad
   p\bigl((R^2+\Delta_{\B'})f_N\bigr)\longrightarrow\infty.
\]
Hence the two order-two norms are inequivalent.

The question of $G$-continuity of $p_s$ is directly related to the dependence of $p_{s}$ on the choice of the maximal compact subgroup $K$ as all such groups are conjugate under $G$. Indeed,
if $(p'_{s})_{s\geq 0}$ denotes the family of Sobolev norms obtained from a conjugate $K'=gKg^{-1}$ of $K$, then
$$
p'_{s}
\lesssim p_{s}(g^{-1}\dotvar).
$$
We will return to this issue in Remark \ref{Rem pseudometric} where we show that the Sobolev gap is independent of the choice of $K$.
\end{rmk}

\section{The Sobolev gap of a Harish-Chandra module}\label{section Sobolev gap}
Let $\Noc[V]$ be the set of all equivalence classes of $G$-continuous norms on a Harish-Chandra module $V$. Note that 
$\Noc[V]=\bigcup_{w} \Noc(V,w)$ is naturally filtered by the weights
$w$. 
With fine Sobolev norms $p_s$ for every real $s\in\R$ at our disposal, we are going to define now a Sobolev distance between any two elements 
$[p],[q]\in\Noc[V]$ inducing a pseudometric on the space. If $\No(V,w)\neq \varnothing$, then the distance between $[p_{\min}]$ and $[p_{\max}]$ in $\Noc(V,w)$ is of special interest and yields a new invariant $\gap(V,w)\in \R_{\geq 0}$ of $V$, the so-called Sobolev $w$-gap of $V$.

The section starts with the definition of the Sobolev gap and a discussion of its basic properties. After that we define the pseudometric on $\Noc[V]$ and discuss special properties such as stabilization or distance to the unitary norm in case $V$ is unitarizable and $w=\1.$

Throughout this section the weight $w$ is fixed.

By $V$ we will always denote a Harish-Chandra module with $\No(V,w) \ne \varnothing$. 
As $w$ is fixed we abbreviate
$$
p^{G}:=p^{G,w}
$$
for any norm $p$ on $V^\infty$. Finally, we let $p_{\max}$ and $p_{\min}$ be representatives of the maximal and minimal element in $\Noc(V,w)$.

Let $p\in \No(V,w)$ be a $G$-continuous norm of growth class $w$. We warn the reader that one cannot expect that $p_k^{\rm st}$ or $p_s$ is of growth class $w$ as well. To explain that we define the standard norm on $G$ by 
$$\|g\|=\max\{\|\Ad(g)\|, \|\Ad(g^{-1})\|\}\qquad(g\in G). $$
Note that $\|\cdot \|$ is a weight and define for any $t\geq 0$ a new weight 
\begin{equation} \label{def w_t} w_t(g):=w(g) \|g\|^t\qquad (g \in G).
\end{equation}
From the definition of the standard Sobolev norms we obtain that 
$p_k^{\rm st}\in \No(V,w_k)$. This is a uniform upper estimate. Its optimal growth depends on the
differentiated directions that act nontrivially on the representation. For instance a factor acting
trivially contributes no growth.

\subsection{Definition of the Sobolev gap}
Given $p\in \No(V,w)$ we recall that the topology of $V_p^\infty=V^\infty$ can be either defined with the standard Sobolev norms $(p_k^{\rm st})_{k\in \N_0}$ or via the $K$-Laplace Sobolev norms $(p_s)_{s\in\R}$, see Proposition \ref{Prop s-topology}.
The Casselman-Wallach Theorem \ref{Thm Casselman-Wallach} thus implies there exists $s\geq 0$ and $k\in \N_0$ such that
\begin{equation} \label{minmax domination 1}
p_{\max}\lesssim p_{{\min}, s}\quad \hbox{and} \quad p_{\max} \lesssim p_{\min, k}^{\rm st}.
\end{equation} 
Recall that we cannot expect that the right hand sides in the above dominations are contained in $\No(V,w)$. However, with $p_{\min, s}^G:= (p_{\min, s})^G$ and $p_{\min, k}^{{\rm st}, G}:=(p_{\min,k}^{\rm st})^G$ it follows from Lemma \ref{lemma Sobolev domination with and without infimum construction} that 
$p_{\min, s}^G, p_{\min, k}^{\rm{st},G}\in \No(V,w)$ and, moreover, that 
\eqref{minmax domination 1} is equivalent to 

\begin{equation} \label{minmax domination 2}
p_{\max}\lesssim p_{{\min}, s}^G\quad \hbox{and} \quad p_{\max} \lesssim p_{\min, k}^{{\rm st},G}.
\end{equation} 

We arrive at the key concept of this paper. Define the {\it Sobolev $w$-gap} of $V$ to be the non-negative number

\begin{align*}
\gap(V,w):=&\inf\{ s\geq 0\mid p_{\max} \lesssim p_{{\min},s}^{G}\}\\
=&\inf\{ s\geq 0\mid p_{\max} \lesssim p_{{\min},s}\}.
\end{align*}

We set $\gap(V,w)=0$ when $\No(V,w)=\varnothing$.
We show in Remark \ref{Rem pseudometric} below that $\gap(V,w)$ is independent of the choice of $K$. In case $w=\1$ we abbreviate the notation and write $\gap(V):=\gap(V,\1)$.

\begin{rmk} One could also define {\it standard Sobolev $w$-gap} of $V$ by 
\begin{align*}
\gap^{\rm st}(V,w):=&\min \{ k\in\N_0\mid p_{\max} \lesssim p_{{\min},k}^{{\rm st},G}\}\\
=&\min \{ k\in\N_0\mid p_{\max} \lesssim p_{{\min},k}^{\rm st}\}.
\end{align*}

The integer-valued invariant does not record fractional orders. In particular the value of the real-parameter gap alone does not decide whether domination holds at an integer endpoint.
The language developed in this section for $\gap(V,w)$ adapts to the standard Sobolev gaps $\gap^{\rm st}(V,w)$ in a straightforward manner.
\end{rmk}

\subsection{Basic properties of the Sobolev gap}
In this subsection 
we investigate the Sobolev gap with regard to submodules, duality, and monotonicity in the weight.

\begin{prop}\label{prop Duality and Monotonicity}
\, \begin{enumerate}[(i)]
\item\label{prop Duality and Monotonicity - item Duality} {\rm Duality:} $\gap(V, w) = \gap(\tilde V, w^\sharp)$. In particular,
$\gap(V, w)=\gap(V, w^\sharp)$ if $V$ is self-dual, i.e., $V\simeq \tilde V$.
\item\label{prop Duality and Monotonicity - Item Monotonicity}{\em Monotonicity:} Let $w_1, w_2$ be two weights. If $w_1\leq w_2$, then $\gap(V,w_1)\leq \gap(V, w_2)$.
\end{enumerate}
\end{prop}

\begin{proof} Choose representatives $m,M$ for the minimum and maximum elements in $\No(V,w)$, respectively, and $m',M'$ for those in $\No(\tilde V,w^\sharp)$. Restricted duality interchanges the extrema:
$m'\sim\tilde M$ and $M'\sim\tilde m$. Suppose $M\lesssim m_s$,
where $s\geq0$. Taking restricted duals reverses this comparison:
\[
 (M')_{-s}\sim(\tilde m)_{-s}
       =\widetilde{m_s}\lesssim\tilde M\sim m'.
\]
Here $\widetilde{m_s}=(\tilde m)_{-s}$ follows by changing variables
in the defining supremum: $D^s=(1+\Delta_K)^{s/2}$ is self-transpose
for the pairing and bijective on the algebraic $K$-finite modules.
No $G$-continuity of the shifted norm is required. Evaluating at
$D^su$ gives $M'(u)\lesssim (m')_s(u)$. Thus every admissible order
for $(V,w)$ is admissible for $(\tilde V,w^\sharp)$. Apply the same
argument to the dual pair, using restricted double duality and
$(w^\sharp)^\sharp=w$, and take infima.

We move on to prove monotonicity. We may assume that $\operatorname{Norm}(V,w_1)\neq \varnothing$ by the zero-gap convention. Let $p_{\min}^{1}$ and $p_{\min}^{2}$ be representatives of the minimal elements in $\Noc(V,w_{1})$ and $\Noc(V,w_{2})$, respectively. Likewise, let $p_{\max}^{1}$ and $p_{\max}^{2}$ be the maximal elements. Since $w_{1}\leq w_{2}$ we have $\No(V,w_{1})\subseteq \No(V,w_{2})$. It follows that for every $s\in \R$ with $p_{\max}^{2}\lesssim p_{\min,s}^{2}$ we have
$$
p_{\max}^{1}
\lesssim p_{\max}^{2}
\lesssim p_{\min,s}^{2}
\lesssim p_{\min,s}^{1}\,.
$$
For the last comparison we have used (\ref{pq comparison}).
It thus follows that $\gap(V,w_{1})\leq \gap(V, w_{2})$.
\end{proof}

The following is a partial result relating Sobolev gaps for a Harish-Chandra module and a submodule.

\begin{lemma} \label{prop embed}
Let $U\subset V$ be a submodule of the Harish-Chandra module $V$.
Let $p_{\min}$ and $q_{\min}$ be representatives for the minimal elements in $\Noc(V,w)$ and $\Noc(U,w)$, respectively.
Then
$$
[q_{\min}]
= [p_{\min}|_{U}].
$$
If there exists a norm $p\in\No(V,w)$ such that $q_{\max}\lesssim p|_U$, then
\begin{equation} \label{submodule comparison}
\gap(U,w)\leq \gap(V,w)
\end{equation}
In particular, the inequalities above hold in case $U\subset V$ is a direct summand.
\end{lemma}

\begin{proof}
Dual to the inclusion $U\hookrightarrow V$ is the projection
$\pi:\tilde V \twoheadrightarrow \tilde U$. If $C_V=\{\tilde{v}_{1}, \dots, \tilde{v}_{n}\}$ is a set of generators of $\tilde{V}$, then $C_U:=\{\pi(\tilde{v}_{1}),\dots,\pi(\tilde{v}_{n})\} $ is a set of generators for $\tilde{U}$. Let $p_{\min}$, resp. $q_{\min}$ be representatives of the minimal elements in $\Noc(V,w)$, resp. $\Noc(U,w)$ given by (\ref{def pmin}) for $C$ given by $C_V$ and $C_U$, respectively.

Thus $q_{\min}=p_{\min}|_U$. By maximality, the hypothesis gives
$q_{\max}\lesssim p|_U\lesssim p_{\max}|_U$. For every admissible order $s$ on $V$, restrict
$p_{\max}\lesssim p_{\min,s}$ to $U$. Since $D^s$ preserves $U$, the restricted right-hand side is
$q_{\min,s}$. Taking the infimum proves \eqref{submodule comparison}.

If $V=U\oplus W$ for some submodule $W$ of $V$ and $r\in \No(W,w)$, then the norm
$$
p:V=U\oplus W\to\R_{\geq0},\quad u+w\mapsto q_{\max}(u)+r(w)
$$
is contained in $\No(V,w)$ and satisfies $p|_U=q_{\max}$.
Therefore, \eqref{submodule comparison} holds in this case.
\end{proof}

\begin{rmk}\label{rmk:finite-direct-sums}
For a finite direct sum $V=\bigoplus_a V_a$ of unitarizable modules,
\[
                         \gap(V)=\max_a\gap(V_a).
\]
Indeed, extend dual generators of each summand by zero on the others.
Formula \eqref{def pmin} then gives the maximum of the minimal norms,
and the dual construction gives the sum of the maximal norms. Since
the Sobolev shifts preserve the summands, every order admissible on
$V$ is admissible on each $V_a$. Conversely, any order larger than
the largest summand gap works on $V$, by summing the finitely many
comparisons.
\end{rmk}

\subsection{\texorpdfstring{$\Noc[V]$ as a pseudometric space}{Norm(V,w) as a pseudometric space}}
Recall the filtered space 
$\Noc[V]=\bigcup_{w}\Noc(V,w)$ of all equivalence classes of $G$-continuous norms. 

Given $[p], [q] \in \Noc[V]$ we define the finite quantity
$$
d_\to ([p], [q]):=\inf\{ s\geq 0\mid p \lesssim q_s\}. 
$$

\begin{definition} Define a pseudometric $d$ on $\Noc[V]$ by
$$
d([p], [q])
=\max\{ d_\to ([p], [q]), d_\to ([q], [p])\}.
$$
\end{definition} 
If $p\lesssim q_a$ and $q\lesssim r_b$, substituting $D^av$
in the second inequality gives $q_a\lesssim r_{a+b}$.
Taking infima gives the directed triangle inequality.
Taking the maximum of the two orientations gives the triangle
inequality for $d$.
Note that $\gap(V,w)=d([p_{\min}],[p_{\max}])$ for the minimal and maximal elements $[p_{\min}]$ and $[p_{\max}]$ in $\Noc(V,w)$. In particular, $\gap(V,w)$ is the diameter of the subspace $\Noc(V,w)\subset \Noc[V]$. 

As an immediate consequence of \eqref{s plus t} and 
\eqref{ps dual} we record duality of the distance: 

\begin{lemma} \label{Lemma duality} For any two elements $[p],[q]\in \Noc[V]$ we have:
$$d([p],[q]) = d([\tilde p],[\tilde q]).$$
\end{lemma}

\begin{rmk}\label{Rem pseudometric}
The pseudometric $d$ and the Sobolev gap $\gap(V,w)$ are independent of the choice of the maximal compact subgroup $K$. In fact, any other maximal compact subgroup is given by a conjugate $K_g:=gKg^{-1}$ for some $g\in G$. Instead of the $(\gf, K)$-module $V$ we then consider the space $V_g$ of $K_{g}$-finite vectors in $V^{\infty}$. Clearly $V_{g}=g\cdot V\subset V^\infty$. Now suppose that $p$ and $q$ are $G$-continuous norms and $p\lesssim q_s$. Let $q_{s,g}$ be the
$s$-th $K_g$-Sobolev norm on $V^\infty$. Then as norms on $V^\infty$
$$p \sim p(g^{-1}\dotvar)\lesssim q_s(g^{-1}\dotvar) \sim q_{s,g}\, .$$
\end{rmk}

\subsection{A stabilization property for norms}
Let $(s_{n})_{n\in\N}$ be a sequence in $\R_{\geq0}$ with $s_n\nearrow\infty $. Then by monotonicity of the $p_s$, see Proposition \ref{Prop monotonous} we have 
$$p_{s_1}\lesssim p_{s_2}\lesssim p_{s_3} \lesssim \ldots .$$
If $\dim V=\infty$ we note that the inequalities become strict provided $(s_n)_{n\in\N}$ is strictly increasing. Taking $G$-invariants we get 
$$p_{s_1}^G\lesssim p_{s_2}^G\lesssim p_{s_3}^G \lesssim \ldots .$$
But, if $N\in\N$ is such that $p_{\max} \lesssim p_{s_N}$, then we observe stabilization from $N$ on, i.e. 
$[p_{s_n}^G]=[p_{\max}]$ for all $n\geq N$. 
More precisely, we have the following proposition.

\begin{prop}\label{Prop stabilization 1} Let $p\in \No(V,w)$.
Then 
$$
[p_{\max}]
= [p_s^{G}] \qquad \big(s > d([p], [p_{\max}])\big).
$$
\end{prop}
\begin{proof} Since $(p_s)_{s\in\R}$ is monotone by Proposition \ref{Prop monotonous} we have for any $s>d([p],[p_{\max}])$ that
$p_{\max}\lesssim p_s$ and hence $[p_{\max}]=[p_{s}^{G}]$. 
\end{proof}

\subsection{Isometric norms}
We now consider the case where $w=\1$. For 
$p\in \No(V)$ and $p_{\rm iso}:=\sup_{g\in G} p(g\dotvar)$ we observe that $p\leq p_{\rm iso}\leq C p$. Hence $[p]=[p_{\rm iso}]$ and we can and will assume henceforth that all norms considered are isometric.

As before, we write $p_{\min}$ and $p_{\max}$ for isometric representatives of the minimal and maximal elements in $\Noc(V)$, respectively.

\begin{prop}\label{prop unitary}
Assume that $V$ is unitarizable and let $q$ be a unitary norm.
Then, in the pseudometric space $(\Noc(V), d)$ we have
$$
d([q], [p_{\min}])
= d([q], [p_{\max}])
$$
and in particular
$$
\gap(V)=\gap(V,\1)
= 2 d([q], [p_{\max}]).
$$
\end{prop}

\begin{proof}
Let us denote by $\oline V$ the complex conjugate vector space, i.e., $\oline V=V$ as real vector space, but with opposite complex structure. Clearly $\oline V$ is a Harish-Chandra module as well and we denote for any 
$G$-continuous norm $p$ the corresponding norm on $\oline V$ by $\bar p$. 

As $V$ is unitarizable, the Hermitian form defining $q$ $\la \dotvar, \dotvar\ra$ identifies $\oline V$ with $\tilde V$ via 
$$\oline V\to \tilde V, \qquad v \mapsto \tilde v = \la\dotvar, v\ra. $$
Let $p_{\min}$ be a minimal norm on $V$. Then it is clear that $\bar p_{\min}:=\oline{p_{\min}}$ defines a minimal norm on $\oline V$ and 
$$d([p_{\min}], [q])= d([\bar p_{\min}], [\bar q]).$$
As $\tilde{\bar q}=q$ and $\tilde{\bar p}_{\min}=:p_{\max}$ defines a maximal norm on $V$, equality of the two distances follows from duality in Lemma \ref{Lemma duality}. 

Write $d_0$ for their common value. The triangle inequality gives $\gap(V)\leq2d_0$; we now prove $2d_0\leq\gap(V)$. Use the identification afforded by $q$ and take $p_{\max}=\tilde{\bar p}_{\min}$. For $s>\gap(V)$ we have $p_{\max}\lesssim p_{\min,s}$. Since $D^{s/2}$ is positive and self-adjoint for $q$, duality gives, for $v\in V$,
\[
 q(v)^2=|\langle D^{s/2}v,D^{-s/2}v\rangle|
 \leq p_{\min}(D^{s/2}v)p_{\max}(D^{-s/2}v)
 \lesssim p_{\min,s/2}(v)^2 .
\]
For the last inequality, apply $p_{\max}\lesssim p_{\min,s}$ to
$D^{-s/2}v$. By the definition of the Sobolev norm and the identity
$D^sD^{-s/2}=D^{s/2}$, this gives
\[
  p_{\max}(D^{-s/2}v)
  \lesssim p_{\min,s}(D^{-s/2}v)
  =p_{\min}(D^{s/2}v)=p_{\min,s/2}(v).
\]
The first factor in the estimate for $q(v)^2$ is also $p_{\min,s/2}(v)$.
As $p_{\min}\lesssim q$, this gives $d([q],[p_{\min}])\leq s/2$. Let $s\downarrow\gap(V)$ and use the equality of the two distances and the triangle inequality.
\end{proof}

\begin{rmk}\label{rmk unitary in the middle}\,
By Propositions~\ref{prop unitary} and~\ref{Prop stabilization 1},
$[p_{\max}]=[q_s^G]$ for every $s>\gap(V)/2$. For these same orders
we also have $[p_{\min}]=[{}^Gq_{-s}]$, where
$$
{}^Gq_{-s}(v)
:=\sup_{g \in G} q_{-s}(g\cdot v)
\qquad(v\in V).
$$
Indeed, let $p_{\min}$ be an isometric representative of $[p_{\min}]$.
The midpoint comparison gives $q_{-s}\lesssim p_{\min}$, also on the common
smooth globalization. Applying it to $gv$ and taking the supremum
shows ${}^Gq_{-s}\lesssim p_{\min}$. Conversely,
$q_{-s}\leq{}^Gq_{-s}\leq q$. The middle norm is isometric; its
orbit maps are continuous by domination by $q$, and density and the
finite-dimensional isotypic projections identify the $K$-finite
vectors of its completion with $V$. It is therefore a globalization
norm, so minimality gives $p_{\min}\lesssim{}^Gq_{-s}$. Hence
$$
\gap(V)
= d([{}^Gq_{-s}], [q_s^G])
\qquad\big(s>\gap(V)/2\big).
$$
\end{rmk}

\section{Uniform finiteness results for the Sobolev gap}\label{Section 5: Uniform finiteness}

We begin with a uniform upper bound for the Sobolev gap of the
discrete series at the constant weight $\1$. We then turn to a
bound in the other direction: for a connected semisimple group
with finite center, 
the gap at weight $\1$ of every nonzero finite-length unitarizable
Harish-Chandra module is at least its Gelfand--Kirillov dimension.

For the minimal principal series $V_{\sigma,\lambda}$ we obtain
an explicit upper bound at the weight $w_{\re\lambda}$, uniform in
$\sigma$ and growing at most linearly in $\|\re\lambda\|$. We also
obtain uniform finiteness for each fixed weight among the parameters
for which that weight is admissible. When $\lambda$ is imaginary, the bound $2c_\gf$ is attained for every induction when $G$ is connected semisimple with finite
center. We conclude with a fixed-weight result for the irreducible Harish-Chandra modules of $\SL(2,\R)$.

\subsection{Discrete series}

\begin{theorem} \label{thm disc series}
There exists a constant $C>0$ such that $\gap(V)\leq C$ for all Harish-Chandra modules of the discrete series of $G$.
\end{theorem}

For the proof we may assume that $G$ is connected semisimple with
finite center. Indeed, on restriction to $G_e$ it suffices to treat
the irreducible constituents separately
(Remark~\ref{rmk:finite-direct-sums}), and the connected center acts
on each by a unitary character, which does not affect the Sobolev gap. The proof rests on the group case of the wavefront estimate \cite[(1.3)]{KSS} which we record in the Lemma below.

\begin{lemma}\label{KSS-lemma}
On $Z=(G\times G)/\diag G$, fix $1\le p'<p<\infty$ and a joint infinitesimal character $\chi$ for
$G\times G$. For each smooth $L^{p'}$ joint eigenfunction $f$ with character $\chi$ one has
\[
 \|f\|_{L^p(Z)}\le C_{\chi,p',p}\|f\|_{L^\infty(Z),k},
\]
where the Sobolev order can be chosen independently of the eigencharacter for the fixed group and exponents.
\end{lemma}

\begin{proof}[Proof of Theorem \ref{thm disc series}]
As already mentioned we apply Lemma \ref{KSS-lemma} to the group $G$ considered as the homogeneous space $Z=G \times G/ \diag(G)$.
Let $V$ be an irreducible Harish-Chandra module of the discrete series for $G$. Let $\chi$ be the
infinitesimal character of $V$. Then for every $\tilde{v}\in\tilde{V}$ and $v\in V$ the matrix
coefficient $f=m_{v, \tilde v}$ is a joint eigenfunction for $G\times G$, with the two
infinitesimal characters of $\tilde V$ and $V$ under left and right translations, respectively.

Let $q$ be a unitary norm on $V$ and recall the orthogonality relations
$$
\|m_{v, \tilde v}\|_{L^2(G)}^2
= \frac{1}{d_V} q(v)^2 \tilde q(\tilde v)^2
\qquad (v \in V, \tilde v\in \tilde V)
$$
with $d_V>0$ the formal degree of $V$.
We need one $p'<2$ for all discrete series of the fixed group, not
a uniform bound for their $L^{p'}$-norms. If there are no noncompact
factors the assertion is immediate. 
Otherwise Mili\v{c}i\a'{c}'s
criterion \cite[pp.~59--60]{Milicic}, for a compact Cartan and
discrete series parameter $\Lambda$, gives $L^{p'}$-integrability
of the $K$-finite coefficients when
\[
 |(\Lambda,\alpha)|>(2/p'-1)k(\alpha)
 \quad\text{for every noncompact root }\alpha,\qquad
 k(\alpha)=\tfrac14\sum_{\beta\in\Phi}|(\alpha,\beta)|.
\]
For a fixed root, the coroot evaluations of the parameters lie in a
translate of a discrete subgroup of $\R$: on the fixed finite cover,
a positive integral multiple of the coroot is a cocharacter.
Regularity excludes zero. Since there are finitely many roots
(and only finitely many choices of positive system), the nonzero
ratios $|(\Lambda,\alpha)|/k(\alpha)$ have a common positive lower
bound $b$. Choose $1<p'<2$ with $2/p'-1<b$ and put $\e=2-p'$.
The group case of the wavefront estimate
\cite[(1.3), \S7.2, especially (7.17)--(7.20)]{KSS} is used with
these fixed exponents: the Sobolev order is $k=k(G,p',2)$, whereas
the comparison constant may depend on the joint infinitesimal
character. Only the order enters the bound for the gap.
Applying Lemma~\ref{KSS-lemma} with $p'=2-\e$ and $p=2$ and using the
orthogonality relation in its square-root form gives
\[
 d_V^{-1/2}q(v)\tilde q(\tilde v)
   =\|m_{v,\tilde v}\|_{L^2(G)}
   \leq C\|m_{v,\tilde v}\|_{L^\infty(G),k}.
\]
Fix $\tilde v\ne0$. Right and left derivatives of order at most
$k$ are coefficients with derivatives of $v$ and $\tilde v$,
respectively. Enlarge the finite dual test family to include the
derivatives of $\tilde v$ which occur. Since $V$ is irreducible
and the family contains $\tilde v$, it still gives a representative
of the minimal norm. Hence
\[
                         q\lesssim p_{\min,k}^{\rm st},
\]
with $k$ independent of $V$; the domination constant need not be.

Take $s=2\lceil k/2\rceil+m$, where $m$ is the even integer of
Proposition~\ref{Prop equivalence of Sobolev norms}. By that
proposition and monotonicity of the standard integer-order norms,
\[
 q\lesssim p_{\min,k}^{\rm st}\lesssim (p_{\min})_s^\Delta .
\]
Use one and the same Laplace element and choose $R$ sufficiently large
for both norms. Substitution of $(R^2+\Delta)^{-s/2}v$ in this
comparison gives $q_{-s}^{\Delta}(v)\lesssim p_{\min}(v)$.
The choice of $R$ and the comparison constants may depend on the
module, but the even integer $s$ does not.
Since $q$ is $K$-Hermitian, $q^{\Delta}_{-s}$ is equivalent to $q_{-s}$ by Lemma \ref{Lemma equivalence of sobolev norms for K-hermitian norms}.
Interchanging $V$ with $\tilde V$ we obtain dually $p_{\max}\lesssim q_s$
and thus
$$
\gap(V)
\leq 2s
$$
for all Harish-Chandra modules of the discrete series $V$. This proves Theorem \ref{thm disc series}.
\end{proof}

\subsection{A lower bound for unitarizable representations}
\label{subsec:unitarizable-gk-gap}
We recall the definition of the Gelfand-Kirillov dimension of a Harish-Chandra module $V$.  Choose a finite-dimensional $K$-stable generating
subspace $F\subset V$. Write
\[
 F_n=\U_{\leq n}(\gf)F,\qquad N_n=\dim F_n.
\]
This is a good filtration, so $N_n\asymp n^d$ for some $d\geq 0$ independent of the choice of $F$. The Gelfand-Kirillov dimension of $V$ is defined by 
$$\operatorname{GKdim}V=d.$$

\begin{theorem}\label{thm:unitarizable-gk-gap}
Let $G$ be a connected semisimple group with finite center.
For every unitarizable Harish-Chandra module $V$
of $G$, at the constant weight $\1$ one has
\[
                      \operatorname{GKdim}V\leq\gap(V).
\]
\end{theorem}
It suffices to prove the theorem for irreducible $V$. Indeed, a
unitarizable Harish-Chandra module is a finite direct sum of irreducible
modules, and both the Gelfand--Kirillov dimension and the Sobolev gap
take the maximum over the summands. For the former this follows by
taking direct sums of good filtrations, and for the latter see
Remark~\ref{rmk:finite-direct-sums}.

Henceforth let $V$ be irreducible, let $q$ be a unitary norm on $V$,
and write $d=\operatorname{GKdim}V$. We may assume $d>0$, since
otherwise there is nothing to prove. Choose a $K$-finite unit vector
$w\in V$. Under the unitary pairing it represents a nonzero vector
of the irreducible contragredient module and hence a generator.
By Lemma~\ref{lemma smooth minimal} we may therefore take
\[
   p_{\min}(v)=\sup_{g\in G}|\langle v,\pi(g)w\rangle|.
\]
The following observation tells us how small we need to make this
coefficient.

\begin{lemma}[The Sobolev obstruction]\label{lem:gk-obstruction}
Suppose that for arbitrarily large $n$ there is a unit vector
$v_n\in F_n$ such that
\[
   p_{\min}(v_n)\lesssim n^{-d/2}\sqrt{\log n}.
\]
Then $\gap(V)\geq d$.
\end{lemma}

\begin{proof}
Put $D=(1+\Delta_K)^{1/2}$. The space $F_n$ is a $K$-equivariant
image of $\bigoplus_{j=0}^n\gf_\C^{\otimes j}\otimes F$. Its highest
weights thus have norm $O(n+1)$, and the eigenvalues of $D$ on $F_n$
are $O(n+1)$ by \eqref{eq D^s}. In particular, for $s\geq0$,
\[
   q(D^{-s}v_n)\gtrsim_s n^{-s}.
\]
If $q\lesssim p_{\min,s}$ for some $s<d/2$, $s\geq0$, applying this
domination to $D^{-s}v_n$ would give
\[
   n^{-s}\lesssim_s q(D^{-s}v_n)
   \lesssim p_{\min}(v_n)
   \lesssim n^{-d/2}\sqrt{\log n},
\]
which is impossible for arbitrarily large $n$. Thus
$d([q],[p_{\min}])\geq d/2$, and the unitary midpoint formula
(Proposition~\ref{prop unitary}) gives $\gap(V)\geq d$.
\end{proof}

It remains to find a unit vector $v_n\in F_n$ whose coefficient
against $w$ is small on all of $G$. We do this in two steps.
Outside a Cartan ball of radius $C\log n$, coefficient decay gives
the required bound for \emph{every} unit vector in $F_n$. Inside
the ball, we use $\dim F_n\asymp n^d$ to \emph{choose} one unit
vector for which the bound holds throughout the ball.

We may divide out the connected noncompact simple factors acting
trivially on $V$: this changes neither the filtration nor the gap,
since these factors act trivially also in the Sobolev operator.
For $g=k\exp(H)k'$ with $H\in\overline{\af^+}$, put
$r(g)=\|H\|$ and define Cartan balls $\mathcal B(R)=\{g\in G:r(g)\leq R\}$ for $R>0$.

\begin{lemma}[Decay outside a Cartan ball]\label{lem:gk-exterior}
There is a constant $C>0$ such that, for $n\geq2$ and $v\in F_n$,
\[
   \sup_{r(g)>C\log n}|\langle v,\pi(g)w\rangle|
   \lesssim n^{-d/2-1}q(v).
\]
\end{lemma}

\begin{proof}
There is an integer $m\geq1$ for which $\pi^{\otimes m}$ is
tempered. Indeed, on passing to the finite covering product of the
simple factors, $\pi$ is a tensor product and each noncompact factor
acts nontrivially. By Cowling's finite-integrability theorem
\cite{CowlingCoeff103}, the $K$-finite coefficients on each such
factor belong to some $L^{p_i}$. Since the coefficients are bounded,
we may choose a common finite even exponent $2m\geq p_i$ for all
factors. Taking products and finite sums gives $L^{2m}$-integrability
on $G$; compact factors cause no restriction. By H\"older's
inequality the coefficients of $\pi^{\otimes m}$ on the dense span
of elementary $K$-finite tensors lie in $L^2(G)$. Temperedness now
follows from \cite[Theorem~1]{CHH92}.

Write $r_K(u)=\dim\Span(Ku)$. Applying \cite[Theorem~2]{CHH92} to
$v^{\otimes m}$ and $w^{\otimes m}$ and taking $m$-th roots gives
\[
   |\langle v,\pi(g)w\rangle|
   \leq \sqrt{r_K(v)r_K(w)}\,q(v)q(w)\,\Xi(g)^{1/m},
\]
since $r_K(u^{\otimes m})\leq r_K(u)^m$. Here $\Xi$ is
Harish-Chandra's elementary spherical function at parameter zero.
As $F_n$ is $K$-stable, $r_K(v)\leq N_n\ll n^d$. Moreover,
on the closed positive chamber,
\[
   \Xi(\exp H)\lesssim (1+\|H\|)^a e^{-\rho(H)},
   \qquad \rho(H)\geq c\|H\|
\]
for some $a,c>0$. With $w$ fixed, the coefficient is therefore
$O(n^{d/2}e^{-b r(g)}q(v))$ for some $b>0$. Taking $bC\geq d+1$
proves the assertion.
\end{proof}

\begin{lemma}[Small coefficients on a Cartan ball]\label{lem:gk-interior}
For every fixed $C>0$ and all sufficiently large $n$, there is a
unit vector $v_n\in F_n$ such that
\[
   \sup_{r(g)\leq C\log n}|\langle v_n,\pi(g)w\rangle|
   \lesssim_C N_n^{-1/2}\sqrt{\log n}.
\]
\end{lemma}

\begin{proof}
Give $G$ a left-invariant Riemannian metric. The orbit map
$g\mapsto\pi(g)w$ is Lipschitz: its derivative along $g\exp(tX)$
has norm $q(d\pi(X)w)$, independent of $g$. Thus it suffices to
control the coefficient on a sufficiently fine net.

Put $\delta=n^{-d/2-1}$ and choose a maximal $\delta$-separated
subset $\{g_1,\ldots,g_J\}$ of $\mathcal B(C\log n)$. This is a
$\delta$-net, and $J$ grows at most polynomially in $n$. Indeed,
the Cartan decomposition places the net in a metric ball of radius
$O(1+\log n)$. By at most exponential volume growth, the volume of
this ball enlarged by one is polynomially bounded in $n$, whereas
each of the disjoint balls of radius $\delta/2$ about the net points
has volume comparable to $\delta^{\dim G}$.

Fix an orthonormal basis $e_1,\ldots,e_{N_n}$ of $F_n$ and put
\[
   v=N_n^{-1/2}\sum_{k=1}^{N_n}\epsilon_k e_k,
\]
where the $\epsilon_k$ are independent random signs. Every choice
gives a unit vector. For each $j$, Bessel's inequality gives
\[
   \sum_{k=1}^{N_n}
      |\langle e_k,\pi(g_j)w\rangle|^2\leq q(w)^2=1.
\]

Applying Hoeffding's inequality to the real and imaginary parts
and summing over the net points, we obtain
\[
   \mathbb P\!\left\{
      \max_{1\leq j\leq J}|\langle v,\pi(g_j)w\rangle|\geq t
   \right\}
   \leq 4J e^{-N_n t^2/4}.
\]
For $t=2\sqrt{\log(8J)/N_n}$ the right-hand side is $1/2$.
Thus some choice $v_n$ has all its coefficients on the net bounded
by $t\lesssim_C N_n^{-1/2}\sqrt{\log n}$. Passing from the net to
the ball costs only $O(\delta)$ by the Lipschitz estimate.
\end{proof}

\begin{proof}[Conclusion of the proof of
Theorem~\ref{thm:unitarizable-gk-gap}]
Take $C$ from Lemma~\ref{lem:gk-exterior} and $v_n$ from
Lemma~\ref{lem:gk-interior} for this $C$. Since $N_n\asymp n^d$,
the coefficient of $v_n$ is $O(n^{-d/2}\sqrt{\log n})$ inside
$\mathcal B(C\log n)$ and $O(n^{-d/2-1})$ outside. Taking the
supremum over $G$ gives the hypothesis of
Lemma~\ref{lem:gk-obstruction}, and the theorem follows.
\end{proof}

\subsection{Minimal principal series}\label{subsection minimal principal series}

For $\sigma\in \hat M$ we denote by $U_\sigma$ a model Hilbert space on which $M$ acts. For $\lambda\in \af_\C^*$ we shall define Harish-Chandra modules
$V_{\sigma, \lambda}$ of the (minimal) principal series which feature natural Hilbert globalizations $\Hc_{\sigma,\lambda}$. For that let us define a finite-dimensional representation of the minimal parabolic $P$ on $U_\sigma$ by
$$
\sigma_\lambda(m a n)
= \sigma(m) a^{\lambda +\rho}
\qquad (man\in P = MAN)
$$
and define the smooth principal series by
$$
V_{\sigma, \lambda}^\infty
= \{ f \in C^\infty (G, U_\sigma)\mid f(gp)
= \sigma_\lambda(p)^{-1} f (g) \text{ for all }g\in G, p\in P\}.
$$
The left regular action of $G$ on $V_{\sigma, \lambda}^\infty$ is a smooth Fr\'echet representation of moderate growth and we denote by $V_{\sigma, \lambda}$ its Harish-Chandra module of $K$-finite vectors. By restricting to $K$ we obtain an isomorphism of $K$-modules
$$
V_{\sigma, \lambda}^\infty
\simeq C^\infty(K,\sigma),
$$
where
$$
C^\infty(K,\sigma)
:=\{ f \in C^\infty(K, U_\sigma) \mid f(km)= \sigma_\lambda(m)^{-1} f (k)\text{ for all }k\in K, m\in M\}.
$$
Notice that the $G$-action in the compact picture is given by 
\begin{equation}\label{eq:compact-Iwasawa-action}
  (\pi_{\sigma,\lambda}(g)f)(k)
    ={\bf a}(g^{-1}k)^{-(\lambda+\rho)}
       f\bigl({\bf k}(g^{-1}k)\bigr)
\end{equation}
for $f \in V_{\sigma,\lambda}^\infty=C^\infty(K,\sigma)$ and $g^{-1}k= {\bf k}(g^{-1}k){\bf a}(g^{-1}k){\bf n}(g^{-1}k)$ according to the Iwasawa decomposition $G=KAN$.

Integration over $K$ yields the non-degenerate $G$-invariant pairing
$$
V_{\sigma, \lambda}^\infty \times V_{\sigma^*, -\lambda}^\infty \to \C,
\quad(f_1, f_2)\mapsto \int_K \la f_1(k), f_2(k)\ra_\sigma \,dk,
$$
with $\la\cdot, \cdot\ra_\sigma$ denoting the natural pairing of $U_\sigma$ and its dual $U_{\sigma^*} = U_\sigma^*$. In particular, we obtain the natural isomorphism of Harish-Chandra modules:
$$
\tilde V_{\sigma, \lambda}
\simeq V_{\sigma^*, - \lambda}\,.
$$

We define
$$
\Hc_{\sigma,\lambda}:=L^2(K,\sigma)
=\{f\in L^2(K,U_\sigma)\mid f(km)=\sigma(m)^{-1}f(k)\}
$$
where the covariance condition holds for every $m\in M$ and almost every $k\in K$. We recall
that the $G$-action on $V_{\sigma,\lambda}$ extends to a Hilbert
representation on $\Hc_{\sigma, \lambda}$ which is unitary for this compact-picture inner product if and only if $\lambda\in i\af^*$. We denote the $K$-Hermitian norm on $\Hc_{\sigma, \lambda}$ by $q$. More precisely, for $f \in \Hc_{\sigma,\lambda}$:
$$
q(f)^2
=\int_{K} \|f(k)\|_{\sigma}^2 \,dk.
$$

For $\mu\in \af^*$ we recall the exponential weights from \eqref{standard weight exp}
$$
w_\mu(g)
= \max_{w\in W} \{ a^{w\cdot \mu}, a^{-w\cdot \mu}\}
\qquad (g\in KaK, a\in \overline{A^{+}}).
$$
Note that $w_\mu = w_\mu^\sharp = w_{-\mu}$.
It is a straightforward computation to see that $q\in \No(V_{\sigma, \lambda}, w_{\re \lambda}) \neq \varnothing$. 

To formulate our main result concerning minimal principal series we introduce a structural constant $c_\gf$ as follows.
Choose $X_0$ in $\af\cap[\gf,\gf]$ by the conditions
$\alpha(X_0)=1$ for all simple restricted roots (take $X_0=0$
when there are none), and set
\begin{equation} \label{def str const}
c_\gf
:=\rho(X_0)
\in \tfrac{1}{2}\N_0.
\end{equation}
This constant can be computed out of the structural data of $\gf$. If $\Pi$ is the set of
simple restricted roots and $\alpha=\sum_{\beta\in\Pi}n_\beta\beta$
is a positive restricted root, its height is
\[
 \operatorname{ht}(\alpha)=\sum_{\beta\in\Pi}n_\beta=\alpha(X_0),
\]
where $\beta(X_0)=1$ for each $\beta\in\Pi$ as in \eqref{def str const}. In particular, with $m_\alpha=\dim\gf^\alpha$ the dimension of the root space

\begin{equation}\label{eq: def cg} c_\gf = \frac12 \sum_{\alpha>0} m_\alpha\operatorname{ht}(\alpha).\end{equation} 

\begin{rmk}\label{rmk constant} We compute the constant $c_\gf$ in a few important cases. For example, if $\gf$ has real rank one, then $\{\alpha\} \subset \Sigma^+\subset \{\alpha, 2\alpha\}$ and 
$$
c_\gf
=\frac{1}{2}(m_\alpha + 2m_{2\alpha}).
$$
In particular,
\begin{enumerate}[(a)]
    \item $c_\gf=\frac{n-1}{2}$ for $\gf=\so(1,n)$, $n\geq 2$.
    \item $c_\gf=n$ for $\gf=\su(1,n)$, $n\geq 2$,
    \item $c_\gf=2n+1$ for $\gf=\sp(1,n)$, $n\geq 2$,
    \item $c_\gf=11$ for $\gf=\mathfrak{f}_{4(-20)}$.
\end{enumerate}

A look in the tables of \cite{Bou} gives us $2\rho$ in the split cases as an explicit sum of simple roots. For example, if $\gf=\sl(n,\R)$ we obtain that 
$$
c_\gf
=\rho(X_0)
=\frac{1}{2} \sum_{j=1}^{n-1} j(n-j)=\frac{n(n^2-1)}{12}\, .
$$
\end{rmk}

\begin{theorem} \label{thm sup}

There is a constant $c>0$, depending only on the structural data and
the fixed norm on $\af^*$, such that, for every $R\geq0$,
\begin{equation}\label{eq mps bound}
 \sup_{\substack{\sigma\in\hat M,\ \lambda\in\af_\C^*\\
                  \|\re\lambda\|\leq R}}
 \gap(V_{\sigma,\lambda},w_{\re\lambda})
 \leq 2c_\gf+2cR .
\end{equation}
Moreover, $[p_{\max}]=[(q_s)^{G,w_{\re\lambda}}]$ for every
$s>c_\gf+cR$. In particular,
\[
 \gap(V_{\sigma,\lambda})\leq 2c_\gf
 \qquad(\lambda\in i\af^*).
\]
\end{theorem}

The proof of the theorem relies on lower bounds for matrix coefficients from \cite[Sect. 12]{BK} which we adapt now for our purpose. To start with we need an appropriate model to compute the minimal norm. For that in turn it is suitable to employ the non-compact model for principal series representations: We write $\Nc_{\sigma, \lambda}^\infty \subset C^\infty (\oline N, U_\sigma)$ for the restrictions of the elements in $V_{\sigma, \lambda}^\infty \subset C^\infty(G, U_\sigma)$ to $\oline N$. Integration over $\oline N $ yields a non-degenerate $G$-invariant pairing
$$
\Nc_{\sigma, \lambda}^\infty \times \Nc_{\sigma^*, -\lambda}^\infty \to \C,
\quad (\phi_1, \phi_2)\mapsto \la \phi_1, \phi_2\ra
:=\int_{\oline N} \la \phi_1(\oline n ), \phi_2(\oline n )\ra_\sigma \, d\oline n.
$$
In order to define the minimal norm, we introduce a suitable compact set 
$C_\sigma \subset \Nc_{\sigma^*, -\lambda}^\infty$
and invoke Lemma \ref{lemma smooth minimal}. Using the exponential function we identify $\oline N$ with $\oline \nf$ and let $\phi\in C^\infty(\oline \nf)$ be a non-negative function supported in the unit ball $B\subset \oline \nf$ with 
$\int \phi=1$. For $u^*\in U_{\sigma^*}$ we set $\phi_{u^*}(\oline n)=\phi(\oline n)u^*$ and note that $\phi_{u^*} \in \Nc_{\sigma^*, -\lambda}^\infty$ is a smooth vector. 
Let $B_{\sigma^*}$ be the unit ball in $U_\sigma^*$
and 
$$
C_\sigma
:=\{ \phi_{u^*}\mid u^*\in B_{\sigma^*}\}.
$$
This is certainly a compact set in $\Nc_{\sigma^*, -\lambda}^\infty$, and we need to see that
$G\cdot C_\sigma$ separates the original module. Indeed,
recall we choose $X_0\in \af^+$ such that $\alpha(X_0)=1$ for all simple roots $\alpha\in \Sigma^+$. We set 
$$
a_t
:=\exp(tX_0) \qquad (t\in \R).
$$
Let $\phi_{u^*, t}:= \pi_{\sigma^*, -\lambda}(a_t)\phi_{u^*}$ and note that 
$$
\phi_{u^*, t}(\oline n)
= a_t^{-\lambda +\rho} \phi_{u^*}(\Ad(a_t)^{-1} \oline n)\qquad (\oline n \in \oline N).
$$
Hence $\supp \phi_{u^*, t} \subset \Ad(a_t)B \to \{\1\}$ for $t\to \infty$ 
and Dirac approximation yields that 
\begin{equation} \label{Dirac1}
\lim_{t\to \infty} a_t^{\rho+\lambda} \la \phi_{u^*, t}, \psi\ra
= u^*(\psi(\1)) \qquad (\psi \in \Nc_{\sigma, \lambda}^\infty).\end{equation} 

For nonzero $\psi$ some $K$-translate has nonzero value at $\1$. Choosing $u^*$ nonzero on that value,
\eqref{Dirac1} shows that $G\cdot C_\sigma$ separates the module.

Hence a representative of the minimal element in $\Noc(V_{\sigma, \lambda},w_{\re \lambda})$, is a $G$-continuous norm for $V_{\sigma, \lambda}$ that can be defined as: 
\begin{equation} \label{def pmin mps}
p_{\min}(\psi)
=\sup_{u^*\in B_{\sigma^*}}
\sup_{g\in G} \frac{|\la \phi_{u^*}, g\cdot \psi\ra|}{w_{\re \lambda}(g)}
\qquad (\psi\in \Nc_{\sigma, \lambda}^{\infty}).
\end{equation}

The quantitative Dirac approximation of \cite[Sect.~12]{BK} gives
a lower bound for the minimal norm on each $K$-type. We apply the
same approximation to a whole $K$-spectral band before passing to a
Hilbert Sobolev norm.

We define an auxiliary norm on $V_{\sigma, \lambda}^\infty= C^\infty(K,\sigma)$ by 
$$
q^\infty(f)
=\sup_{k\in K} \|f(k)\|_\sigma
$$
and note that $q^\infty\geq q$. Note that we have a correspondence $C^\infty(K,\sigma)\leftrightarrow \Nc_{\sigma,\lambda}^\infty$ given by
\begin{equation}\label{eq switch picture}
\psi(\oline n) 
=f({\bf k}(\oline n)) {\bf a}(\oline n)^{-(\lambda +\rho)} \qquad (\oline n \in \oline N),
\end{equation}
where $\oline n= {\bf k}(\oline n){\bf a}(\oline n){\bf n}(\oline n)$ according to the Iwasawa decomposition $G=KAN$.
Using this correspondence we define $q^\infty(\psi):=q^\infty(f)$.

\begin{lemma}[Key Lemma]\label{Lemma Uniform lower bounds}
 Let $L\geq 1$ and let $E_L\subset V_{\sigma,\lambda} $ be the sum of the
eigenspaces of $D=(1+\Delta_K)^{1/2}$ with eigenvalue at most $L$.
Suppose that $\|\re\lambda\|\leq R$, and put $s=c_\gf+cR$.
For the representative \eqref{def pmin mps} one has
\begin{equation}\label{ulb1}
q^\infty(f)\leq C L^s p_{\min}(f)\qquad(f\in E_L).
\end{equation}
Here $c$ depends only on the structural data and the fixed norm on
$\af^*$, while $C$ is independent of $L$ and $\sigma$ and can be
chosen locally uniformly in $\lambda$.
\end{lemma}

\begin{proof}
We work first in the compact picture. If $X\in\kf$ is a unit vector,
the frequencies of $u\mapsto f(\exp(uX)k)$ lie in $[-L,L]$.
Indeed, on each unitary $K$-type occurring in $E_L$ this follows
by completing $X$ to an orthonormal basis in the sum of squares
defining $\Delta_K$. For a unit covector $\ell\in U_\sigma^*$, the scalar function
$h(u)=\ell(f(\exp(uX)k))$ is therefore an exponential polynomial of
exponential type at most $L$, bounded on the real axis by $q^\infty(f)$.
Bernstein's inequality on the real line \cite[p.~554]{DuffinSchaeffer92}
gives $|h'(0)|\leq Lq^\infty(f)$. Taking the supremum over $\ell$
and $k$ yields
\[
 \|Xf\|_\infty\leq Lq^\infty(f).
\]
The bound is for $f$ itself, not the sum of the bounds for its
individual $K$-types. Since $w_{\re\lambda}$ is bi-$K$-invariant,
$p_{\min}$ in \eqref{def pmin mps} is $K$-invariant. We may therefore
translate $f$ so that $\|f(1)\|_\sigma=q^\infty(f)$.

Let $\psi$ correspond to $f$ under \eqref{eq switch picture}.
On a fixed coordinate neighborhood of $0\in\oline\nf$ the coordinate
maps have bounded first derivatives; the scalar factor in
\eqref{eq switch picture} and its first derivatives are bounded there
locally uniformly in $\lambda$. Hence there is a number $0<r\leq1$,
independent of $L$ and $\sigma$, and uniform for $\lambda$ in any prescribed compact parameter set (with $r$ depending on that set), such that
\[
 \|\psi(Y)-\psi(0)\|_\sigma\leq\tfrac12 q^\infty(f)
 \qquad(\|Y\|\leq r/L).
\]
Take $e^{-t}=r/L$. In a root-space orthogonal norm,
$\Ad(a_t)B\subset e^{-t}B$. By change of variables, the normalized
integral in \eqref{Dirac1} is an average against the nonnegative
function $\phi$ of integral one. Taking a unit covector norming
$\psi(0)$ gives
\begin{equation}\label{Dirac2}
 a_t^{\rho+\re\lambda}
 \sup_{u^*\in B_{\sigma^*}}|\la\phi_{u^*,t},\psi\ra|
 \geq \tfrac12 q^\infty(f).
\end{equation}
Now use $g=a_t^{-1}$ in \eqref{def pmin mps}. Since
$w_{\re\lambda}(a_t^{-1})=w_{\re\lambda}(a_t)$ and
$a_t^{\re\lambda}\leq w_{\re\lambda}(a_t)$, we obtain
\[
 q^\infty(f)\leq 2a_t^\rho w_{\re\lambda}(a_t)^2 p_{\min}(f).
\]
Finally, $a_t^\rho=e^{tc_\gf}$ and the finite Weyl-group maximum
defining the weight gives $w_{\re\lambda}(a_t)^2\leq e^{tcR}$.
Substituting $e^t=L/r$ proves \eqref{ulb1}.
\end{proof}

Key Lemma \ref{Lemma Uniform lower bounds} features an abstract convexity bound for minimal principal series which is of particular interest in the unitary case, i.e., $R=0$.
We record the following corollary. 

\begin{cor}\label{Cor mps convexity}
Let $R\geq 0$ and $V_{\sigma,\lambda}$ be a minimal principal series with  $\|\re \lambda\|\leq R$ and $s=c_\gf+cR $. Let $p_{\min}$ and $p_{\max}$ be representatives of the minimal and maximal norms in $\Noc(V_{\sigma,\lambda}, w_{\re \lambda})$. Then the following assertions hold:
\begin{enumerate}[(i)]
\item\label{Cor mps convexity - item 2}
For every $r>0$ one has $q_{-s-r}\lesssim p_{\min}$ and
$p_{\max}\lesssim q_{s+r}$. In particular,
$[p_{\max}]=[(q_{s+r})^{G,w_{\re\lambda}}]$.

\item\label{Cor mps convexity - item 3}{\rm(Abstract Convexity Bounds)} There exist constants $C_1,C_2>0$ such that 
for all $\tau\in \hat K$ and $v_\tau\in V_{\sigma, \lambda}[\tau]$ one has 
\begin{enumerate}[(a)]
    \item\label{Cor mps convexity - item 3a}
    $p_{\min}(v_\tau)\geq C_1\,  q_{-s} (v_\tau)\,.$
    \item\label{Cor mps convexity - item 3b}
    $p_{\max}(v_\tau)\leq C_2\, q_{s} (v_\tau)\,.$
\end{enumerate}
\end{enumerate}
\end{cor} 
\begin{proof}
Let $P_j=\mathbf{1}_{[2^j,2^{j+1})}(D)$ in the compact-picture
Hilbert space. For the smooth cutoff $S_j$ of
Section~\ref{Subsection p_s} we have $P_jS_j=P_j$ and
$S_jv\in E_{C2^j}$ for a fixed $C$. Taking $p_{\min}$ to be
$K$-invariant, we obtain
\begin{equation}\label{eq:mps-shell-estimate}
 q(P_jv)\leq q(S_jv)\leq q^\infty(S_jv)
 \lesssim 2^{js}p_{\min}(S_jv)
 \lesssim 2^{js}p_{\min}(v).
\end{equation}
Here the first inequality is Hilbert-space orthogonality, the third
is \eqref{ulb1}, and the last follows from Lemma~\ref{lem:K-dyadic}.

For~\textup{(\ref{Cor mps convexity - item 2})}, orthogonality of
the shells and \eqref{eq:mps-shell-estimate} give, for every $r>0$,
\[
   q_{-s-r}(v)^2
      \asymp \sum_{j\geq0}2^{-2j(s+r)}q(P_jv)^2
      \lesssim p_{\min}(v)^2\sum_{j\geq0}2^{-2jr}.
\]
Hence $q_{-s-r}\lesssim p_{\min}$. Apply this inequality to the
contragredient $V_{\sigma^*,-\lambda}$ and dualize in the
compact-picture pairing. Since $D$ is self-transpose, the dual
of the minimal norm represents the maximal norm, and
$w_{-\re\lambda}=w_{\re\lambda}=w_{\re\lambda}^\sharp$, we obtain
$p_{\max}\lesssim q_{s+r}$. The asserted stabilization follows
from Proposition~\ref{Prop stabilization 1}.

For~\textup{(\ref{Cor mps convexity - item 3})}, let $v_\tau$
belong to the $j$-th shell. Then $P_jv_\tau=v_\tau$, and
\eqref{eq:mps-shell-estimate} gives
\[
   q_{-s}(v_\tau)\asymp 2^{-js}q(v_\tau)
      \lesssim p_{\min}(v_\tau).
\]
For the maximal norm, let $u$ be a test vector in the contragredient,
and denote its compact-picture and minimal norms by $q^\vee$ and
$p_{\min}^\vee$, respectively. Since $P_j$ is self-transpose,
the shell estimate on that module gives
\[
   |\langle v_\tau,u\rangle|
      =|\langle v_\tau,P_ju\rangle|
      \leq q(v_\tau)\,q^\vee(P_ju)
      \lesssim 2^{js}q(v_\tau)\,p_{\min}^\vee(u).
\]
Taking the supremum over $p_{\min}^\vee(u)\leq1$ yields
$p_{\max}(v_\tau)\lesssim 2^{js}q(v_\tau)\asymp q_s(v_\tau)$.
\end{proof}

\begin{proof}[Proof of Theorem~\ref{thm sup}]
Fix $V=V_{\sigma,\lambda}$ with $\|\re\lambda\|\leq R$, put
$w=w_{\re\lambda}$ and $a=c_\gf+cR$, and let $q$ be the
compact-picture Hilbert norm. Recall that $q\in\No(V,w)$.
We first compare $q$ separately with the two extremal norms.
The distinction matters: the estimate for the minimal norm
comes from Dirac approximation, whereas the estimate for the
maximal norm comes from duality.

Let $t>a$. The minimal-norm estimate of
Corollary~\ref{Cor mps convexity} gives $q_{-t}\lesssim p_{\min}$.
With $D=(1+\Delta_K)^{1/2}$, applying this to $D^t v$ yields
$q\lesssim p_{\min,t}$. Since $p_{\min}\lesssim q$, it follows
that $d([q],[p_{\min}])\leq t$.

For the other distance use that the family of minimal
principal series is closed under contragredients:
$\tilde V\simeq V_{\sigma^*,-\lambda}$. Its parameter satisfies
the same bound $R$, and its weight is
$w_{-\re\lambda}=w=w^\sharp$. Write $q^\vee$ for its
compact-picture norm and $p_{\min}^{\vee}$ for its minimal norm
at that weight. Applying the \emph{minimal}-norm estimate to
this module gives
\[
                      (q^\vee)_{-t}\lesssim p_{\min}^{\vee}.
\]
Under the compact-picture pairing the dual of $p_{\min}^{\vee}$
represents $p_{\max}$ on $V$. The dual of $(q^\vee)_{-t}$ is
$q_t$, since the two compact-picture Hilbert norms are dual and
$D$ is self-transpose for this pairing. Reversing the domination
by duality therefore gives $p_{\max}\lesssim q_t$. Together with
$q\lesssim p_{\max}$ this proves $d([q],[p_{\max}])\leq t$.
It does not require the compact-picture norm to be $G$-invariant.

Letting $t\downarrow a$, we have proved both estimates
\[
       d([q],[p_{\min}])\leq c_\gf+cR,\qquad
       d([q],[p_{\max}])\leq c_\gf+cR .
\]
The triangle inequality now gives
\[
 \gap(V,w)=d([p_{\min}],[p_{\max}])\leq 2c_\gf+2cR ,
\]
uniformly in the stated range. For $t>a$ the domination
$p_{\max}\lesssim q_t$ also gives
$[p_{\max}]=[(q_t)^{G,w}]$ by
Proposition~\ref{Prop stabilization 1}. Finally, for imaginary
$\lambda$ we may take $R=0$ and $w=\1$.
\end{proof}

\begin{cor}\label{cor:full-induced-gap}
Assume the minimal principal series $V_{\sigma,\lambda}$ admits a norm of growth $\1$. Then
\[
             \gap(V_{\sigma,\lambda})\leq 2c_\gf .
\]
In particular this holds for every unitarizable full induced module,
including the complementary series.
\end{cor}

\begin{proof}

Write $q$ for the compact-picture $L^2(K,\sigma)$-norm; we do not
assume it is $G$-invariant. Take $p_{\min}$ and $p_{\max}$ for the
weight $\1$ and set $b=(\re\lambda)(X_0)$. By
Lemma~\ref{lemma smooth minimal} we may use the generating set
$C_\sigma$ of \eqref{def pmin mps}, now without a weight in the
denominator. Its separation property is independent of the weight.

For $f\in E_L$ use \eqref{Dirac2} with $e^{-t}=r/L$ as in the Key
Lemma. Since the supremum defining $p_{\min}$ 
includes $a_t^{-1}$ and there is no weight factor to estimate we obtain that 
\[
 q^\infty(f)\lesssim a_t^{\rho+\re\lambda}p_{\min}(f)
              \lesssim L^{c_\gf+b}p_{\min}(f).
\]

Let $P_j$ and $S_j$ be the sharp projection and the smooth cutoff,
respectively, from the proof of Corollary~\ref{Cor mps convexity}.
Since $P_jS_j=P_j$ and $S_jv\in E_{C2^j}$, the band estimate gives
\[
   q(P_jv)\leq q(S_jv)\leq q^\infty(S_jv)
     \lesssim 2^{j(c_\gf+b)}p_{\min}(S_jv)
     \lesssim 2^{j(c_\gf+b)}p_{\min}(v).
\]
The square sum with weights
$2^{-2j(c_\gf+b+\epsilon)}$ thus gives
\[
 q_{-c_\gf-b-\epsilon}\lesssim p_{\min}\qquad(\epsilon>0).
\]
The contragredient $V_{\sigma^*,-\lambda}$ also admits a norm of
growth $\1$. Its signed parameter is $-b$. Applying the last estimate
there and dualizing under the compact-picture pairing gives
$p_{\max}\lesssim q_{c_\gf-b+\epsilon}$ on $V_{\sigma,\lambda}$.
Here the compact-picture norms are dual under that pairing, independently
of whether they are unitary, and the Sobolev operator is self-transpose.
Consequently
\[
 p_{\max}\lesssim q_{c_\gf-b+\epsilon}
       \lesssim p_{\min,\,2c_\gf+2\epsilon},
\]
where the second inequality is the first domination shifted by
$2c_\gf+2\epsilon$. Let $\epsilon\downarrow0$.
\end{proof}

Recall from \eqref{eq: def cg} the formula $c_\gf = \frac12 \sum_{\alpha>0} m_\alpha\operatorname{ht}(\alpha)$. 

\begin{theorem}\label{thm:flag-gap}

Let $G$ be connected semisimple with finite center. For every unitary $M$-type $\sigma$ and $\lambda\in i\af^*$, the normalized induction satisfies
\[
 \gap(V_{\sigma,\lambda})=2c_\gf.
\]
\end{theorem}

The upper bound is Theorem~\ref{thm sup}. For the lower bound we
may assume that $G$ has positive real rank, the compact case being
immediate. Put $c=c_\gf$ and write $q$ for the unitary norm.
By the midpoint formula (Proposition \ref{prop unitary}), we must rule out
\[
                     q_{-s}\lesssim p_{\min}
                     \qquad (0\leq s<c).
\]
The idea is to find unit vectors of frequency at most of order $L$
whose minimal norm is of order $L^{-c}$, up to logarithms.
The following is the precise statement we need.

\begin{lemma}\label{lem:flag-oscillation}
There exist smooth unit vectors $v_L$, $L\geq2$, and a constant
$b\geq0$ such that
\[
   q_1(v_L)\lesssim L,
   \qquad
   p_{\min}(v_L)\lesssim L^{-c}(1+\log L)^b .
\]
\end{lemma}

\begin{proof}[Deduction of Theorem~\ref{thm:flag-gap}]
The first estimate says that a fixed proportion of the mass of
$v_L$ lies at frequencies $\lesssim L$. Indeed, put
$D=(1+\Delta_K)^{1/2}$ and choose $a>0$ so that $q_1(v_L)\leq aL$.
Then
\[
   q\bigl(\mathbf{1}_{(2aL,\infty)}(D)v_L\bigr)
       \leq (2aL)^{-1}q_1(v_L)\leq\tfrac12 .
\]
Since $q(v_L)=1$, it follows that $q_{-s}(v_L)\gtrsim_s L^{-s}$
for every $s\geq0$. Thus $q_{-s}\lesssim p_{\min}$ would imply
\[
                 L^{c-s}\lesssim_s (1+\log L)^b,
\]
which is impossible for $s<c$. 
\end{proof}

We construct $v_L$ on the flag manifold $X=G/P$. Choose a smooth
section $S$ of the inducing bundle $G\times_P U_\sigma \to G/P$, supported in a small neighborhood
of the base flag $o=P\in X$, and a real phase function $\phi: X\to \R$ there. Our vectors will be
\[
                         v_L=e^{iL\phi}S .
\]
Multiplication by the phase does not change the unitary norm, so we
normalize $q(S)=1$. Note that $q(v_L)=1$. Differentiating once in the $K$-directions gives
$q_1(v_L)\lesssim L$. The issue is to choose $\phi$ so that
$p_{\min}(v_L)$ is small.

To see what this asks of the phase, write $g=k\exp(H)k'$ for $k,k'\in K$, $H\in \oline{\af^+}$ and put
$$m(H)=\min_{\alpha\in\Pi}\alpha(H).$$
When a derivative in direction
$Z\in\gf$ is transferred to the translated test vector $\pi(g)w$, it acts
on $w$ in direction $\Ad(g^{-1})Z$. If $\alpha(H)=m(H)$, derivatives
in the space
\[
       \Ad(k)\sf_\alpha,\qquad
       \sf_\alpha=\pf_P+\gf^{-\alpha},\qquad
       \pf_P=\operatorname{Lie}(P),
\]
therefore cost at most $O(e^{m(H)})$. A derivative of the oscillatory
factor, on the other hand, supplies a factor $L$, provided the phase
varies in that direction. The phase must be chosen before $g$ is
known. We thus need one phase which varies on every
$\Ad(k)\sf_\alpha$, uniformly in $k\in K$ and $\alpha\in \Pi$.

\begin{lemma}[A phase for all Cartan translates]\label{lem:flag-phase}
For $Y\in\gf$ let $Y_X$ be the vector field generated by the
$G$-action on $X=G/P$. Fix a norm on $\gf$. There are a neighborhood
$U$ of $o$, a real smooth phase $\phi$ on $U$, and $\delta>0$
such that for every $x\in U$, $k\in K$, and $\alpha\in\Pi$,
\[
 \sup_{\substack{Y\in\sf_\alpha\\ \|Y\|=1}}
       |(\Ad(k)Y)_X\phi(x)|\geq\delta .
\]
\end{lemma}

\begin{proof}
Use the Killing form to identify $T_o^*X$ with
$\pf_P^\perp=\nf$. At $o$ the assertion asks for a single covector
$J\in\nf$ outside all $\Ad(k)\sf_\alpha^\perp$, where
\[
 \sf_\alpha^\perp
          =\bigoplus_{\beta>0,\ \beta\ne\alpha}\gf^\beta .
\]
We find $J$ by comparing orbit dimensions.

Let $G_\C$ be the complex adjoint group and $P_\C$ the parabolic
with Lie algebra $\pf_{P,\C}$. Set
$d=\dim_\C\nf_\C=\dim_\C G_\C/P_\C$. The subspace $\sf_{\alpha,\C}^\perp$
is $P_\C$-stable: the Levi algebra preserves each restricted root
space, and bracketing with $\nf_\C$ adds a positive restricted root,
so cannot produce the missing simple root $\alpha$. This argument
also applies when $2\alpha$ is a root: the $2\alpha$-space is
retained in $\sf_\alpha^\perp$. The closure of the $G_\C$-saturation of
$\sf_{\alpha,\C}^\perp$ therefore has dimension at most
\[
 \dim_\C(G_\C\times_{P_\C}\sf_{\alpha,\C}^\perp)
                  =2d-\dim_\R\gf^\alpha<2d .
\]

By Richardson's theorem \cite{Richardson}, $P_\C$ has an open orbit
in $\nf_\C$. For $e$ in this orbit,
$[\pf_{P,\C},e]=\nf_\C$. Invariance of the Killing form implies
$\gf_\C^e\subset\nf_\C^\perp=\pf_{P,\C}$, so the centralizer
has dimension $\dim_\C\pf_{P,\C}-d$. Hence $G_\C\cdot e$ has
dimension $2d$ and is disjoint from each of the invariant closed
sets above. The real form $\nf$ is Zariski dense in $\nf_\C$;
choose $J\in\nf$ in the open Richardson orbit. This $J$ avoids
every $\Ad(k)\sf_\alpha^\perp$.

Take a real local phase with $d\phi_o=J$ under the cotangent
identification. At $o$ the supremum in the assertion is positive
for each $k$ and $\alpha$. By compactness of $K$ and finiteness of
$\Pi$, it has a positive uniform lower bound there. By continuity,
a smaller positive bound holds throughout one neighborhood $U$.
\end{proof}

\noindent\emph{Split type $A$.}
For $\SL(n,\R)$ one can take the regular nilpotent matrix
$J=\sum_{i=1}^{n-1}E_{i,i+1}$. Under the trace pairing a matrix
in $\sf_{\alpha_i}^\perp$ is strictly upper triangular with zero
$(i,i+1)$-entry, so its $(n-1)$-st power vanishes. But
$J^{n-1}=E_{1n}$, proving the required avoidance directly.
In general the proof uses a Richardson element for $P_\C$;
it need not be regular nilpotent in $\gf_\C$.

\medskip
Shrink $U$ if necessary and choose a smooth section $S$ compactly
supported in $U$, with $q(S)=1$. For $L\geq2$ put
$v_L=e^{iL\phi}S$, extended by zero to $X$.

\begin{lemma}[Oscillation and spherical decay]\label{l:ibp}
For each fixed smooth test section $w$ and integer $r\geq0$,
\begin{equation}\label{e:ibp}
 |\langle v_L,\pi(k\exp(H)k')w\rangle|
 \leq C_{r,w}\,\Xi(\exp H)
             \min\{1,(e^{m(H)}/L)^r\}.
\end{equation}
\end{lemma}

\begin{proof} 
We now pass to the compact picture to estimate the matrix
coefficients. Since $\lambda$ is imaginary, the multiplier in
\eqref{eq:compact-Iwasawa-action} has modulus
${\bf a}(g^{-1}k)^{-\rho}$. Integrating over $K$ with normalized
Haar measure gives for all $u\in V^\infty$ and $g\in G$
\begin{equation}\label{e:xi}
|\langle u,\pi(g)w\rangle|
      \leq \|u\|_\infty\|w\|_\infty\,\Xi(g).
\end{equation}
Since $\|v_L\|_\infty=\|S\|_\infty$ is independent of $L$, this proves the first bound.

For the other bound, take $\alpha$ attaining $m(H)$ and fix an
orthonormal basis $(Y_j)$ of $\sf_\alpha$. 
Put $Z_j=\Ad(k)Y_j$. The restricted root decomposition gives
\begin{equation}\label{e:ad}
 \|\Ad(g^{-1})Z_j\|
   =\|\Ad(k'^{-1})\Ad(\exp(-H))Y_j\|\lesssim e^{m(H)} .
\end{equation}
To differentiate the phase, set $b_j=(Z_j)_X\phi$ and
$a_j=b_j/\sum_h b_h^2$ near $\supp S$.
Lemma~\ref{lem:flag-phase} bounds the denominator away from zero.
Thus $\sum_j a_jb_j=1$, and the $a_j$ and their derivatives are
uniformly bounded there. This smooth combination avoids choosing
a different nonzero $b_j$ at each point.

Write $T_Z=d\pi(Z)$, whose principal part is $-Z_X$. For any smooth
amplitude $F$ supported in $\supp S$, the product rule gives
\[
 iL e^{iL\phi}F
  =\sum_j\bigl(e^{iL\phi}T_{Z_j}(a_jF)
                    -T_{Z_j}(e^{iL\phi}a_jF)\bigr).
\]
Pair with $\pi(g)w$. In the first term the derivative falls on the
amplitude. In the second, skew-adjointness transfers it to
$\pi(g)w$, where $T_{Z_j}\pi(g)w=\pi(g)T_{\Ad(g^{-1})Z_j}w$.
Thus each repetition gains $L^{-1}$ and, by \eqref{e:ad}, costs
at most $O(e^{m(H)})$ on the test section. After $r$ repetitions
the amplitudes remain uniformly bounded and the test sections have
supremum norms $O_{r,w}(e^{rm(H)})$: expand the at most $r$ Lie
algebra derivatives in a fixed basis. Applying \eqref{e:xi} term
by term and taking the smaller of the two bounds proves \eqref{e:ibp}.
\end{proof}

\begin{proof}[Proof of Lemma \ref{lem:flag-oscillation}]
Note that $q_1(v_L)\lesssim L$ holds by construction. 
To prove 
$$p_{\min}(v_L)\lesssim L^{-c}(1+\log L)^b$$  
we need a model for the minimal norm. Let $w_1,\ldots,w_\ell$ be generators of $V$. Then Lemma~\ref{lemma minimal norm} provides a minimal norm by
\begin{equation}\label{e:p}
 p_{\min}(v)=\max_{1\leq j\leq\ell}\sup_{g\in G}
                    |\langle v,\pi(g)w_j\rangle|.\end{equation}

Write $g\in G$ as $g=k \exp(H) k'$ for $k,k'\in K$ and $H\in \oline{\af^+}$. 
Harish-Chandra's
estimate \cite{GangolliVaradarajan} gives 
\[
 \Xi(\exp H)\lesssim(1+\|H\|)^b e^{-\rho(H)} \qquad (H\in \oline{\af^+})
\]
for a constant $b\geq 0$. Write $m=m(H)$. Since $\alpha(X_0)=1$ for every simple root,
$H-mX_0$ is dominant and $\rho(H)\geq cm$. It follows from the estimate that
\[\Xi(\exp H)\lesssim (1+m)^b e^{-cm}.\]
Indeed, use $\|H\|\lesssim\rho(H)$ on the
chamber, write $\rho(H)=cm+u$ with $u\geq0$, and absorb
$(1+u)^b$ into $e^u$.

If $m\leq\log L$, choose an integer $r>c$ in \eqref{e:ibp}.
Apart from the factor $(1+m)^b\leq(1+\log L)^b$, the bound is
\[
      e^{-cm}(e^m/L)^r=L^{-c}(e^m/L)^{r-c}\leq L^{-c}.
\]
This yields 
        $$ |\langle v_L,\pi(g)w_j\rangle|\lesssim_j  L^{-c}(1+\log L)^b.$$

If $m\geq\log L$, use \eqref{e:ibp} without integration by parts.
Writing $m=\log L+t$, the extra factor $(1+t)^be^{-ct}$ is
bounded. Both ranges give the asserted estimate, uniformly in $g$. Taking the supremum over $1\leq j\leq \ell$ proves the lemma. 
\end{proof}

\begin{rmk}
Under the hypotheses of Theorem~\ref{thm:flag-gap}, and in positive
real rank, the exponent $c_\gf$ in the band estimate
\eqref{ulb1} is optimal, even with $q^\infty$ replaced by $q$.
Indeed, if $q(v)\lesssim L^a p_{\min}(v)$ held on $E_L$ for some
$0\leq a<c_\gf$, the shell argument in the proof of
Corollary~\ref{Cor mps convexity} would give
$q_{-a-r}\lesssim p_{\min}$ for every $r>0$. Taking
$0<r<c_\gf-a$ contradicts the lower bound just proved.
\end{rmk}

\begin{rmk}\label{rmk mps weight}
\begin{enumerate}[(a)]
\item By Proposition~\ref{prop unitary}, $q$ has distance $c_\gf$ from either extremal isometric
norm. For $\SL(3,\R)$ these distances equal two. By Remark~\ref{rmk constant}, the full gaps for
$\so(1,n)$, $\su(1,n)$, $\sp(1,n)$ and $\mathfrak f_{4(-20)}$ are respectively $n-1,2n,4n+2$ and $22$
(with $n\ge2$ in the classical examples).

\item
Fix a weight $w$, replacing it by an equivalent bi-$K$-invariant
one. For $a_t=\exp(tX_0)$ choose $M\geq0$ such that
\[
       \max\{w(a_t),w(a_t^{-1})\}\lesssim e^{Mt}
       \qquad(t\geq0).
\]
Such an $M$ depends only on $w$ and $X_0$: write $t$ as an integer
plus a number in $[0,1]$ and use local boundedness and
submultiplicativity. We claim that
\begin{equation}\label{thm sup reloaded}
 \gap(V_{\sigma,\lambda},w)\leq 2c_\gf+2M
 \qquad\bigl(\No(V_{\sigma,\lambda},w)\ne\varnothing\bigr).
\end{equation}

Indeed, let $b=(\re\lambda)(X_0)$ and compute the minimal norm
using the test family $C_\sigma$ and the denominator $w(g)$.
Its translates separate points independently of the weight.
Taking $g=a_t^{-1}$ in the defining supremum and applying
\eqref{Dirac2} with $e^{-t}=r/L$ gives, for $f\in E_L$,
\[
 q^\infty(f)\lesssim
    a_t^{\rho+\re\lambda} w(a_t^{-1})p_{\min}(f)
    \lesssim L^{c_\gf+b+M}p_{\min}(f).
\]
Constants here may depend on the representation; the order $M$
does not. The smooth dyadic summation used in
Corollary~\ref{Cor mps convexity} therefore gives
$q_{-c_\gf-b-M-\epsilon}\lesssim p_{\min}$ for every $\epsilon>0$.
On the contragredient use $w^\sharp$, which satisfies the same
two-sided bound on $a_t$, and replace $b$ by $-b$. Dualizing
gives $p_{\max}\lesssim q_{c_\gf-b+M+\epsilon}$. Shifting the
first inequality now yields
\[
 p_{\max}\lesssim q_{c_\gf-b+M+\epsilon}
       \lesssim p_{\min,\,2c_\gf+2M+2\epsilon}.
\]
Letting $\epsilon\downarrow0$ proves \eqref{thm sup reloaded}.

\end{enumerate}
\end{rmk}

\subsubsection{\texorpdfstring{Fixed weights for irreducible modules of $\SL(2,\R)$}{Fixed weights for irreducible modules of SL(2,R)}}\label{Section Conj D}

Conjecture~\ref{conj uniform finiteness} predicts a bound for all Harish-Chandra modules at a fixed
weight $w$. For irreducible modules of $\SL(2,\R)$ the bound follows as below.

For the discrete series and their two limits we use
Proposition~\ref{r97:weighted-disk}, proved after the disk-model
calculation in Appendix~\ref{Sec9: key estimates}.

By the rank-one classification every irreducible Harish-Chandra module for $\SL(2,\R)$ is finite
dimensional, an irreducible principal series, a discrete series or one of its limits. In the first
case all norms are equivalent. Use \eqref{thm sup reloaded} in the second and Proposition~\ref{r97:weighted-disk}
in the remaining cases. Consequently
\[
 \sup_{\substack{V\ {\rm irreducible}\\\No(V,w)\ne\varnothing}}\gap(V,w)<\infty .
\]

It is an open question to what extent $\gap(\pi)$ and $\gap^{\rm st}(\pi)$ depend on $\pi\in \hat G\bs\{\1\}$ for a simple Lie group $G$.

\section{\texorpdfstring{Sobolev gap for $\SL(2,\R)$}{Sobolev gap for SL(2,R)}}\label{Sobolev SL2}

In this section we let $G=\SL(2,\R)$ and $K=\SO(2)$. 
For a unitarizable and irreducible Harish-Chandra module $V$ we denote by
$S(V)=\Spec_K(V)\subset \hat K=\Z$ the $K$-spectrum. We fix a $G$-invariant unitary norm $q$ on $V$ and we write $V=\bigoplus_{n\in S(V)}\C e_{n}$ with the $e_{n}$ normalized with respect to $q$.

\begin{theorem} \label{thm ultimate}
Let $V\neq \C$ be a unitarizable irreducible Harish-Chandra module and
$[q]$ be the equivalence class of the unitary norm. Then
$$
d([q], [p_{\max}])
=\frac{1}{2}.
$$
Moreover,
$$
\gap(V)
=1
$$
and $[p_{\max}]=[q_s^{G}]$ for any $s>\frac{1}{2}$.
\end{theorem}

We record the following obvious consequence of Theorem \ref{thm ultimate} and its proof that leads to somewhat surprising consequences. 

\begin{theorem}[\rm Abstract Convexity Bound]\label{thm.abst. conv. bnds}
Let $V$ be an irreducible unitarizable Harish-Chandra module for $\SL(2,\R).$ Let $q$ be a unitary norm,
$S=\Spec_K(V)\subset \hat K=\Z$ and $(e_n)_{n\in S}$ an orthonormal basis consisting of $K$-types. Let
$p$ be any isometric $G$-continuous norm on $V$ and $\epsilon>0$. Then there exists a constant
$C_\e=C_\e(V,p,q)>0$, such that $p\le C_\e q_{\frac12+\epsilon}$. Moreover, there exists a
constant $C=C(V,p,q)>0$ such that
\begin{equation} \label{abstr conv bound}
p(e_n)
\leq C (1+|n|)^{\frac{1}{2}} \qquad (n \in S(V)).
\end{equation} 
\end{theorem}
\begin{proof}
For the trivial module all norms are proportional and the only weight
is zero, so both assertions are immediate. Assume now $V\ne\C$.
The first assertion $p\lesssim q_{\frac{1}{2}+\e}$ for any $\e>0$ follows from $p\lesssim p_{\max}$ and Theorem \ref{thm ultimate}. For \eqref{abstr conv bound} we use the estimate from Proposition \ref{thm upper-lower} (\ref{thm upper-lower - item 2}) below which is slightly sharper on individual $K$-types. 
\end{proof} 

We also need estimates on individual $K$-types. The upper bound for
$p_{\max}(e_n)$ gives the endpoint in \eqref{abstr conv bound}, while
its lower bound will distinguish the maximal norm from an automorphic
norm in Section~\ref{Sec: AF}.

\begin{prop}\label{thm upper-lower}
Let $V$ be a nontrivial irreducible unitarizable Harish-Chandra module
of $\SL(2,\R)$, and choose representatives $p_{\min},p_{\max}$ of
the extremal isometric norm classes. For $q$-unit weight vectors $e_n$
there exist $c,C>0$, independent of $n\in S(V)$, such that:
\begin{enumerate}[(i)]
\item\label{thm upper-lower - item 1}
 $c(1+|n|)^{-1/2}\leq p_{\min}(e_n)\leq C(1+|n|)^{-1/2}$.
\item\label{thm upper-lower - item 2}
$c(1+|n|)^{1/2}\leq p_{\max}(e_n)\leq C(1+|n|)^{1/2}$.
\end{enumerate}
The constants may depend on $V$, the normalization of $q$, and
the chosen representatives.
\end{prop} 

The following averaged estimate is needed only for the discrete
series and their limits. Here $A(n)\asymp B(n)$ means comparison
of absolute values up to positive constants for all sufficiently
large $|n|$ in the indicated spectrum.

\begin{prop}\label{thm estimate convex combination}
Let $V$ belong to the discrete series of $\SL(2,\R)$ or be one of
their two limits. Take $m$ to be its extremal $K$-type: the lowest
weight in the holomorphic case and the highest in the antiholomorphic
case. For every $\epsilon>0$ there is a probability measure
$\beta_\epsilon$ on $\oline{A^+}$ such that
\begin{equation}\label{eq Def I_(nu,epsilon)}
 I_{n,\epsilon}:=\int_{\oline{A^+}}
   \langle a\cdot e_m,e_n\rangle\,d\beta_\epsilon(a)
 \qquad(n\in S(V))
\end{equation}
satisfies
\begin{equation}\label{main estimate}
 |I_{n,\epsilon}|\asymp (1+|n|)^{-1/2-\epsilon}.
\end{equation}
\end{prop}

The individual estimates are proved in
Subsection~\ref{subsection Proof for Theorem upper-lower} and the
averaged estimate in
Subsection~\ref{subsection Proof for estimate convex combination}.

\begin{proof}[Proof of Theorem~\ref{thm ultimate}]

Every nontrivial irreducible unitarizable Harish-Chandra module of
$\SL(2,\R)$ is infinite dimensional and has Gelfand--Kirillov dimension
one. Indeed, its $K$-types have multiplicity one, and a filtered piece
$F_n$ has $K$-frequencies of size $O(n+1)$; thus $\dim F_n=O(n+1)$.
Since the filtration is unbounded, its dimension polynomial has degree
one. 
Theorem~\ref{thm:unitarizable-gk-gap} therefore gives
$\gap(V)\geq1$ in every case.

For the irreducible unitary principal series and for the
complementary series, Corollary~\ref{cor:full-induced-gap} gives the
opposite inequality, since $2c_{\gf}=1$. Put
$\langle n\rangle=(1+n^2)^{1/2}$.

It remains to prove the upper bound for the discrete series and their
limits. Fix $\epsilon>0$, set $s=\tfrac12+\epsilon$, and choose an
isometric representative $p_{\max}$. The vector
$$u=\int_{\oline{A^+}} a\cdot e_m\,d\beta_\epsilon(a)$$ 
is a Bochner
integral in the $p_{\max}$-completion with
$p_{\max}(u)\leq p_{\max}(e_m)$. Its image in the unitary
globalization has Fourier coordinates $u_n$ satisfying
$|u_n|=|I_{n,\epsilon}|$. By \eqref{main estimate},
$|u_n|\gtrsim\langle n\rangle^{-s}$ outside a finite set $F$.

Let $v=\sum_n v_ne_n\in V$. Choose a trigonometric polynomial $h$
on $K$ such that 
$$\int_K h(k)\chi_n(k)\,dk=v_n/u_n\qquad (n\in S(V)\setminus F)$$
and all other moments are zero; Haar measure
has mass one. 
Parseval gives 
$$\|h\|_1\leq\|h\|_2\lesssim q_s(v).$$
The vector $\int_K h(k)k\cdot u\,dk$ is precisely
$\sum_{n\notin F}v_ne_n$: its coordinates agree, and the canonical
map from the maximal completion into the unitary globalization is
injective. Its maximal norm is therefore at most
$p_{\max}(e_m)\|h\|_1$. The remaining finitely many coordinates are
bounded in $p_{\max}$ by a constant times $q_s(v)$. This proves
$p_{\max}\lesssim q_{1/2+\epsilon}$. Since $q\lesssim p_{\max}$,
we obtain $d([q],[p_{\max}])\leq\tfrac12$, and
Proposition~\ref{prop unitary} gives $\gap(V)\leq1$.

We have now proved $\gap(V)=1$ in every case. The same proposition
gives $d([q],[p_{\max}])=\tfrac12$, and
Proposition~\ref{Prop stabilization 1} gives
$[p_{\max}]=[q_s^G]$ for all $s>\tfrac12$.

\end{proof}

\section{\texorpdfstring{Sobolev regularity and bounds on Fourier coefficients of $H$-distinguished functionals}{Bounds on fourier coefficients of H-distinguished functionals}}\label{Sec: SR}

In this section $G$ will be a real reductive group. 
Let us assume that $H<G$ is a closed unimodular subgroup of $G$ and let $X=H \bs G$ with an invariant measure $\mu_X.$
Motivated by questions of Harmonic analysis on $X$, we study irreducible Harish-Chandra modules $V$ such that
$$
(V^{-\infty})^{H}
= \Hom_{H, {\rm cont}} (V^\infty, \C)\neq \{0\}.
$$
For $\eta\in (V^{-\infty})^{H}\setminus\{0\}$ the associated matrix coefficients
$$
m_{v,\eta}(H g)
:= \eta(g \cdot v) \qquad (v\in V^\infty, g \in G)
$$
are commonly referred to as generalized matrix coefficients. Note that the $m_{v, \eta}$'s are smooth functions on $X$. 
We record the bijection from Frobenius reciprocity 
$$
(V^{-\infty})^{H}\to \Hom_{G, {\rm cont}}(V^\infty, C^\infty(X)),\quad \eta\mapsto T_\eta,
$$
where
$$
T_\eta(v)
=m_{v,\eta}\qquad (v\in V^{\infty}).
$$

We recall that an $H$-distinguished pair is a pair $(V,\eta)$ of a Harish-Chandra module $V$ and a functional $\eta\in (V^{-\infty})^{H}\setminus\{0\}$. 
Let $1 \leq r \leq \infty.$ We say that the pair $(V,\eta)$ is {\it $(X,r)$-bounded} if all matrix
coefficients $m_{v,\eta}$, $v\in V^\infty$, lie in $L^{r}(X,\mu_X)$. In such a case we can define
an invariant seminorm $q_{\eta,X,r}$ on $V^\infty$ given by:
$$
q_{\eta,X,r}(v)
:=\|m_{v, \eta}\|_{r} \qquad (v \in V^\infty).
$$

The map $T_\eta:V^\infty\to L^r(X)$ is continuous by the closed graph theorem: convergence in either
target $C^\infty(X)$ or $L^r(X)$ implies convergence as distributions. Its seminorm defines a norm on the
quotient by its kernel, and on $V$ whenever $T_\eta$ is injective, in particular for irreducible $V$. For
$r<\infty$ translations are strongly continuous. For $r=\infty$, factor each smooth vector as a finite
sum of convolutions by $C_c^\infty(G)$ using \cite[Th\'eor\`eme 3.3]{DixmierMalliavin}. The translation
difference is then bounded by the $L^1$ translation difference of the convolution kernel times the
sup-norm of the corresponding coefficient. This proves strong continuity on the image and its completion.
For $r=\infty$ we abbreviate $(X,\infty)$-bounded to $X$-bounded
and write
$$
q_{\eta,X}
:=q_{\eta,X,\infty}\, .
$$
The concept of $X$-boundedness is ubiquitous in harmonic analysis on homogeneous spaces. We collect some general results from the literature. 

\begin{lemma}\label{Lemma X-bounded}
Let $(V,\eta)$ be an $H$-distinguished Harish-Chandra module. Then the following assertions hold:
\begin{enumerate}[(i)]
\item\label{Lemma X-bounded - item 1}
If $X=H\bs G$ is compact, for instance if $H=\Gamma$ is a cocompact lattice, then $(V,\eta)$ is $X$-bounded.
\item\label{Lemma X-bounded - item 2}
Suppose $H$ is an algebraic reductive subgroup. If $(V,\eta)$ is $(X,p)$-bounded for some $1 \leq p < \infty$ then $(V,\eta)$ is $X$-bounded. 
\item\label{Lemma X-bounded - item 3}
Suppose $H$ is symmetric or more generally a real spherical reductive subgroup of wavefront type and suppose $V$ is unitarizable. Then $(V,\eta)$ is $X$-bounded. 
\end{enumerate} 
\end{lemma}

\begin{proof}
Item (\ref{Lemma X-bounded - item 1}) is clear and (\ref{Lemma X-bounded - item 2}) is \cite[Theorem 1.2]{KSS2}. Finally (\ref{Lemma X-bounded - item 3}) is \cite[Th. 7.6]{KKSS}.
\end{proof}

We will consider the $K$-coefficients of $H$-fixed functionals as follows: For $\tau \in \hat{K}$ we 
let $\eta_{\tau}=\eta|_{V[\tau]}$ be the restriction of the functional $\eta$ to the $\tau$-isotypical component $V[\tau] \subset V$.

\begin{prop}\label{Prop X-bounded}
Let $V$ be a unitarizable Harish-Chandra module equipped with a unitary norm $q$. Let $\eta\in(V^{-\infty})^H$ be a continuous functional on $V^\infty$. Assume that the pair $(V,\eta)$ is $X$-bounded. Then for any $s>d([p_{\max}],[q])$ the functional $\eta$ extends continuously to the completion $V_{q_s}$.
In particular, there exist constants $C=C(s)>0$ so that for any $\tau \in \hat{K}$ we have
$$
\tilde q(\eta_{\tau})
\leq C (1+|\tau|)^s
$$
with $\tilde q$ the dual norm of $q$. 
\end{prop}

\begin{proof}

For possibly reducible $V$, $q+q_{\eta,X}$ is an isometric globalization norm: its orbit maps are
continuous by the preceding argument, and isotypic projection identifies the $K$-finite vectors in the
completion. Hence $q_{\eta,X}\le q+q_{\eta,X}\lesssim p_{\max}\lesssim q_s$ for every
$s>d([p_{\max}],[q])$. Evaluation at the identity coset gives $|\eta(v)|\le q_{\eta,X}(v)$.

Let $\tau$ be an irreducible representation of $K$.
Observe that
$$
\tilde q(\eta_{\tau})
=\sup_{v \in V[\tau] \setminus \{0\}} \frac{|\eta(v)|}{q(v)}.
$$
Now there exist constants $C,C'$ such that for any $\tau \in \hat{K}$ and any $v \in V[\tau]$ we have
$$
|\eta(v)|
\leq \| m_{v,\eta}\|_{\infty}
=q_{\eta,X}(v) \leq C'q_{s}(v)
\leq C(1+|\tau|)^{s} q(v).
$$
This finishes the proof.
\end{proof}

\begin{rmk}\label{rmk Sobolev class of functionals for SL2} Proposition \ref{Prop X-bounded} becomes of relevance when $d([p_{\max}],[q])$ is exactly known, for example 
in case of $G=\SL(2,\R)$ where $d([p_{\max}],[q])=\frac{1}{2}$ for all nontrivial irreducible unitarizable Harish-Chandra modules $V$, 
see Theorem \ref{thm ultimate}. Consider, for example, $X=\SO(1,1)\bs \SL(2,\R)$ and note that
$H=\SO(1,1)$ is a symmetric subgroup. Lemma \ref{Lemma X-bounded} (\ref{Lemma X-bounded - item 3})
implies that all distinguished pairs $(V,\eta)$ are $X$-bounded; hence Proposition \ref{Prop X-bounded}
applies and gives the result that all $\eta \in (V^{-\infty})^H$ are of Sobolev class
$-\frac{1}{2}-\epsilon$ for all $\epsilon>0$, regardless of the type of irreducible unitarizable
module $V$. In fact $\eta$ has Sobolev order $-\epsilon$ for
every $\epsilon>0$. After compact conjugation, take the split
generator $h=\diag(1,-1) \in \operatorname{Lie}(H)$. Choose phases for the unit $K$-vectors
so that $d\pi(h)e_n=a_ne_{n+2}-a_{n-2}e_{n-2}$.
On each link of the weight string the circle and disk models
of Appendix~\ref{Sec9: key estimates} give
$a_n=\tfrac12\sqrt{(n+1)^2-\lambda^2}$, with a missing link
assigned coefficient zero. Here $\lambda$ is the principal- or
complementary-series parameter; for lowest or highest absolute
weight $\ell$ take $\lambda=\ell-1$.
Writing $u_n=\eta(e_n)$, invariance gives
$a_nu_{n+2}=a_{n-2}u_{n-2}$. On a positive tail,
$a_{n-2}/a_n=1-2/n+O(n^{-2})$. This implies $u_n=O(n^{-1/2})$ on each progression. Reflection gives
the same estimate on negative tails. Thus
$\sum_n(1+n^2)^{-\epsilon}|u_n|^2<\infty$ for every
$\epsilon>0$, as asserted.
\end{rmk}

\section{Relations to automorphic forms}\label{Sec: AF}
Let $\Gamma<G$ be a lattice and put $X=\Gamma\bs G$.
We study the sup-norm of vectors in a fixed automorphic
realization. We begin with compact quotients, where smooth
automorphic functions are bounded, and then treat $\SL(2,\R)$
for an arbitrary lattice. For a nonzero cuspidal realization of a
nontrivial irreducible unitary representation of $\SL(2,\R)$,
the exact range of Sobolev domination is $s>1/2$.

In arbitrary real rank we first bound the sup-norm on
individual $K$-types of the unitary minimal principal series.
We then turn to the Sobolev order needed to control all smooth
vectors. For a nonzero $X$-bounded realization of an
infinite-dimensional irreducible unitary representation of a
connected real reductive group, integration over $X$ gives the
necessary condition $s>\tfrac12\operatorname{GKdim}V$
(Proposition~\ref{prop:aut-finite-volume-bounds}). On a compact
quotient this condition is also sufficient. On a noncompact quotient
the threshold can be larger: for an irreducible tempered minimal
principal series, a nonzero regular Fourier coefficient can force
it to equal $\tfrac12\gap(V)$
(Proposition~\ref{prop:regular-fourier-threshold}).

\subsection{Cocompact lattices} 

Let us assume that $\Gamma<G$ is a cocompact lattice and let $X=\Gamma\bs G$.
Our interest lies in irreducible Harish-Chandra modules $V$ such that
$$
(V^{-\infty})^\Gamma
= \Hom_{\Gamma, {\rm cont}} (V^\infty, \C)\neq \{0\}.
$$
Let now $\eta\in (V^{-\infty})^\Gamma\setminus\{0\}$. The associated generalized matrix coefficients
$$
m_{v,\eta}(\Gamma g)
:= \eta(g \cdot v) \qquad (v\in V, g \in G)
$$
are commonly referred to as automorphic functions. Note that the $m_{v, \eta}$'s are smooth functions on $X$. Therefore, they are bounded as $X$ is compact. It follows that
$V$ is unitarizable with a unitary norm $q$ given by
$$
q(v)^2
= \int_X | m_{v,\eta}(\Gamma g)|^2\,d(\Gamma g)
\qquad(v\in V).
$$
Another isometric norm of interest is
$$
p_{\aut}(v)
:= \sup_{x \in X} |m_{v, \eta}(x)|
\qquad (v \in V).
$$
The possible dependence of $p_{\rm aut}$ on $\eta$ is suppressed in the notation and will be discussed at the end of this Section. 

Note that $q\leq \sqrt{\vol(X)}p_{\aut} $ and thus $q\lesssim p_{\aut}$.
Hence $d_{\to} ([q], [p_{\aut}])=0$.

\subsection{Automorphic forms for \texorpdfstring{$G=\SL(2,\R)$}{G=SL(2,R)} - cocompact case}

The following result is implicit in \cite{BR}.

\begin{prop}\label{Prop SL2 distance}
Let $G=\SL(2,\R)$ and let $\Gamma$ be a cocompact lattice of $G$. Let $V$ be a Harish-Chandra module of a $K$-spherical unitary principal series representation of $G$ and assume that $V$ is $\Gamma$-automorphic. Then
$$
d([p_{\rm aut}],[q])
=\frac{1}{2}.
$$
\end{prop}

\begin{proof}
It was shown in
\cite{BR} that
$$
p_{\aut}
\lesssim q_s \iff s > \frac{1}{2}.
$$
Therefore, $d_{\to} ( [p_{\aut}],[q])=\frac{1}{2}$.
\end{proof}

The preceding proposition and Theorem \ref{thm ultimate} show that $p_{\rm aut}$ and $p_{\max}$ are both at distance $1/2$ from $q$. Nevertheless, they are at positive distance from one another.

\begin{prop}\label{prop upper gap}
Let $G=\SL(2,\R)$ and let $\Gamma$ be a cocompact lattice of $G$. Let $V$ be a Harish-Chandra module of a $K$-spherical unitary principal series representation of $G$ and assume that $V$ is $\Gamma$-automorphic. Then
$$
d([p_{\rm aut}], [p_{\max}])
\geq \frac{1}{6}.
$$
\end{prop}
\begin{proof} Suppose that
$p_{\max}\lesssim p_{{\rm aut},s}$. On $e_n$ the right-hand side
is comparable to $\langle n\rangle^s p_{\rm aut}(e_n)$, where
$\langle n\rangle=(1+n^2)^{1/2}$. The lower bound in
Proposition~\ref{thm upper-lower}(\ref{thm upper-lower - item 2})
and the following estimate of Reznikov force
$s+\tfrac13+\epsilon\geq\tfrac12$ for every $\epsilon>0$:
indeed $(1+|n|)^{1/2}\lesssim_\epsilon
(1+|n|)^{s+1/3+\epsilon}$ on the unbounded $K$-spectrum.
Letting $\epsilon\downarrow0$ gives $s\geq\tfrac16$.
\end{proof}
\begin{theorem} [Reznikov]
Under the assumptions of Proposition \ref{prop upper gap} there exists for every $\epsilon >0$ a constant $c>0$ so that
$$
p_{\rm aut} (e_n)
\leq c\, (1+|n|)^{\frac{1}{3}+\epsilon}
\qquad \big(n\in S(V)\big).
$$
\end{theorem}

\begin{proof} The result is already contained, in a different guise, in
Reznikov's estimate for spherical Fourier coefficients
\cite[Th.~1.5]{R}. In our notation, his estimate bounds
$|\eta(e_n)|$ at a fixed base point. To pass to $p_{\rm aut}(e_n)$,
we need only let that point vary.

Indeed, $K$ is the stabilizer of the reference point $o=eK$ in
$G/K$. Moving the center from $o$ to $go$ replaces $K$ by
$gKg^{-1}$ and its weight vector $e_n$ by $g e_n$. Thus
$\eta(g e_n)$ is the same spherical Fourier coefficient, now
centered at $\Gamma gK$. The choice of frame at the new center
does not affect its absolute value: replacing $g$ by $gk$
multiplies the coefficient by the $K$-character of $e_n$ at $k$.

Since $X=\Gamma\backslash G$ is compact, the centers range over
the compact quotient $X/K$. Reznikov's bound is uniform over
these centers: the only estimate in his proof that depends on
the center is the pointwise spectral mean-square bound
\cite[(4.10)]{R}, which is uniform on a compact hyperbolic
surface (passing to a fixed finite surface cover if necessary).
Taking the supremum over the center therefore gives
\[
   p_{\rm aut}(e_n)
   = \sup_{\Gamma g\in X} |\eta(g e_n)|
   \ll_{\epsilon,V,X} (1+|n|)^{1/3+\epsilon},
\]
as claimed.
\end{proof}

\begin{rmk} The conjectural estimate $|\eta(e_n)|\lesssim_\epsilon(1+|n|)^\epsilon$ in \cite{R} for every
$\epsilon>0$ for an automorphic functional $\eta$ suggests a strengthening to 
$p_{\rm aut}(e_n)\lesssim_\epsilon(1+|n|)^\epsilon$. This would imply $d([p_{\rm aut}],[p_{\max}])=\tfrac12$. Indeed, the same
comparison on $e_n$ gives the lower bound. For the upper bound,
use $p_{\rm aut}\lesssim p_{\max}$ and, for $s>\tfrac12$,
$p_{\max}\lesssim q_s\lesssim p_{{\rm aut},s}$, since
$q\lesssim p_{\rm aut}$. This argument does not determine
$d([p_{\min}],[p_{\rm aut}])$.
\end{rmk}

\subsection{Automorphic forms for \texorpdfstring{$G=\SL(2,\R)$}{G=SL(2,R)} - general lattices}

We will now exploit the Abstract Convexity Bound from Theorem \ref{thm.abst. conv. bnds} some more.
The requirement of cocompactness of the lattice $\Gamma$ is now dropped.

Let us assume for the moment that the automorphic functional $\eta$ is cuspidal. This ensures that each automorphic function $m_{v,\eta}$ is bounded and that $p_{\rm aut}(v):= \sup_{x \in X} |m_{v, \eta}(x)|$ defines an isometric $G$-continuous norm on $V$. 
For $G=\SL(2,\R)$, the relative-trace estimates of
Bernstein--Reznikov \cite[(3.7)--(3.8)]{BR} also apply on
noncompact quotients and record the height in the cusp.
For cuspidal unitary principal series a uniform sup-norm bound
is discussed in \cite[Appendix~A.2]{BR2}. The abstract convexity
bound of Theorem~\ref{thm.abst. conv. bnds} gives the following
extension to the irreducible unitary dual.

\begin{theorem}\label{thm: key automorphic}
Let $\Gamma$ be a lattice in $G=\SL(2,\R)$ and $\eta:V^\infty \to \C$ a $\Gamma$-invariant (automorphic) functional on an irreducible unitarizable Harish-Chandra module $V$ with $K$-spectrum $S=S(V)\subset \Z$. Let $(e_n)_{n \in S}$ be an orthonormal basis of $K$-types with regard to a unitary norm $q$. Let $\epsilon >0$.
\begin{enumerate}[(i)]
\item \label{key automorphic - item1}Suppose that $\eta$ is cuspidal. Then there exist constants $C_\e, C>0$ depending on $V, q, \eta$ such that 
$$p_{\rm aut}\leq C_\e\,  q_{\frac{1}{2}+\epsilon}$$
and 
$$
|\eta(e_n)|
\leq C (|n|+1)^{\frac{1}{2}}
\qquad (n\in S).
$$
\item Suppose that $\eta$ gives a realization in $L^r(X)$ for some $1\leq r \leq \infty$. Then there exists a constant $C_r=C_r(V,q,\eta,\e)>0$  such that
$$
\|m_{v,\eta}\|_{L^r(X)}
\leq C_r\,  q_{\frac{1}{2}+\epsilon}(v)
\qquad (v\in V^\infty).
$$
\end{enumerate}
\end{theorem} 

\begin{rmk} The first upper bound in \eqref{key automorphic - item1} follows from
Theorem~\ref{thm.abst. conv. bnds} because the cuspidal
sup-norm is an isometric globalization norm. For a nonzero
realization with $V$ nontrivial, the converse follows from
Proposition~\ref{prop:aut-finite-volume-bounds}, since
$\operatorname{GKdim}V=1$. This gives the equivalence in
Theorem~\ref{introthm: theorem automorphic}\textup{(i)}.

For finer results on automorphic
distributions for congruence subgroups, see Schmid~\cite{Sch2}.
\end{rmk}

\subsection{Automorphic norms attached to unitary minimal principal series}

We maintain the notation of Subsection \ref{subsection minimal principal series} for a real reductive group $G$, let $P=MAN$ be a minimal parabolic subgroup and $V=V_{\sigma, \lambda}$ be a Harish-Chandra module of the unitary minimal principal series with $\sigma \in \hat{M}$ and $\lambda\in i\af^*$. We denote by 
$q$ the standard unitary norm on $V_{\sigma,\lambda}$, i.e., the Hilbert globalization to $L^2(K,\sigma)$.

Let $X=\Gamma\bs G$ and $\eta\in(V^{-\infty})^\Gamma$.
Recall from Section~\ref{Sec: SR} that $(V,\eta)$ is $X$-bounded
if $m_{v,\eta}$ is bounded on $X$ for every $v\in V^\infty$.
For such a pair put
\[
       p_{\aut}(v)=\|m_{v,\eta}\|_{L^\infty(X)}.
\]
This is an invariant seminorm and is a norm when $T_\eta$ is
injective. Boundedness holds on compact quotients and for smooth
cuspidal realizations in the arithmetic setting of
\cite{MillerSchmidDecay}.

\begin{theorem}\label{thm Lorentz}
Let $V=V_{\sigma,\lambda}$ be a normalized minimal principal
series of $G$, with $\sigma\in\widehat M$, $\lambda\in i\af^*$,
and compact-picture unitary norm $q$. Let $\Gamma<G$ be a lattice
and $\eta\in(V^{-\infty})^\Gamma$ an automorphic functional.
Assume that $(V,\eta)$ is $X$-bounded, where $X=\Gamma\bs G$.
Then
\[
 p_{\rm aut}(v_\tau)
 \lesssim (1+|\tau|)^{c_{\gf}}q(v_\tau)
 \qquad(\tau\in\widehat K,\ v_\tau\in V[\tau]).
\]
On arbitrary smooth vectors one has
\begin{equation}\label{eq Sob dom intro2}
 p_{\rm aut}\lesssim q_s
 \qquad(s>c_{\gf}).
\end{equation}
\end{theorem}

\begin{proof}
By $X$-boundedness, $p_{\aut}$ is finite on $V^\infty$.
Even if it has a kernel, $q+p_{\rm aut}$ is an isometric
globalization norm, as in the proof of
Proposition~\ref{Prop X-bounded}. Hence
$p_{\rm aut}\leq q+p_{\rm aut}\lesssim p_{\max}$.
Apply Corollary~\ref{Cor mps convexity} with $R=0$:
its single-$K$-type estimate gives the first assertion, and its
full Sobolev estimate gives \eqref{eq Sob dom intro2} directly.
\end{proof}

The implicit constant in Theorem~\ref{thm Lorentz} may depend on
the realization. We now fix $X=\Gamma\bs G$ and the inducing representation $\sigma$ of 
$M$. We look for a bound where that is uniform in both the
realization and the imaginary parameter $\lambda$ on a compact
set. To fix the scale, we require $T_\eta$ to be an isometry
from the compact-picture Hilbert space into $L^2(X)$.

For the following proposition only, put
$\mathcal H_X=L^2(X)$ if $X$ is compact. If $X$ is noncompact,
assume the reduction hypotheses of \cite[\S 2]{LanglandsFE}
and put $\mathcal H_X=L^2_{\rm cusp}(X)$. The property we need
is that $R_X(f)$ is compact on $\mathcal H_X$ for every
$f\in C_c^\infty(G)$. On a compact quotient its integral kernel
is smooth on $X\times X$. In the noncompact case we use the
corollary to Lemma~3.1 in \cite[\S 3]{LanglandsFE}, which gives
compactness on the cuspidal subspace.

\begin{prop}\label{prop mps aut}
Fix $\sigma\in\widehat M$ and a compact set
$Q\subset i\af^*$. There is a constant $C=C(X,\sigma,Q)$ such
that for every $\lambda\in Q$ and every $X$-bounded
$L^2$-isometric realization $T_\eta$ of $V_{\sigma,\lambda}$
in $\mathcal H_X$ one has
\[
   \|T_\eta v_\tau\|_{L^\infty(X)}
   \leq C(1+|\tau|)^{c_{\gf}}q(v_\tau)
   \qquad
   (\tau\in\widehat K,\ v_\tau\in V_{\sigma,\lambda}[\tau]).
\]
Here $q$ is the compact-picture unitary norm.
\end{prop}

\begin{proof}
For every $f\in C_c^\infty(G)$, the operator $R_X(f)$ is compact
on $\mathcal H_X$. For compact $X$ this follows from the smoothness
of its integral kernel on $X\times X$. In the noncompact case this
is the corollary to Lemma~3.1 in \cite[\S 3]{LanglandsFE}.

We first show that only finitely many $\lambda\in Q$ admit an
isometric realization in $\mathcal H_X$. Choose a smooth unit
vector $u$ in the common compact picture. The map
$(\lambda,g)\mapsto\pi_{\sigma,\lambda}(g)u$ is continuous in
the Hilbert norm. Since $Q$ is compact, we may choose a
nonnegative $f\in C_c^\infty(G)$ of integral one, supported
sufficiently near the identity, such that
\[
   \|\pi_{\sigma,\lambda}(f)u-u\|<\tfrac12
   \qquad(\lambda\in Q).
\]
For fixed $\sigma$, the infinitesimal character determines
$\lambda$ up to finitely many possibilities, by the
Harish-Chandra parameter formula and finiteness of the Weyl group.
If infinitely many parameters occurred, we could therefore choose
isometric realizations $T_j$ with pairwise distinct infinitesimal
characters. Their images are mutually orthogonal. Thus $(T_j u)_j$
would be an orthonormal sequence, whereas
\[
   \|R_X(f)T_j u\|
   =\|T_j\pi_{\sigma,\lambda_j}(f)u\|>\tfrac12
\]
for every $j$. This contradicts compactness of $R_X(f)$.

Fix one of the remaining parameters and write $\Hc_\lambda$ for
the unitary completion of $V_{\sigma,\lambda}$. Each irreducible
representation occurs in $\mathcal H_X$ with finite multiplicity,
by the same compactness argument applied to orthogonal copies of
a single representation. Since $\Hc_\lambda$ has finite length,
the space of bounded $G$-intertwiners from $\Hc_\lambda$ to
$\mathcal H_X$ is finite dimensional.

Choose a basis $T_1,\ldots,T_m$ of the span of the isometric
realizations in the statement, chosen among those realizations.
By Theorem~\ref{thm Lorentz}, each $T_j$ satisfies the required
single-$K$-type estimate with a constant $C_j$. Write an arbitrary
isometric realization as $T=\sum_j a_jT_j$. Since
$\|T\|_{\rm op}=1$, equivalence of norms on this finite-dimensional
space bounds $\sum_j |a_j|C_j$ uniformly in $T$. Hence
\[
   \|Tv_\tau\|_{L^\infty(X)}
   \leq \sum_j |a_j|\,\|T_jv_\tau\|_{L^\infty(X)}
   \leq \Bigl(\sum_j |a_j|C_j\Bigr)
          (1+|\tau|)^{c_\gf}q(v_\tau).
\]
Taking the maximum over the finitely many occurring parameters
proves the proposition.
\end{proof}

\subsubsection[Automorphic Sobolev threshold]
  {Automorphic Sobolev threshold}
\label{subsec:aut-finite-volume-threshold}

\begin{definition}[Automorphic Sobolev threshold]
Let $\Gamma<G$ be a lattice, put $X=\Gamma\bs G$, and let $T_\eta$ be
a nonzero $X$-bounded automorphic realization of an irreducible
unitarizable Harish-Chandra module $V$. Thus $T_\eta v$ is bounded
on $X$ for every $v\in V^\infty$. Normalize the unitary norm by this
realization and write
\[
   q(v)=\|T_\eta v\|_{L^2(X)},\qquad
   p_{\aut}(v)=\|T_\eta v\|_{L^\infty(X)}.
\]
The \emph{automorphic Sobolev threshold} of the realization is
\[
   s_{\aut}(V,\eta)
   :=\inf\{s\geq0:p_{\aut}\lesssim q_s\}.
\]
Since $q\leq\vol(X)^{1/2}p_{\aut}$, the directed distance
$d_\to([q],[p_{\aut}])$ is zero. Hence
$s_{\aut}(V,\eta)=d([q],[p_{\aut}])$.
\end{definition}

\begin{prop}[Finite-volume bounds]\label{prop:aut-finite-volume-bounds}
Suppose that $G$ is a connected real reductive group and that $V$ in
the preceding definition is infinite dimensional. Then
\begin{equation}\label{automorphic sandwich}
   \tfrac12\operatorname{GKdim}V
   \ \leq\ s_{\rm aut}(V,\eta)\
   \leq\ \tfrac12\gap(V).
\end{equation}
On every such quotient, domination $p_{\aut}\lesssim q_s$ requires
$s>\tfrac12\operatorname{GKdim}V$. This strict inequality is also
sufficient if $X$ is compact or if $\gap(V)=\operatorname{GKdim}V$.
\end{prop}

\begin{proof}
For the lower bound we use the relative-trace argument of Bernstein
and Reznikov. The idea in \cite{BR} is to integrate the squared norm
of point evaluation over $X$. This turns a bound for the automorphic
realization into a trace condition on the representation.

Put $D=(1+\Delta_K)^{1/2}$, so that $q_s(v)=q(D^s v)$, and choose a
$q$-orthonormal basis $(e_j)$ adapted to the $K$-isotypes, with
$De_j=\lambda_j e_j$. For $x\in X$ let $\ell_x(v)=(T_\eta v)(x)$.
Since $(\lambda_j^{-s}e_j)$ is an orthonormal basis in the
$q_s$-completion, the squared norm of evaluation is
\[
   \|\ell_x\|_{q_s^*}^{\,2}
   :=\sup_{\substack{v\in V\\q_s(v)\leq1}}|\ell_x(v)|^2
   =\sum_j\lambda_j^{-2s}|(T_\eta e_j)(x)|^2,
\]
where both sides may be infinite. Tonelli's theorem and the
$L^2(X)$-normalization give
\[
   \int_X\|\ell_x\|_{q_s^*}^{\,2}\,dx
   =\sum_j\lambda_j^{-2s}\|T_\eta e_j\|_{L^2(X)}^2
   =\sum_j\lambda_j^{-2s}
   =\operatorname{Tr}(D^{-2s}).
\]
If $p_{\aut}\leq Cq_s$, then $\|\ell_x\|_{q_s^*}\leq C$ for every
$x$, since the functions $T_\eta v$ are continuous. Hence
$\operatorname{Tr}(D^{-2s})\leq C^2\vol(X)<\infty$.

On a compact quotient the converse holds at the same Sobolev order.
Indeed, if the trace is finite, the integral identity gives a point
$x_0\in X$ at which $\ell_{x_0}$ is $q_s$-bounded. The bound extends
from $V$ to $V^\infty$ by smooth $K$-finite approximation. Choose a
compact set $B\subset G$ such that $x_0B=X$. Equivariance gives
$\ell_{x_0g}(v)=\ell_{x_0}(\pi(g)v)$, and
Lemma~\ref{lem:unitary-K-Sobolev} gives
$\sup_{g\in B}q_s(\pi(g)v)\lesssim_{B,s}q_s(v)$. Thus
\[
   p_{\aut}(v)
   =\sup_{g\in B}|\ell_{x_0}(\pi(g)v)|
   \lesssim_{x_0,B,s}q_s(v).
\]
This is the compact relative-trace criterion
\cite[Theorem~1.1 and Proposition~A.2]{BR}. On a noncompact quotient,
finite volume alone does not turn the integral bound into a uniform
bound for point evaluation.

It remains to determine when the trace is finite. Put
$d=\operatorname{GKdim}V>0$ and
$N_V(R)=\#\{j:\lambda_j\leq R\}$. This counts each $K$-isotype with
its full dimension, not just once for each highest weight. Since the
eigenvalue of $D$ on $V[\tau]$ is comparable to $1+|\tau|$, Vogan's
$K$-type growth theorem \cite[Proposition~4.6]{VoganSize103} gives
$N_V(R)\asymp R^d$ for large $R$. For $s>0$, summing
$\lambda_j^{-2s}=2s\int_{\lambda_j}^\infty R^{-2s-1}\,dR$ gives
\[
   \operatorname{Tr}(D^{-2s})
   =2s\int_1^\infty N_V(R)R^{-2s-1}\,dR.
\]
The integral converges exactly when $2s>d$ and diverges
logarithmically at $2s=d$. For $s\leq0$ the trace also diverges,
since $D\geq1$ and $V$ is infinite dimensional. This proves the
strict necessary condition on every finite-volume quotient and its
sufficiency on a compact quotient.

For the upper bound, $p_{\aut}$ is an isometric globalization norm,
so $p_{\aut}\lesssim p_{\max}$. The midpoint formula
(Proposition~\ref{prop unitary}) therefore gives
\[
   p_{\aut}\lesssim p_{\max}\lesssim q_s
   \qquad\bigl(s>\tfrac12\gap(V)\bigr).
\]
Taking the infimum proves the upper bound in
\eqref{automorphic sandwich}. If $\gap(V)=d$, this estimate gives
domination for every $s>d/2$, while the trace excludes $s=d/2$.
\end{proof}
\begin{rmk}
For example, every nontrivial irreducible unitarizable Harish-Chandra
module of $\SL(2,\R)$ has $\operatorname{GKdim}V=\gap(V)=1$
(Theorem~\ref{thm ultimate} and its proof). Thus for every nonzero
cuspidal realization on a lattice quotient of $\SL(2,\R)$ one has
\[
   p_{\aut}\lesssim q_s
   \quad\Longleftrightarrow\quad s>\tfrac12 .
\]
\end{rmk}

\subsubsection[Compact quotients: the criterion in terms of K-types]
{Compact quotients: the criterion in terms of $K$-types}
Assume in this subsubsection that $X$ is compact. By
\S\ref{subsec:aut-finite-volume-threshold}, the Sobolev orders that
dominate $p_{\aut}$ are determined by the $K$-spectrum of $V$, with
multiplicities. We spell out the criterion for several families of
representations. To describe the cases in which the multiplicities
are at most one, we first recall the notion of a spherical pair.

A homogeneous space $K/H$ attached to a compact connected Lie group $K$ and closed algebraic subgroup $H$ is called spherical if its complexification $K_\C/H_\C$ is spherical. The latter means that one of the following equivalent conditions is satisfied:
\begin{itemize}
    \item There is a Borel subgroup $B_\C\subset K_\C$ such that 
$B_\C H_\C\subset K_\C$ is open.
\item There is a Borel subalgebra $\mathfrak{b}_\C\subset \kf_\C$ such that $\kf_\C=\hf_\C+\mathfrak{b}_\C$.
\item For every finite dimensional irreducible representation $V$ of $K_\C$ one has $\dim V^H\leq 1$. 
\end{itemize}
In particular, if $(K\times H)/\diag H $ is spherical, then $\dim \Hom_H(\sigma,\tau|_H)\leq1$ for all $\sigma\in \hat H$ and $\tau \in \hat K$.

In case $G$ is a connected reductive group of real rank one we recall that $K/M$ is spherical where $M=Z_K(\af)$. Moreover, among the simple real rank one algebras $\gf$ we have $(K\times M)/\diag (M)$ spherical if and only if  $\gf=\so(1,n)$ or $\gf=\su(1,n)$.
In these two cases the pairs $(\kf_\C,\mf_\C)$ are given by 
\begin{itemize}
    \item $(\so(n,\C), \so(n-1,\C))$ for $\gf=\so(1,n)$, $n\geq 2$,
    \item $(\gl(n,\C), \gl(n-1,\C))$ for $\gf=\su(1,n)$, $n\geq 2$.
\end{itemize}

For $\tau\in\hat K$ we abbreviate $d_\tau:=\dim \tau$. The next result explicates \cite[Th. 1.1]{BR}. 

\begin{prop}\label{prop BR}
Let $G$ be a connected real reductive group and $V$ a unitarizable irreducible Harish-Chandra module with $q$ a unitary norm. 
Let $V=\bigoplus_{\tau \in \hat K} V[\tau]$ be the isotypical decomposition and $m_\tau$ the multiplicity
of the $K$-type $\tau$. Let $\Gamma<G$ be a cocompact lattice, $\eta\in (V^{-\infty})^\Gamma$ be
a non-zero automorphic functional and $p_{\rm aut}$ the associated automorphic norm.
Then the following assertions hold true:
\begin{enumerate}[(i)]
\item\label{prop BR - item 1}
$p_{\rm aut}\lesssim q_s$ if and only if 
\begin{equation}\label{eq BR1} 
\sum_{\tau\in \hat K} \frac{m_\tau d_\tau}{ (1+|\tau|)^{2s}}<\infty.
\end{equation} 
In particular, if $V$ is infinite dimensional, then
\[
 p_{\rm aut}\lesssim q_s
 \quad\Longleftrightarrow\quad s>\tfrac12\operatorname{GKdim}V.
\]
For finite-dimensional $V$ the domination holds for every real $s$.
\item\label{prop BR - item 2}
Condition \eqref{eq BR1} is satisfied if 
\begin{equation} \label{eq BR2}
s> \frac{\dim K}{2}.
\end{equation}
\item\label{prop BR - item 3} If $(K\times M)/\diag (M)$ is spherical, then condition \eqref{eq BR1} is equivalent to 
\begin{equation} \label{eq BR3}
\sum_{\tau\in \Spec_K(V)} \frac{d_\tau}{ (1+|\tau|)^{2s}}
<\infty.
\end{equation}
\item\label{prop BR - item 4} Suppose that $\gf=\so(1,n)$ for $n\geq 2$ and that $V$ is non-trivial. Then condition \eqref{eq BR1} is satisfied if and only if  $s>\frac{n-1}{2}$. 
\item\label{prop BR - item 5} Suppose that $\gf=\su(1,n)$ for $n\geq 2$ and that $V$ is non-trivial.
Then the following assertions hold.
\begin{enumerate}[(a)]
\item\label{prop BR - item 5a}
If $V$ is a principal series $V_{\sigma,\lambda}$,
then \eqref{eq BR1} holds if and only if $s>n-\tfrac12$.
\item\label{prop BR - item 5b}
If $V$ is a generalized Verma module of holomorphic or
antiholomorphic type, then \eqref{eq BR1} holds if and only if
$s>\tfrac n2$. This includes the holomorphic and antiholomorphic
discrete series.
\item\label{prop BR - item 5c}
If $V$ is an intermediate discrete series of $\SU(1,n)$
(neither holomorphic nor antiholomorphic), then
$\operatorname{GKdim}V=2n-1$. Consequently, on the given compact
quotient, $p_{\rm aut}\lesssim q_s$ holds if and only if
$s>n-\tfrac12$.
\end{enumerate}
\end{enumerate}
\end{prop}

\begin{proof}Let $c_\tau$ be the eigenvalue of $\Delta_K$ on $V[\tau]$.
Since $\dim V[\tau]=m_\tau d_\tau$, the trace in the compact
criterion of \S\ref{subsec:aut-finite-volume-threshold} is
\[
   \operatorname{Tr}(D^{-2s})
      =\sum_{\tau\in\widehat K}
         m_\tau d_\tau(1+c_\tau)^{-s}.
\]
As $(1+c_\tau)^{1/2}\asymp1+|\tau|$, finiteness of this trace is
precisely \eqref{eq BR1}. For infinite-dimensional $V$, the same
subsection shows that this is equivalent to
$2s>\operatorname{GKdim}V$.

When $V$ is finite dimensional all Sobolev norms are equivalent.
For (\ref{prop BR - item 2}) we use Weyl's dimension formula to bound $d_\tau\lesssim (1+|\tau|)^{\frac{1}{2}(\dim K - \rank K)}$. According to Harish-Chandra we have $m_\tau\leq d_\tau$ in general. Hence \eqref{eq BR1} is satisfied if
$$
\sum_{\tau\in \hat K} \frac{(1+|\tau|)^{\dim K -\rank K}}{ (1+|\tau|)^{2s}}<\infty,
$$
that is $2s - \dim K +\rank K >\rank K$ which in turn is condition \eqref{eq BR2}.

\par For (\ref{prop BR - item 3}) we note that $(K\times M)/ \diag(M)$ is spherical implies  $m_\tau\leq 1$ for all $\tau$
as one can embed $V$ into a minimal principal series $V_{\sigma,\lambda}$ by Casselman's subrepresentation theorem and $V_{\sigma,\lambda} \simeq \C[K\times_M \sigma]$ as a $K$-module. Thus  \eqref{eq BR3} follows.

By the Gelfand--Kirillov criterion above, it remains only to identify
$\operatorname{GKdim}V$ in the cases at hand. For a minimal
principal series, the induced-module dimension formula gives
\[
   \operatorname{GKdim}V=\dim\overline N=\dim(K/M).
\]
Thus, if $\dim(K/M)>0$, condition~\eqref{eq BR1} holds precisely for
$s>\dim(K/M)/2$. If $\dim(K/M)=0$, the module is finite dimensional
and every real $s$ is admissible.

For~\textup{(\ref{prop BR - item 4})} we recall a bound on the Gelfand-Kirillov dimension that follows from its identification with the dimension of the associated variety of $V$. 

For a good filtration of $V$, its associated variety
$\operatorname{AV}(V)$ is the support of $\operatorname{gr}V$
as an $S(\gf)$-module \cite{VoganAssociatedVarieties}.
The cumulative Hilbert polynomial of $\operatorname{gr}V$ has
degree $\dim_\C\operatorname{AV}(V)$ by the Hilbert--Serre
theorem. By the definition of Gelfand--Kirillov dimension this
degree equals $\operatorname{GKdim}V$ for $V\ne0$.
If $V$ has an infinitesimal character, then
$\operatorname{AV}(V)\subseteq\pf_\C^*\cap\mathcal N$, where $\mathcal N$ is the nilpotent cone of $\gf^{*}_{\C}$.

Indeed, choose a $K$-stable good filtration $F_\bullet V$. Since
$\kf_\C F_nV\subseteq F_nV$, the degree-one symbols of $\kf_\C$
annihilate $\operatorname{gr}V$. Hence
\[
   \operatorname{AV}(V)\subseteq \kf_\C^\perp
   \simeq \pf_\C^*.
\]
Recall the center ${\mathcal Z}(\gf)$ of $\U(\gf)$ and let  $\chi$ be the infinitesimal character of $V$. For every
$z\in {\mathcal Z}(\gf)$ of positive filtration degree, the element
$z-\chi(z)1$ annihilates $V$ and has the same principal symbol
as $z$. Since
$\operatorname{gr}{\mathcal Z}(\gf)=S(\gf)^{\gf}$,
every homogeneous invariant polynomial of positive degree
annihilates $\operatorname{gr}V$. Their common zero set in
$\gf_\C^*$ is the nilpotent cone $\mathcal N$. Thus
\[
   \operatorname{AV}(V)\subseteq \pf_\C^*\cap\mathcal N.
\]

Suppose now that $\gf=\so(1,n)$, $n\geq 2$. Every nontrivial irreducible unitary module is infinite dimensional,
since the connected group has noncompact simple Lie algebra.
Its associated variety is therefore a nonzero closed conical
$K_{\C}$-stable subset of the nilpotent cone in $\pf_{\C}^*$.
Under the identification $\pf_{\C}^*\simeq\C^n$, this cone is the
isotropic quadric
\[
   \{z\in\C^n:z_1^2+\cdots+z_n^2=0\}.
\]
For $n\geq3$, the nonzero isotropic vectors form a single
$\SO(n,\C)$-orbit, so the associated variety is the entire quadric.
For $n=2$, the quadric is the union of two lines, and every nonzero
closed conical subset of it has dimension one. In either case
    $\operatorname{GKdim}V=n-1$, giving the threshold $s>(n-1)/2$.

For~\textup{(\ref{prop BR - item 5a})}, take $\gf=\su(1,n)$.
Here $\dim(K/M)=2n-1$, so the minimal principal series threshold is
$s>n-\tfrac12$. For~\textup{(\ref{prop BR - item 5b})}, let $W$ be
the inducing $K$-type of a holomorphic generalized Verma module.
The opposite summand $\pf^-$ in the Hermitian decomposition
$\pf_{\C}=\pf^+\oplus\pf^-$ is abelian of complex dimension $n$.
By the Poincar\'e--Birkhoff--Witt theorem, the filtration generated
by $W$ has associated graded module
\[
   \operatorname{gr}V\simeq S(\pf^-)\otimes W
   \qquad\text{as a graded vector space}.
\]
Hence $\operatorname{GKdim}V=n$, and the threshold is $s>n/2$.
The antiholomorphic case follows by interchanging $\pf^+$ and
$\pf^-$. Here it matters that $V$ is the generalized Verma module
itself: a proper quotient can have smaller Gelfand--Kirillov dimension.

\par For~\textup{(\ref{prop BR - item 5c})}, put $E=\C^n$.
Using the trace pairing and the two off-diagonal blocks, identify
$\pf_\C^*$ with $E\oplus E^*$. In an intermediate chamber, the
sum of the negative noncompact root spaces has the form
\[
   E_j\oplus\operatorname{Ann}(E_j),
   \qquad E_j\subset E,\quad \dim E_j=j,\quad 0<j<n.
\]
By \cite[Theorem~3.1]{YamashitaDS}, $\operatorname{AV}(V)$ is
the closure of the $K_\C$-saturation of this subspace.

Since $K_\C$ acts transitively on the $j$-planes in $E$, this
saturation consists of the pairs $(x,\xi)$ for which some $j$-plane
$L$ satisfies $\C x\subseteq L\subseteq\ker\xi$. Such a pair
satisfies $\xi(x)=0$. Conversely, if $\xi(x)=0$, then
$\C x\subseteq\ker\xi$, and
\[
   \dim(\C x)\leq 1\leq j\leq n-1\leq\dim\ker\xi.
\]
We may therefore extend $\C x$ to a $j$-plane inside $\ker\xi$.
Thus the saturation is already closed and
\[
   \operatorname{AV}(V)
   =\{(x,\xi)\in E\oplus E^*:\xi(x)=0\}.
\]
This hypersurface has dimension $2n-1$. Hence
$\operatorname{GKdim}V=2n-1$, and part~\textup{(\ref{prop BR - item 1})}
gives $p_{\rm aut}\lesssim q_s$ exactly for $s>n-\tfrac12$.
\end{proof}

\begin{rmk}
The reasoning  in (\ref{prop BR - item 5b}) extends to any
connected simple Lie group of Hermitian type with finite center.
For a generalized Verma module
of holomorphic or antiholomorphic type one obtains that \eqref{eq BR1} holds if and only if $s>\dim_\C(G/K)/2$.
The qualification that $V$ be the generalized Verma module
itself matters here: in a proper quotient the dimensions of
the homogeneous pieces can be smaller.
\end{rmk}

\begin{cor}[Distance from the maximal norm]\label{cor:aut-max-distance}
Under the hypotheses of Proposition~\ref{prop BR}, one has
\[
   d([p_{\max}],[p_{\rm aut}])
   \geq \tfrac12\bigl(\gap(V)-\operatorname{GKdim}V\bigr).
\]
\end{cor}

\begin{proof}
Since $q$ is proportional to the $L^2$-norm of the realization
and $\Gamma\bs G$ has finite volume, $q\lesssim p_{\aut}$. Hence
\[
   d([q],[p_{\rm aut}])
   =d_{\to}([p_{\rm aut}],[q])
   =s_{\aut}(V,\eta)
   =\tfrac12\operatorname{GKdim}V.
\]
On the other hand, Proposition~\ref{prop unitary} gives
$d([q],[p_{\max}])=\tfrac12\gap(V)$. The assertion follows from
the triangle inequality.
\end{proof}

\begin{rmk}\label{rmk automorphic}
Suppose that $G$ is connected semisimple with finite center and that
$V=V_{\sigma,\lambda}$ is an irreducible minimal principal
series with $\lambda\in i\af^*$ occurring on a compact quotient.
In this case the two thresholds are given by two dimensions of $N$:
the abstract threshold $c_{\gf}=\tfrac12\gap(V)$ is half its
homogeneous dimension, whereas the cocompact automorphic threshold
$s_{\aut}(V)=\tfrac12\dim N$ is half its ordinary dimension.
Indeed, Theorem~\ref{thm:flag-gap} gives $\gap(V)=2c_{\gf}$, and
$\operatorname{GKdim}V=\dim N=\dim(K/M)$. Thus
Corollary~\ref{cor:aut-max-distance} gives
\[
   d([p_{\max}],[p_{\aut}])
   \geq c_{\gf}-\tfrac12\dim N
   =\tfrac12\sum_{\alpha>0}
      m_\alpha\bigl(\operatorname{ht}(\alpha)-1\bigr).
\]
In particular, the two norms have positive Sobolev distance whenever
some positive restricted root has height greater than one. For
$\SL(n,\R)$ the lower bound is $n(n-1)(n-2)/12$, and for
$\SU(1,n)$, $n\geq2$, it is $1/2$.

For $\gf=\so(1,n)$ the two thresholds agree. However, this does not imply
that $d([p_{\max}],[p_{\aut}])=0$: already for the $K$-spherical
unitary principal series of $\SL(2,\R)$ on a compact quotient,
Proposition~\ref{prop upper gap} gives a lower bound of $1/6$.
\end{rmk}

\subsubsection{Regular Fourier coefficients and the upper threshold}

The upper bound in \eqref{automorphic sandwich} can be attained on a noncompact quotient even
when it strictly exceeds the cocompact threshold.

\begin{prop}[Regular Fourier coefficients]
\label{prop:regular-fourier-threshold}
Let $G$ be a connected linear real semisimple group with finite
center, and let $P=MAN$ be a minimal parabolic subgroup.
Let $V=V_{\sigma,\lambda}$ be an irreducible minimal 
principal series, with $\sigma\in\widehat M$ and $\lambda\in i\af^*$.
Let $T_\eta$ be a nonzero $X$-bounded automorphic realization
on $X=\Gamma\bs G$, where $\Gamma$ is a lattice.

Suppose that $\Gamma_N=\Gamma\cap N$ is a lattice in $N$.
Let $\psi$ be a unitary character of $N$, trivial on $\Gamma_N$,
whose differential is nonzero on every simple restricted root space.
If the functional
\[
 W_\psi(v)=\int_{\Gamma_N\bs N}
       (T_\eta v)(\Gamma n)\,\overline{\psi(n)}\,dn
\]
is nonzero on $V^\infty$, then
\[
                 s_{\rm aut}(V,\eta)=c_\gf=\frac12\gap(V).
\]
\end{prop}

\begin{proof}
Give $\Gamma_N\bs N$ probability Haar measure. Since right
translation preserves the sup-norm,
\[
       |W_\psi(\pi(g)v)|\leq p_{\aut}(v)
       \qquad(g\in G,\ v\in V^\infty).
\]
We use this inequality to obtain the lower bound.

Choose a representative $w_0$ of the longest restricted Weyl
group element. On the open cell $Nw_0P/P$ write $F_v(n)=v(nw_0)$
in the induced model of Section~\ref{subsection minimal principal series}.
Every $F\in C_c^\infty(N,U_\sigma)$ defines a smooth vector by
extension by zero. By \cite[Lemma~13.4(c)]{vdBWhittaker}, a Whittaker
functional for a regular character is determined by its restriction
to this cell. Thus the restriction of $W_\psi$ is nonzero.
Since $F_{\pi(u)v}(n)=F_v(u^{-1}n)$ and
$W_\psi(\pi(u)v)=\psi(u)W_\psi(v)$ for $u\in N$, it has the form
\[
        W_\psi(v)=\int_N b(F_v(n))\,\psi(n)\,dn
        \qquad(F_v\in C_c^\infty(N,U_\sigma))
\]
for some nonzero $b\in U_\sigma^*$.

Take $X_0$ as in \eqref{def str const}, so that
$\alpha(X_0)=1$ for every simple restricted root. For $T\geq1$
put $a_T=\exp((\log T)X_0)$ and $\delta_T(n)=a_Tna_T^{-1}$.
The Haar Jacobian of $\delta_T$ is $a_T^{2\rho}=T^{2c_\gf}$.
Since $w_0X_0=-X_0$, normalized induction gives
\[
 F_{\pi(a_T)v}(n)=a_T^{-(\lambda+\rho)}F_v(\delta_T^{-1}n),
 \qquad |a_T^{-(\lambda+\rho)}|=T^{-c_\gf}.
\]
Choose $u\in U_\sigma$ and $h\in C_c^\infty(N)$ with
$b(u)=1$ and $\int_N h(n)\,dn=1$. Define $v_T$ by
\[
             F_{v_T}(n)=h(n)\overline{\psi(\delta_T n)}\,u .
\]
Changing variables $n=\delta_T m$ in the open-cell integral
cancels the two characters and gives
\[
       p_{\aut}(v_T)\geq |W_\psi(\pi(a_T)v_T)|=T^{c_\gf}.
\]

The gain $T^{c_\gf}$ is the product of the Haar Jacobian $T^{2c_\gf}$ and the induction factor $T^{-c_\gf}$. The Sobolev cost, by contrast, depends only on the dilation on the abelianization. The
abelianization of $\mathfrak n$ is the sum of the simple restricted
root spaces, on which $\delta_T$ acts by $T$. Writing
$\psi(\exp Y)=e^{i\ell(Y)}$, we therefore have
$\psi(\delta_T(\exp Y))=e^{iT\ell(Y)}$. The support of $h$ is
fixed, so every derivative of order $k$ in the compact picture
is $O_k(T^k)$ there. It follows that $q_{2k}(v_T)\ll_k T^{2k}$
for every integer $k\geq0$, and Hilbert interpolation (see the proof of Lemma \ref{lem:unitary-K-Sobolev}) gives
$q_s(v_T)\ll_s T^s$ for all $s\geq0$.
Thus $p_{\aut}\lesssim q_s$ forces $s\geq c_\gf$.
Theorem~\ref{thm Lorentz} gives domination for every $s>c_\gf$,
which proves the assertion about the infimum.
\end{proof}

\begin{cor}[The level-one spherical case]
\label{cor:sln-cuspidal-threshold}
Let $G=\SL(n,\R)$ and $\Gamma=\SL(n,\Z)$, with $n\geq2$.
If an irreducible spherical minimal principal series with
imaginary parameter has a nonzero cuspidal realization $T_\eta$
on $\Gamma\bs G$, then
\[
             s_{\rm aut}(V,\eta)=\frac{n(n^2-1)}{12}.
\]
\end{cor}

\begin{proof}
Let $v_0$ be a nonzero $K$-fixed vector. Its image under $T_\eta$
is a nonzero Maass cusp form. By the Fourier expansion in
\cite[Theorem~5.3.2]{GoldfeldGLn}, some nondegenerate Fourier
coefficient of this form is nonzero at some $g\in G$.
Equivalently, $W_\psi(\pi(g)v_0)\ne0$ for a regular character
of the upper unitriangular group $N$. Here $\Gamma\cap N$ is a
lattice in $N$, and rapid decay gives $X$-boundedness
\cite{KroetzOpdamCrown,MillerSchmidDecay}. Proposition~\ref{prop:regular-fourier-threshold}
applies. There are $n-j$ positive roots of height $j$, so
\[
           c_\gf=\frac12\sum_{j=1}^{n-1}j(n-j)
                =\frac{n(n^2-1)}{12}.
\]
\end{proof}

Regular imaginary spherical cuspidal parameters occur at level one by the joint Weyl law
of Finis and Lapid \cite[Theorem 5.11]{FinisLapidWeyl}, applied to a dilating region away from the Weyl walls.
For the scalar Weyl law see M\"uller \cite{Mueller}.

\section{\texorpdfstring{Topology and geometry of $(\Noc(V,w),d)$}{Topology and geometry of (Norm(V,w),d)}}\label{Appendix B}
In this section we make first investigations into the topology and geometry of the pseudometric space $(\Noc(V,w),d)$. 

\subsection{The pseudometric is not metric}

We will show that $(\Noc(V,w),d)$ is not a metric space in general.

\begin{lemma} Let $V$ be a unitarizable Harish-Chandra module with unitary norm $q$ and $\gap(V)>0$. Then there exist $[p]\in \Noc(V)$ with 
$d([p],[q])=0$ and $[p]\neq [q]$.
\end{lemma}
\begin{proof} On the Hilbert globalization of $V$ put $T=(1+\log D)^{-1}$ and set
\[
   p(v)=\sup_{g\in G}q(T\pi(g)v).
\]
Since $D\geq1$, the operator $T$ is an injective contraction.
Thus $q(Tv)\leq p(v)\leq q(v)$, and $p$ is a $G$-invariant norm.
The action extends isometrically to the $p$-completion and is
strongly continuous there by $p\leq q$ and density. Each $K$-isotypic
projection has the same finite-dimensional range in this completion
as in $V$, so $p$ is a globalization norm.

For every $\epsilon>0$ one has
$t^{-\epsilon}\lesssim_\epsilon (1+\log t)^{-1}$ for $t\geq1$.
Hence $q_{-\epsilon}(v)\lesssim_\epsilon q(Tv)\leq p(v)$.
Applying this to $D^\epsilon v$ gives $q\lesssim_\epsilon p_\epsilon$.
Together with $p\leq q$, this proves $d([p],[q])=0$.

To see that $p$ and $q$ are inequivalent, let $m$ be an isometric
minimal norm. By Proposition~\ref{prop unitary},
$d([m],[q])=\gap(V)/2>0$, so there are $v_j\in V$ with
$q(v_j)=1$ and $m(v_j)\to0$. Fix $a>\gap(V)/2$. The midpoint
comparison gives $q_{-a}\lesssim m$ on $V^\infty$.
For every $\delta>0$ there is a constant $C_\delta$ such that
\[
   (1+\log t)^{-1}\leq \delta+C_\delta t^{-a}
   \qquad(t\geq1).
\]
By the spectral theorem,
$q(Tu)\leq\delta q(u)+C_\delta q_{-a}(u)$.
Take $u=\pi(g)v_j$ and use the invariance of $q$ and $m$. Taking
the supremum over $g$ gives $p(v_j)\leq\delta+C'_\delta m(v_j)$.
Thus $\limsup_j p(v_j)\leq\delta$ for every $\delta>0$, and
$p(v_j)\to0$. Since $q(v_j)=1$, the two norm classes are distinct.
\end{proof} 
Let $\mathbf{Norm}(V,w)$ be the metric quotient of $\Noc(V,w)$.

\subsection{Completeness}

\begin{prop}\label{thm:sobolev-norm-geometry}
If $\No(V,w)\ne\varnothing$, the metric space
$\mathbf{Norm}(V,w)$ is complete.
\end{prop}

\begin{proof}
Work on the common smooth globalization $E=V^\infty$.
We may assume every norm $p$ under consideration satisfies
$p(gv)\leq w(g)p(v)$ on $E$: replace it by the equivalent
norm $\sup_g p(gv)/w(g)$ and use submultiplicativity of $w$.
Comparisons of Sobolev norms on $V$ extend to $E$ by density.
We will use the following observation. If a norm $r$ on $E$ is
continuous for the smooth topology and satisfies this same
growth bound, it defines a norm in $\No(V,w)$ on restriction to
$V$. The action extends to the completion, and the orbit maps
are continuous there by density and local boundedness of $w$.
Each compact isotypic projection is bounded in $r$ and maps
the dense subspace $V$ into a finite-dimensional space. Its
range in the completion is therefore exactly the corresponding
isotypic component of $V$. The continuous equivariant map
from $E$ to the completion also identifies the derived action
on these vectors with the given one.

Take a Cauchy sequence in the metric quotient and pass to a
subsequence $([p_j])_{j\geq1}$ with
$d([p_j],[p_i])<\varepsilon_j=2^{-j}$ for all $i\geq j$.
For each $n$, choose $c_{ni}>0$ so that
\[
 Q_n(v):=\sum_{i\geq n}c_{ni}p_i(v)
 \leq (p_n)_{\varepsilon_n}(v) \qquad(v\in E).
\]
This is possible because $p_i\lesssim(p_n)_{\varepsilon_n}$
for every $i\geq n$: choose $c_{ni}$ so the $i$-th
summand is at most $2^{-(i-n+1)}(p_n)_{\varepsilon_n}$.
Then $Q_n$ is a continuous norm on $E$ with the same growth
bound. Put $L_j=(p_j)_{-\varepsilon_j}$. For $n\geq j$,
the reverse comparison gives $L_j\lesssim p_n\lesssim Q_n$,
by precomposing with $D^{-\varepsilon_j}$.
Choose $A_n>0$ so that $A_nQ_n\geq L_j$ for all $j\leq n$.
For $n<j$, the sum $Q_n$ contains a positive multiple of $p_j$,
and $L_j\lesssim p_j$ by Sobolev monotonicity. Taking the minimum
of the finitely many additional comparison constants shows
that for each $j$ there is $b_j>0$ with
$A_nQ_n\geq b_jL_j$ for all $n$.

Let $r:=\inf_{n\in \N} A_n Q_n$. The triangle inequality
for $L_j$ gives $r\geq b_jL_j$ for every $j$, while a one-term
decomposition gives $r\leq A_nQ_n$ for every $n$. In particular
$r$ is a continuous norm on $E$. Applying $g$ to a decomposition
shows $r(gv)\leq w(g)r(v)$, so the observation above applies.
Now
\[
 (p_j)_{-\varepsilon_j}\lesssim r\lesssim(p_j)_{\varepsilon_j}
 \quad\Longrightarrow\quad d([r],[p_j])\leq\varepsilon_j .
\]
Thus the subsequence, and hence the full Cauchy sequence,
converges to $[r]$.\end{proof}

\subsection{Contractibility}
\label{app:sobolev-norm-contraction}
We show
that the weighted infimum construction interpolates continuously
between the identity on this space and the constant map to the
maximal norm. In particular, $\mathbf{Norm}(V,w)$ is a contractible space. 

\begin{theorem}\label{thm:sobolev-norm-contraction}
Let $G$ be a real reductive group, $V$ a Harish-Chandra module,
and $w$ a weight with $\No(V,w)\ne\varnothing$. Put
$\Phi_s([p])=[p_s^{G,w}]$ for $s\geq 0$. The map
\[
 [0,\infty)\times\mathbf{Norm}(V,w)\longrightarrow
 \mathbf{Norm}(V,w),\qquad (s,[p])\longmapsto\Phi_s([p]),
\]
is well defined and satisfies
\begin{equation}\label{eq:sobolev-norm-joint-lip}
 d\bigl(\Phi_s([p]),\Phi_t([q])\bigr)
 \leq |s-t|+d([p],[q]).
\end{equation}
For any $S>\gap(V,w)$, the homotopy
$(\lambda,[p])\mapsto\Phi_{\lambda S}([p])$, $0\leq\lambda\leq 1$,
is a strong deformation retraction onto $[p_{\max}]$. 
\end{theorem}

\begin{proof}
We first work with equivalence classes in $\Noc(V,w)$, before
identifying points at distance zero. Put $D=(1+\Delta_K)^{1/2}$
and $r_s=p_s^{G,w}$. Monotonicity gives $r_s\lesssim r_t$
for $s\leq t$. We prove the time estimate
\begin{equation}\label{eq:sobolev-norm-time-lip}
 d([r_s],[r_t])\leq t-s \qquad (0\leq s\leq t).
\end{equation}
If $s=0$, this follows from $r_0\sim p$ and
$p\lesssim r_t\lesssim p_t$. The case $s=t$ is immediate.

For $0<s<t$ we pass to the restricted dual. Write
$P=\widetilde p$, $R_a=\widetilde r_a$ for $a\geq0$, and
$w^\sharp(g)=w(g^{-1})$. For fixed $g$, take the dual of $v\mapsto w(g^{-1})p(D^a\pi(g)v)$ on the smooth globalization. The change of variables $z=D^a\pi(g)v$ is bijective there, and the contragredient pairing gives
\[
 \langle v,u\rangle=\langle\pi(g^{-1})D^{-a}z,u\rangle
 =\langle z,D^{-a}\widetilde\pi(g)u\rangle.
\]
The fixed-$g$ dual norm is therefore $w(g^{-1})^{-1}P(D^{-a}\widetilde\pi(g)u)$. Taking the supremum over $g$ as in Remark~\ref{rmk sup of seminorms} gives

\begin{equation}\label{eq:sobolev-norm-dual-envelope}
 R_a(u)=\sup_{g\in G}\frac{P(D^{-a}(g\cdot u))}{w^\sharp(g)}
 \qquad(u\in\widetilde V).
\end{equation}
We also have $P(g\cdot u)\lesssim w^\sharp(g)P(u)$ by duality.
In view of Lemma~\ref{Lemma duality} our claim \eqref{eq:sobolev-norm-time-lip}
is equivalent to 
\begin{equation}\label{eq:sobolev-norm-time-lipbig}
 d([R_s],[R_t])\leq t-s \qquad (0\leq s\leq t)
\end{equation}
which we will show now.
First suppose that the $D$-spectrum of $u$ is contained in
$[2^{j-1},2^{j+1}]$. Then $S_j u=u$. Applying
Lemma~\ref{lem:K-dyadic} to $D^{-t}u$ and using the $g=\1$ term
of \eqref{eq:sobolev-norm-dual-envelope}, we get
\begin{equation}\label{eq:sobolev-norm-input-band}
 P(u)=P(D^t S_j D^{-t}u)
 \lesssim_t 2^{jt}P(D^{-t}u)
 \lesssim_t 2^{jt}R_t(u).
\end{equation}

For each $g\in G$, decompose $g\cdot u$ into the smooth blocks
$B_k(g\cdot u)$. Although $g\cdot u$ need not be $K$-finite,
it is smooth, so the expansion converges in the smooth topology
by Proposition~\ref{prop rapid decay}. After applying $D^{-s}$,
the tail converges absolutely in $P$, since
$P(D^{-s}B_k(g\cdot u))\lesssim_s 2^{-ks}P(g\cdot u)$.
For $k\le j$ use $D^{-s}B_k=D^{t-s}B_kD^{-t}$ and for $k>j$
apply Lemma~\ref{lem:K-dyadic} with exponent $-s$. This gives
\[
 \frac{P(D^{-s}(g\cdot u))}{w^\sharp(g)}
 \lesssim_{s,t}
 \sum_{k=0}^j 2^{k(t-s)}R_t(u)
 +\sum_{k>j}2^{-ks}P(u)
 \lesssim_{s,t} 2^{j(t-s)}R_t(u),
\]
where the last inequality uses \eqref{eq:sobolev-norm-input-band}.
Taking the supremum over $g$ yields
$R_s(u)\lesssim_{s,t}2^{j(t-s)}R_t(u)$ on this band.

Now let $u\in\widetilde V$ be arbitrary and write $u=\sum_{j\geq0}B_j u$.
This sum is finite since $u$ is $K$-finite. Fix $\varepsilon>0$ and
put $b=t-s+\varepsilon$. Lemma~\ref{lem:K-dyadic}, applied in the
norm $R_t$ to $B_j u=D^{-b}B_jD^b u$, gives
$R_t(B_j u)\lesssim_b 2^{-jb}R_t(D^b u)$. Hence
\[
 R_s(u)\lesssim_{s,t}\sum_{j\geq0}2^{j(t-s)}R_t(B_j u)
 \lesssim_{s,t,\varepsilon}
 \biggl(\sum_{j\geq0}2^{-j\varepsilon}\biggr)R_t(D^b u).
\]
Letting $\varepsilon$ decrease to zero in the definition of the
directed distance yields that $d_\to([R_s],[R_t])\leq t-s$.
The reverse directed distance is zero since $r_s\lesssim r_t$.

For variation of the norm, let $a>d([p],[q])$. Then
$p\lesssim q_a$. Taking the $s$-shift and the weighted
infimum gives $p_s^{G,w}\lesssim q_{s+a}^{G,w}$, by
\eqref{s plus t} and monotonicity of the infimum construction.
By the time estimate,
$d_\to(\Phi_s([p]),\Phi_s([q]))\leq a$.
The symmetric argument gives $d(\Phi_s([p]),\Phi_s([q]))\leq a$.
Letting $a$ decrease to $d([p],[q])$ proves that each
$\Phi_s$ is nonexpansive and descends to the metric quotient.
Combining this bound with \eqref{eq:sobolev-norm-time-lip}
proves \eqref{eq:sobolev-norm-joint-lip}.

Finally, $\Phi_0$ is the identity, and Proposition
\ref{Prop stabilization 1} gives $\Phi_S([p])=[p_{\max}]$
for every $[p]$ if $S>\gap(V,w)$. Also
$\Phi_s([p_{\max}])=[p_{\max}]$ for every $s\geq0$, since
$p_{\max}\lesssim (p_{\max})_s^{G,w}$ by monotonicity and the
opposite comparison holds by maximality. This proves the
strong deformation retraction.
\end{proof}

\subsection{Geodesics}

We now show that $\mathbf{Norm}(V,w)$ is a geodesic space. For norms $p,q$ on $E=V^\infty$, put
\[
 (p\vee q)(v)=\max\{p(v),q(v)\},\qquad
 (p\wedge q)(v)=\inf_{u+z=v}\bigl(p(u)+q(z)\bigr).
\]
We use the same notation for the induced operations on metric
classes, as justified below.

\begin{theorem}\label{thm:sobolev-norm-geodesic}
Let $x,y\in\mathbf{Norm}(V,w)$ and put $b=d(x,y)$.
With $\Phi_s([p])=[p_s^{G,w}]$ as in
Theorem~\ref{thm:sobolev-norm-contraction}, the path
\[
 \gamma_{x,y}(t)
   =\bigl(\Phi_{tb}(x)\wedge y\bigr)
      \vee\bigl(\Phi_{(1-t)b}(y)\wedge x\bigr),
 \qquad 0\leq t\leq1,
\]
is a constant-speed geodesic from $x$ to $y$. In particular,
\[
 d\bigl(\gamma_{x,y}(t),\gamma_{x,y}(u)\bigr)
     = |t-u|\,d(x,y) \qquad(0\leq t,u\leq1).
\]
\end{theorem}

\begin{proof}
Choose representatives satisfying $p(gv)\leq w(g)p(v)$ and
$q(gv)\leq w(g)q(v)$. Both operations preserve this growth bound
and give continuous norms on $E$. Indeed, $p$ and $q$ dominate
a common positive multiple of $p_{\min}$, and hence so does
$p\wedge q$. The completion criterion at the start of the proof
of Proposition~\ref{thm:sobolev-norm-geometry} therefore applies.

Both operations are nonexpansive for the maximum of the two input
distances. To see this, take $a>\max\{d([p],[p']),d([q],[q'])\}$.
For either operation $\star\in\{\vee,\wedge\}$, the comparisons
$p\lesssim p'_a$ and $q\lesssim q'_a$ give
\[
 p\star q\lesssim p'_a\star q'_a=(p'\star q')_a.
\]
The equality holds because $D^a$ is a linear bijection of $E$.
Interchanging the primed and unprimed norms and taking the infimum
over $a$ proves the claim. In particular, both operations descend
to the metric quotient.

Write $x=[p]$ and $y=[q]$. For $s>b$ we have $q\lesssim p_s$,
so Lemma~\ref{lemma Sobolev domination with and without infimum construction}%
\textup{(ii)} gives $q\lesssim p_s^{G,w}$. Hence
$\Phi_s(x)\wedge y=y$. By the time estimate
\eqref{eq:sobolev-norm-joint-lip} of
Theorem~\ref{thm:sobolev-norm-contraction} and continuity of
$\wedge$, we may let $s\downarrow b$ to obtain $\Phi_b(x)\wedge y=y$ in the metric quotient. This does not assert $q\lesssim p_b$. Interchanging $x$ and $y$ gives $\Phi_b(y)\wedge x=x$. Since $\Phi_0$ is the identity, absorption gives
\[
 \gamma_{x,y}(0)=(x\wedge y)\vee x=x,\qquad
 \gamma_{x,y}(1)=y\vee(y\wedge x)=y.
\]
Write $\gamma=\gamma_{x,y}$. The time estimate bounds the distance between $\Phi_{tb}(x)$ and $\Phi_{ub}(x)$ by $b|t-u|$. Taking the meet with $y$ cannot increase it. The second branch has the same bound, since $|(1-t)b-(1-u)b|=b|t-u|$. Their join is nonexpansive for the maximum of the two distances, so $d(\gamma(t),\gamma(u))\le b|t-u|$.

For $0\le t\le u\le1$, apply this upper bound to the two outer segments in the triangle inequality:
\[
 b=d(x,y)\le d(x,\gamma(t))+d(\gamma(t),\gamma(u))+d(\gamma(u),y)
 \le bt+d(\gamma(t),\gamma(u))+b(1-u).
\]
Subtracting gives $d(\gamma(t),\gamma(u))\ge b(u-t)$, hence equality. Symmetry covers $u<t$; when $b=0$ the upper bound alone says that the path is constant.

\end{proof}

\appendix
\section{\texorpdfstring{Estimates of matrix-coefficients for $\SL(2,\R)$}{Estimates of matrix-coefficients for SL(2,R)}}\label{Sec9: key estimates}

We prove the individual $K$-type estimates of
Proposition~\ref{thm upper-lower} and the discrete series averaging
estimate of Proposition~\ref{thm estimate convex combination}.

\subsection{Proof of Proposition \ref{thm upper-lower} }\label{subsection Proof for Theorem upper-lower}

It suffices to prove Proposition~\ref{thm upper-lower}(\ref{thm upper-lower - item 1}) on $V$ and on its
contragredient. Indeed, the projection to a single $K$-type for $K=\SO(2)$ is integration against a
character of modulus one, and so is contractive for an isometric norm. Since the weight space has
dimension one, the restricted dual norm on the paired unit vector is exactly the reciprocal of the
original norm on that line. Choosing dual extremal representatives proves (\ref{thm upper-lower - item
2}), and equivalent choices preserve the estimates.

\subsubsection{Unitary principal series and complementary series representations}

Write $u_\theta=(\cos\theta,\sin\theta)$. In the compact picture
$V_{\sigma,\lambda}^\infty$ consists of the smooth functions on the
circle satisfying $f(\theta+\pi)=\sigma(-1)f(\theta)$. The action
is obtained by extending $f$ to $\R^2\setminus\{0\}$ by
$f(ru_\theta)=r^{-1-\lambda}f(u_\theta)$ for $r>0$, and then
acting by pullback under $g^{-1}$. The $K$-types have basis
\[
             \phi_n^\lambda(u_\theta)=e^{in\theta},
             \qquad (-1)^n=\sigma(-1)
\]
and accordingly the $K$-spectrum is given by 
$$S(V_{\sigma,\lambda})=\begin{cases} 2\Z &\hbox{for $\sigma=1$}\\
2\Z+1 & \hbox{for $\sigma=-\1$}\end{cases}$$

We normalize the invariant bilinear pairing by
\begin{equation}\label{pairing nondeg}
 (h,f)=\int_0^{2\pi}h(u_\theta)f(u_\theta)\,
                       \frac{d\theta}{2\pi},
 \qquad
 h\in V_{\sigma,-\lambda}^\infty,\quad
 f\in V_{\sigma,\lambda}^\infty .
\end{equation}
Thus $(\phi_m^{-\lambda},\phi_n^\lambda)=\delta_{m,-n}$.
For imaginary $\lambda$ the compact $L^2$-norm is unitary.
For the complementary series, that is $\sigma=\1$ and $\lambda\in (-1,1)\bs\{0\}$,  we will compute the norms of
these same Fourier vectors in the invariant scalar product.

The following lemma estimates the matrix coefficients used to bound
$p_{\min}$ on individual Fourier vectors.

\begin{lemma}\label{lem:circle-fixed-test}
Let $-1<\re\lambda<1$ and fix $m\in S(V_{\sigma,\lambda})$.
For $a_t=\diag(t,t^{-1})$, $t\geq1$, put
\[
  c_{m,n}(t)
  =\bigl(\pi_{\sigma,-\lambda}(a_t)\phi_m^{-\lambda},
         \phi_n^\lambda\bigr).
\]
Then
\[
  \sup_{t\geq1}|c_{m,n}(t)|
  \lesssim_{m,\lambda}(1+|n|)^{-(1+\re\lambda)/2}
  \qquad \bigl(n\in S(V_{\sigma,\lambda})\bigr).
\]
\end{lemma}

\begin{proof}
Write $a_t^{-1}u_\theta=r_t(\theta)u_{\psi_t(\theta)}$, choosing
the continuous argument with $\psi_t(0)=0$. Then
\[
  r_t(\theta)^2=t^{-2}\cos^2\theta+t^2\sin^2\theta .
\]
By homogeneity and the bilinear pairing \eqref{pairing nondeg},
\begin{equation}\label{cmn-as-fourier-coeff}
  c_{m,n}(t)=\frac{1}{2\pi}\int_0^{2\pi}
      F_t(\theta)e^{in\theta}\,d\theta,
  \qquad
  F_t(\theta)=r_t(\theta)^{-1+\lambda}e^{im\psi_t(\theta)}.
\end{equation}
Put $\alpha=1-\re\lambda$ and $p=2/\alpha$. Then $0<\alpha<2$,
so $p>1$, and $|F_t|^p=r_t^{-2}$. Taking the determinant of
$a_t^{-1}u_\theta$ with its derivative gives
\[
  r_t^2\psi_t'
  =\det(a_t^{-1}u_\theta,a_t^{-1}u_\theta')=1.
\]
As $\theta$ runs from $0$ to $2\pi$, the ellipse
$a_t^{-1}u_\theta$ goes once around the origin in the positive
direction. Hence
\[
  \int_0^{2\pi}|F_t(\theta)|^p\,d\theta
  =\int_0^{2\pi}\psi_t'(\theta)\,d\theta=2\pi .
\]
In particular, H\"older's inequality for the probability measure
$d\theta/(2\pi)$ gives $|c_{m,0}(t)|\leq1$ when $n=0$ occurs.

Now let $N=|n|\geq1$, put $a=N^{-1}$, and consider the first
quadrant. On $[0,\pi/2]$ the function $r_t$ is nondecreasing
because $t\geq1$. Thus $|F_t|=r_t^{-\alpha}$ is nonincreasing.
On the short interval, H\"older's inequality gives
\begin{equation}\label{est:I1-appendix}
  \int_0^a |F_t(\theta)|\,d\theta
  \leq (2\pi)^{1/p}a^{1-1/p}.
\end{equation}
Monotonicity also gives
$a|F_t(a)|^p\leq\int_0^a|F_t(\theta)|^p\,d\theta\leq2\pi$,
so $|F_t(a)|\leq(2\pi/a)^{1/p}$.

On $[a,\pi/2]$, differentiation gives
\[
  |F_t'|
  \leq |1-\lambda|\,r_t^{-\alpha-1}r_t'
       +|m|\,r_t^{-\alpha}\psi_t'.
\]
The radial term is controlled by
\[
  \int_a^{\pi/2}r_t^{-\alpha-1}r_t'\,d\theta
  =\frac{r_t(a)^{-\alpha}-r_t(\pi/2)^{-\alpha}}{\alpha}
  \leq\frac{|F_t(a)|}{\alpha}.
\]
For the angular term, monotonicity gives
\[
  \int_a^{\pi/2}r_t^{-\alpha}\psi_t'\,d\theta
  \leq |F_t(a)|\int_a^{\pi/2}\psi_t'\,d\theta
  \leq 2\pi |F_t(a)|.
\]
One integration by parts therefore yields
\begin{equation}\label{est:I_2-by-parts}
  \begin{aligned}
  \left|\int_a^{\pi/2}F_t(\theta)e^{in\theta}\,d\theta\right|
  &\leq \frac{1}{N}\left(
       |F_t(a)|+|F_t(\pi/2)|
       +\int_a^{\pi/2}|F_t'(\theta)|\,d\theta\right)\\
  &\lesssim_{m,\lambda} N^{-1}|F_t(a)|
   \lesssim_{m,\lambda} N^{-1}a^{-1/p}.
  \end{aligned}
\end{equation}
Since $a=N^{-1}$, both \eqref{est:I1-appendix} and
\eqref{est:I_2-by-parts} are $O_{m,\lambda}(N^{-1+1/p})$.

Reflection $\theta\mapsto-\theta$ leaves $r_t$ unchanged and
replaces $\psi_t$ by $-\psi_t$; translation by $\pi$ leaves $r_t$
unchanged and adds $\pi$ to $\psi_t$. Thus each remaining quadrant
gives a first-quadrant integral up to a constant of modulus one
and possibly the replacement $(m,n)\mapsto(-m,-n)$. The bound is
unchanged. Summing the four quadrants in \eqref{cmn-as-fourier-coeff}
and using
\[
  -1+\frac1p=-\frac{1+\re\lambda}{2}
\]
proves the assertion for $|n|\geq1$, uniformly in $t\geq1$.
Together with the zero-mode bound this proves the lemma.
\end{proof}

\begin{cor}\label{cor upper}
For the irreducible unitary principal series and the complementary
series the estimates of Proposition~\ref{thm upper-lower} hold.
\end{cor}
\begin{proof}
Write $q_0$ for the unshifted compact-picture norm in this proof and
$q$ for the invariant Hilbert norm. Normalize the compact Fourier
vectors $\phi_n^\lambda$ to have constant $q_0$-norm. For the
unitary principal series $q=q_0$ up to normalization. For the
complementary series use the equivalent realization with $-1<\lambda<0$.
In the convention \eqref{pairing nondeg} one has
\begin{equation}\label{eq q(phi_n)}
 q(\phi_n^\lambda)\asymp(1+|n|)^{-\lambda/2}.
\end{equation}

Indeed, put $h_n=q(\phi_n^\lambda)^2$. On compact Fourier modes
the infinitesimal action of $\diag(1,-1)$ is, up to an overall sign,
$\sin(2\theta)\partial_\theta+(1+\lambda)\cos(2\theta)$.
Skew-adjointness gives
$(n+1+\lambda)h_{n+2}=(n+1-\lambda)h_n$.
Taking the product over the even weights proves \eqref{eq q(phi_n)}
by Stirling's formula, for negative as well as positive $n$.

Fix a dual Fourier vector $\phi_m^{-\lambda}$ of the prescribed
parity, taking $m=0$ in the complementary series. By irreducibility
it generates the contragredient.
The compact factors
in the Cartan decomposition $G=K\oline{A^+}K$ act on the two Fourier vectors by scalars
of modulus one, so for this representative
$p_{\min}(\phi_n^\lambda)=\sup_{t\geq1}|c_{m,n}(t)|$.
Put $N=1+|n|$. For imaginary $\lambda$ the compact Fourier vectors
already have constant unitary norm, giving
$p_{\min}(e_n)\lesssim N^{-1/2}$. For $-1<\lambda<0$,
divide the constant-test bound by \eqref{eq q(phi_n)}:
\[
 p_{\min}(e_n)\lesssim_\lambda
       N^{-(1+\lambda)/2}N^{\lambda/2}=N^{-1/2}.
\]
In both cases the bounded frequencies are absorbed in the constant.

For the other inequality use the unweighted band estimate in the
proof of Corollary~\ref{cor:full-induced-gap}, before dyadic
summation. In the parameter used here
$(\re\lambda)(X_0)=\re\lambda/2$ and $c_\gf=1/2$.
Since $\phi_n^\lambda$ has constant compact-picture supremum norm,
\[
 p_{\min}(\phi_n^\lambda)\gtrsim
 (1+|n|)^{-(1+\re\lambda)/2}.
\]
After unitary normalization this is
$p_{\min}(e_n)\gtrsim(1+|n|)^{-1/2}$, matching the upper bound.
On the contragredient the same two bounds hold on the paired weight
line. Since projection to that line is contractive for an isometric
norm, its restricted dual norm on a paired unit vector is the
reciprocal. This gives
$p_{\max}(e_n)\asymp(1+|n|)^{1/2}$, also for imaginary induction.
These endpoint comparisons concern single weight lines, not sums of
weight vectors.
\end{proof}

\subsubsection{\texorpdfstring{Switch to $\SU(1,1)$}{Switch to SU(1,1)}}

For the rest of this section we will work with $G:=\SU(1,1)$ instead of $\SL(2,\R)$. We recall that 
$$
G=\left\{\begin{pmatrix}
	\alpha && \beta \\ \bar{\beta} && \bar{\alpha}
\end{pmatrix}:\alpha,\beta\in\C \text{ with }|\alpha|^2-|\beta|^2=1\right\}
$$
and let 
$$
K=\begin{Bmatrix}
	k_\phi :=\begin{pmatrix}
		e^{-i\phi} && 0\\0 &&e^{i\phi}
	\end{pmatrix}: \phi\in [-\pi, \pi)
\end{Bmatrix}\simeq \mathbb{S}^1
$$
be our choice for the maximal compact subgroup of $G$. For $n\in\Z$, let $\chi_n:K\rightarrow \mathbb{S}^1$ be the character of $K$ defined by
$$
\chi_{n}(k_{\phi})
=e^{in\phi}.
$$
We will use this enumeration to identify $K$-types of a representation with $\Z$.
Finally, define 
$$
A:=\begin{Bmatrix}
	a_t:=\begin{pmatrix}
		\cosh t && \sinh t \\ \sinh t && \cosh t
	\end{pmatrix}: t\in \R
\end{Bmatrix}
$$
and
$$
A^{+}
=\{a_{t}:t>0\}.
$$
We note that $G=K\overline{A^{+}}K$.

In this closed Cartan decomposition the parameter has range $t\ge0$, with $x=\tanh^2t$ below; in particular $x=0$ corresponds to $a_0=1$.

In the following we will often identify $\overline{A^+}$ with $[0,1)$ via the map
\begin{equation} \label{A-identification}
[0,1)\to \oline{A^{+}},\quad x\mapsto {\bf a}_x
:=\begin{pmatrix} \frac{1}{\sqrt{1-x}} & \frac{\sqrt{x}}{\sqrt{1-x}} \\\frac{\sqrt{x}}{\sqrt{1-x}} & \frac{1}{\sqrt{1-x}}\end{pmatrix}.
\end{equation}

\subsubsection{The discrete series and their limits}\label{ds}
Let $\mathbb D=\{z\in\C:|z|<1\}$. We use the holomorphic disk
model of lowest weight $\ell$. For integers $\ell\geq2$ this is
the holomorphic discrete series \cite[Sect.~9]{B}; for $\ell=1$
it is the holomorphic limit of discrete series on the Hardy space
$H^2(\mathbb D)$. Our aim is to compute the matrix coefficient
of the translate of the lowest weight vector against each unit
$K$-type vector.

We first specify the action, to fix signs and weights. An element
\[
 g=\begin{pmatrix}\alpha&\beta\\
                   \bar\beta&\bar\alpha\end{pmatrix}\in\SU(1,1)
\]
acts on the disk by $g\cdot z=(\alpha z+\beta)/
(\bar\beta z+\bar\alpha)$. On holomorphic functions the
representation uses the inverse transformation:
\[
 (\pi_\ell(g)f)(z)
   =(\alpha-\bar\beta z)^{-\ell}
      f\!\left(\frac{\bar\alpha z-\beta}
                     {\alpha-\bar\beta z}\right).
\]
In particular, for $k_\phi=\diag(e^{-i\phi},e^{i\phi})$,
\[
 (\pi_\ell(k_\phi)f)(z)=e^{i\ell\phi}f(e^{2i\phi}z).
\]
Thus $z^j$ has weight $\ell+2j$, in the convention
$\chi_n(k_\phi)=e^{in\phi}$, and the constant function $1$ is
the lowest weight vector.

We next normalize these vectors. For $\ell\geq2$ the squared norm
is obtained by integrating against
$(\ell-1)(1-|z|^2)^{\ell-2}\,dA(z)/\pi$, where $dA$ is plane
area measure. Distinct monomials are orthogonal by angular
integration. Passing to polar coordinates and putting $u=|z|^2$
gives
\[
 \|z^j\|^2=(\ell-1)\int_0^1 u^j(1-u)^{\ell-2}\,du
          =\binom{\ell+j-1}{j}^{-1}.
\]
For $\ell=1$ the monomials already have norm one in $H^2(\mathbb D)$.
Thus in both cases the orthonormal basis is
\[
 e_{\ell+2j}(z)=b_jz^j,\qquad
 b_j=\binom{\ell+j-1}{j}^{1/2},\qquad j\geq0.
\]
In particular $e_\ell=1$.

Now take $g={\bf a}_x$ from \eqref{A-identification}, so that
$\alpha=(1-x)^{-1/2}$ and $\beta=\sqrt{x}(1-x)^{-1/2}$,
where $0\leq x<1$. Since $e_\ell$ is constant, only the
prefactor in the action remains:
\[
 (\pi_\ell({\bf a}_x)e_\ell)(z)
   =(\alpha-\bar\beta z)^{-\ell}
   =(1-x)^{\ell/2}(1-\sqrt{x}\,z)^{-\ell}.
\]
To read off the matrix coefficient we expand in the
\emph{orthonormal} basis, not in the monomial basis:
\[
 \begin{aligned}
 \pi_\ell({\bf a}_x)e_\ell
  &=(1-x)^{\ell/2}
       \sum_{j=0}^{\infty}\binom{\ell+j-1}{j}x^{j/2}z^j\\
  &=(1-x)^{\ell/2}
       \sum_{j=0}^{\infty}b_jx^{j/2}e_{\ell+2j}.
 \end{aligned}
\]
Indeed, the binomial coefficient is $b_j^2$, while
$z^j=e_{\ell+2j}/b_j$. Taking the inner product with the
$j$-th basis vector therefore gives
\begin{equation}\label{eq:extremal-disk-coefficient}
 u_j(x):=\langle\pi_\ell({\bf a}_x)e_\ell,e_{\ell+2j}\rangle
         =b_jx^{j/2}(1-x)^{\ell/2}.
\end{equation}
With these choices of basis vectors $u_j(x)$ is nonnegative.
This will matter when we average in $x$: there is no cancellation.

Changing the phases of
the unit vectors does not affect any estimate below.

To compute the minimal norm, take the functional
$v\mapsto\langle v,e_\ell\rangle$. Its translates separate points
by irreducibility, so Lemma~\ref{lemma smooth minimal} allows us
to take
\[
       p_{\min}(v)=\sup_{g\in G}
                    |\langle\pi_\ell(g)v,e_\ell\rangle|.
\]
In \eqref{eq:extremal-disk-coefficient} it is $e_\ell$, rather
than $e_{\ell+2j}$, that is translated. To pass from one
coefficient to the other, note that unitarity gives
\[
 |\langle\pi_\ell(g)e_{\ell+2j},e_\ell\rangle|
   =|\langle\pi_\ell(g^{-1})e_\ell,e_{\ell+2j}\rangle|.
\]
Since $g^{-1}$ ranges over $G$ as $g$ does, the two coefficients
have the same supremum.

Now write $g=k{\bf a}_x k'$ with $k,k'\in K$ and $0\leq x<1$.
The factor $k'$ acts on $e_\ell$ by the character $\chi_\ell$.
Moving $k$ across the inner product, we see that $k^{-1}$ acts
on $e_{\ell+2j}$ by the character $\chi_{\ell+2j}$. Both
characters have modulus one. Thus, for our chosen representative,
\[
 \begin{aligned}
 p_{\min}(e_{\ell+2j})
   &=\sup_{0\leq x<1}
       |\langle\pi_\ell({\bf a}_x)e_\ell,e_{\ell+2j}\rangle|\\
   &=b_j\sup_{0\leq x<1}x^{j/2}(1-x)^{\ell/2}.
 \end{aligned}
\]
For $j>0$ the logarithmic derivative on $(0,1)$ is
$j/(2x)-\ell/(2(1-x))$, so the maximum occurs at
$x=j/(j+\ell)$. Hence
\[
 p_{\min}(e_{\ell+2j})
   =b_j\left(\frac{j}{j+\ell}\right)^{j/2}
        \left(\frac{\ell}{j+\ell}\right)^{\ell/2}.
\]
For fixed $\ell$ the middle factor tends to $e^{-\ell/2}$,
the last factor is comparable to $j^{-\ell/2}$, and
$b_j\asymp j^{(\ell-1)/2}$. Their product is therefore
comparable to $j^{-1/2}$. For $j=0$ the maximum is attained
at $x=0$ and equals one. We conclude that
\[
               p_{\min}(e_{\ell+2j})\asymp(1+j)^{-1/2}
               \qquad(j\geq0).
\]

Conjugation gives
the antiholomorphic series and limits, and duality gives the
corresponding estimate for $p_{\max}$. This completes the proof
of Proposition~\ref{thm upper-lower}.

\subsection{Averaging the extremal coefficient}
\label{subsection Proof for estimate convex combination}

\begin{proof}[Proof of Proposition~\ref{thm estimate convex combination}]
In the coordinate \eqref{A-identification} take the probability measure
\[
 d\beta_\epsilon(x)=\epsilon(1-x)^{-1+\epsilon}\,dx,\qquad 0\leq x<1 .
\]
For $n=\ell+2j$ and $m=\ell$, \eqref{eq:extremal-disk-coefficient}
gives
\[
 |I_{n,\epsilon}|
   =\epsilon b_j B\!\left(1+\frac{j}{2},
                           \frac{\ell}{2}+\epsilon\right)
   \asymp (1+j)^{-1/2-\epsilon}.
\]
This follows from the gamma quotient for the beta function with
$\ell$ and $\epsilon$ fixed. Since $1+j\asymp1+|n|$, it proves
\eqref{main estimate}. The calculation includes $\ell=1$ without
taking a limit in $\ell$; conjugation gives the negative spectrum
with the same real measure.
\end{proof}

\begin{prop}[Weighted discrete series]\label{r97:weighted-disk}
Let $w$ be a weight on $\SL(2,\R)$, replaced by a bi-$K$-invariant equivalent one. With the disk-model
Cartan element $a_t$ of parameter $t\ge0$, choose $M\ge0$ with $w(a_t),w^\sharp(a_t)\lesssim e^{Mt}$. For
every discrete series representation $V$ and for each of the two limits one has $\gap(V,w)\le 1+M$.
Constants in norm comparisons may depend on $V$.
\end{prop}
\begin{proof}
In the disk coordinate $x=\tanh^2t$, $e^{Mt}\lesssim(1-x)^{-M/2}$. Given $b>M/2$, the average of
$\pi({\bf a}_x)e_\ell$ against $b(1-x)^{b-1}dx$ converges in the maximal weight-$w$ completion by this
bound. This weight dominates a positive constant since $a_t^{-1}$ and $a_t$ are $K$-conjugate and
$w(1)\le w(a_t)^2$. Thus the unitary norm lies in the class and the averaged vector maps into the unitary
globalization. Its $(\ell+2j)$ coordinate, by \eqref{eq:extremal-disk-coefficient}, has absolute value
\[
 b\binom{\ell+j-1}{j}^{1/2} B(1+j/2,\ell/2+b)\asymp_{\ell,b}(1+j)^{-1/2-b}.
\]
Choose a $K$-invariant maximal norm representative. For a finite vector with coordinates $v_j$, choose
the trigonometric polynomial with the required moments $v_j/u_j$, where $u_j$ are the averaged
coordinates. Its $L^1$ norm is at most its $L^2$ norm, bounded by $Cq_{1/2+b}(v)$. Its $K$-average of the
above vector equals $v$, proving $p_{\max,w}\lesssim q_{1/2+b}$. Apply the same estimate to the dual for
$w^\sharp$ and dualize to get $q_{-1/2-b}\lesssim p_{\min,w}$. Shift and combine, then take $b\downarrow
M/2$. Conjugation covers the negative spectrum.
\end{proof}

\begin{rmk}\label{r97:disk-polar}
The ambient polar in Proposition~\ref{Prop Balls}(i) can be strictly larger than the completed ball. In the holomorphic disk model of lowest weight two,
choose the lowest vector as dual test. By \eqref{eq:extremal-disk-coefficient} its translates give
$p_{\min}(f)=\sup_{|z|<1}(1-|z|^2)|f(z)|$ on polynomials. In their completion the weighted modulus tends
uniformly to zero at the boundary. The distribution vector $F(z)=1/[2(1-z)]$ has coordinates
$1/(2\sqrt{j+1})$ in the orthonormal monomials and satisfies $(1-|z|^2)|F(z)|\le1$. Its weighted modulus
along $z=r\uparrow1$ tends to one. Thus it belongs to the ambient polar and lies outside the completion.
\end{rmk}

\end{document}